\documentclass[11pt,reqno,english]{amsart}

\usepackage[latin9]{inputenc}
\usepackage{geometry}
\usepackage{color}
\usepackage{babel}
\usepackage[mathscr]{euscript}
\usepackage{bm}
\usepackage{amsthm}
\usepackage{amssymb}
\usepackage{graphicx}
\usepackage{setspace}
\usepackage[unicode=true,pdfusetitle,
 bookmarks=true,bookmarksnumbered=false,bookmarksopen=false,
 breaklinks=false,pdfborder={0 0 1},backref=false,colorlinks=false]
 {hyperref}

\makeatletter
\numberwithin{equation}{section}
\numberwithin{figure}{section}
\theoremstyle{plain}
\newtheorem{thm}{\protect\theoremname}[section]
\theoremstyle{plain}
\newtheorem{prop}[thm]{\protect\propositionname}
\theoremstyle{remark}

\theoremstyle{definition}
\newtheorem*{example*}{\protect\examplename}
\theoremstyle{plain}
\newtheorem{lem}[thm]{\protect\lemmaname}
\theoremstyle{remark}

\theoremstyle{remark}
\newtheorem*{notation*}{\protect\notationname}
\theoremstyle{definition}

\theoremstyle{definition}

\theoremstyle{plain}

\usepackage{cite}

\usepackage[dvipsnames]{xcolor}

\newcommand{\mc}[1]{{\mathcal #1}}
\newcommand{\mf}[1]{{\mathfrak #1}}
\newcommand{\mb}[1]{{\boldsymbol #1}}
\newcommand{\bb}[1]{{\mathbb #1}}
\newcommand{\bs}[1]{{\boldsymbol #1}}
\newcommand{\ms}[1]{{\mathscr #1}}

\newtheorem{remark}[thm]{Remark}

\newtheorem{corollary}[thm]{Corollary}

\makeatletter
\providecommand{\leftsquigarrow}{%
  \mathrel{\mathpalette\reflect@squig\relax}%
}
\newcommand{\reflect@squig}[2]{%
  \reflectbox{$\m@th#1\rightsquigarrow$}%
}
\makeatother

\usepackage{babel}
\usepackage{stmaryrd}
\usepackage{bbm}

\makeatother

\providecommand{\assumptionname}{Assumption}
\providecommand{\conditionname}{Condition}
\providecommand{\definitionname}{Definition}
\providecommand{\examplename}{Example}
\providecommand{\lemmaname}{Lemma}
\providecommand{\notationname}{Notation}
\providecommand{\propositionname}{Proposition}
\providecommand{\theoremname}{Theorem}

\begin{document}

\title[Metastability of parabolic equations with
drift]{Metastability and time scales for parabolic equations with
drift 1: the first time scale}

\address{IMPA, Estrada Dona Castorina 110, J. Botanico, 22460 Rio de
Janeiro, Brazil and Univ. Rouen Normandie, CNRS,
LMRS UMR 6085,  F-76000 Rouen, France. \\
e-mail: \texttt{landim@impa.br} }

\address{School of Mathematics, Korea Institute for Advanced Study,
Republic of Korea. \\
e-mail: \texttt{jklee@kias.re.kr} }

\address{Department of Mathematical Sciences, Seoul National
University and
Research Institute of Mathematics, Republic of Korea. \\
e-mail: \texttt{insuk.seo@snu.ac.kr} }

\begin{abstract}
Consider the elliptic operator given by
\begin{equation}
\label{fx2}
\mathscr{L}_{\epsilon}f\,=\, \boldsymbol{b} \cdot \nabla f \,+\,
\epsilon\, \Delta f
\end{equation}
for some smooth vector field $\bs b\colon \bb R^d\to\bb R^d$ and a small
parameter $\epsilon>0$. Consider the initial-valued problem
\begin{equation}\label{fx00}
\left\{
\begin{aligned}
& \partial_ t u_\epsilon\,=\, \ms L_\epsilon u_\epsilon\;, \\
& u_\epsilon (0, \cdot) = u_0(\cdot) \;,
\end{aligned}
\right.
\end{equation}
for some bounded continuous function $u_0$. Denote by $\mc M_0$ the
set of critical points of $\bs b$ which are stable stationary points for
the ODE $\dot {\bs x} (t) = \bs b (\bs x(t))$. Under the hypothesis
that $\mc M_0$ is finite and $\bs b = -(\nabla U + \bs \ell)$, where
$\bs \ell$ is a divergence-free field orthogonal to $\nabla U$, the
main result of this article states that there exist a time-scale
$\theta^{(1)}_\epsilon$, $\theta^{(1)}_\epsilon \to \infty$ as
$\epsilon \rightarrow 0$, and a Markov semigroup $\{p_t : t\ge 0\}$
defined on $\mc M_0$ such that
\begin{equation*}
\lim_{\epsilon\to 0} u_\epsilon ( t \, \theta^{(1)}_\epsilon, \bs x )
\;=\; \sum_{\bs m'\in \mc M_0} p_t(\bs m, \bs m')\, u_0(\bs m')\;,
\end{equation*}
for all $t>0$ and $\bs x$ in the domain of attraction of $\bs m$ [for
the ODE $\dot {\bs x}(t) = \bs b(\bs x(t))$].  The time scale
$\theta^{(1)}$ is critical in the sense that, for all time scale
$\varrho_\epsilon$ such that $\varrho_\epsilon \to \infty$,
$\varrho_\epsilon/\theta^{(1)}_\epsilon \to 0$,
\begin{equation*}
\lim_{\epsilon\to 0} u_\epsilon ( \varrho_\epsilon, \bs x )
\;=\;   u_0(\bs m)
\end{equation*}
for all $\bs x \in \mc D(\bs m)$. Namely, $\theta_\epsilon^{(1)}$ is
the first scale at which the solution to the initial-valued problem
starts to change.  In a companion paper \cite{LLS2} we extend this
result finding all critical time-scales at which the solution of the
initial-valued problem \eqref{fx00} evolves smoothly in time and we
show that the solution $u_\epsilon$ is expressed in terms of the
semigroup of some Markov chain taking values in sets formed by unions
of critical points of $\bs b$.
\end{abstract}

\author{Claudio Landim, Jungkyoung Lee, and Insuk Seo}
\maketitle

\section{Introduction}
\label{sec0}

The main concern of the current article is the behavior of the
solution $u_{\epsilon}$ of the equation \eqref{fx00} in the regime
$\epsilon\rightarrow0$.  This problem is connected to the metastable
behavior of the diffusion process induced by the generator
$\mathscr{L}_{\epsilon}$ given in \eqref{fx2}, which has been a
serious issue in the probability community. Freidlin and Koralov
\cite{fk10a,fk10b} found a critical depth $D>0$ and showed that the
the solution $u_{\epsilon}(t,\,x)$ in the interval
$t\in[0,\,e^{(D-\eta)/\epsilon}]$ and
$t\in[e^{(D+\eta)/\epsilon},\,\infty)$ differ significantly for all
$\eta>0$. Therefore, a dramatic phase transition occurs at the scale
$\theta_{\epsilon}=e^{D/\epsilon}$.  This result has been extended by
Koralov and Tcheuko \cite{kt16} to cases which exhibit multiple
metastable time-scales. Ishii and Souganidis \cite{is15,is17} derived
similar results with purely analytical methods.

In this article and the companion paper \cite{LLS2}, we characterize
the solution $u_{\epsilon}$ assuming that the diffusion process
induced by the generator $\mathscr{L}_{\epsilon}$ has a Gibbs
invariant measure. More precisely, fix a smooth potential
$U\colon\mathbb{R}^{d}\to\mathbb{R}$, and a smooth vector field
$\boldsymbol{\ell}\colon\mathbb{R}^{d}\to\mathbb{R}$.  Assume that the
vector field $\boldsymbol{\ell}$ is divergence free and is orthogonal to
the gradient of $U$:
\begin{equation}
(\nabla\cdot\boldsymbol{\ell})(\boldsymbol{x})\,=\,0\;,\quad(\nabla
U)(\boldsymbol{x})\cdot\boldsymbol{\ell}(\boldsymbol{x})\,=\,0\;,\quad\boldsymbol{x}\in\mathbb{R}^{d}\;.
\label{27}
\end{equation}
For $\epsilon>0$, denote by $\mathscr{L}_{\epsilon}$ the elliptic
operator given by
\begin{equation}
\mathscr{L}_{\epsilon}f\,=\,-\,(\,\nabla
U\,+\,\boldsymbol{\ell}\,)\cdot\nabla f\,+\,\epsilon\,\Delta f\;,\quad
f\in C^{2}(\mathbb{R}^{d})\;,
\label{46}
\end{equation}
which corresponds to the operator \eqref{fx2} when
$\boldsymbol{b}= - (\nabla U+\boldsymbol{\ell})$. It has
been shown in \cite{LS-22} that the Gibbs measure
$\mu_\epsilon(d\bs x) = (1/Z_\epsilon) \exp \{U(\bs x)/\epsilon\}\,
d\bs x$ is an invariant measure for the diffusion induced by the
generator $\ms L_\epsilon$ for all $\epsilon>0$ if and only if
$\bs \ell$ satisfies conditions \eqref{27}.

Denote by $\mathscr{L}_{\epsilon}$ the generator
\eqref{46}, unless otherwise specified. Fix a bounded and
continuous function $u_{0}\colon\mathbb{R}^{d}\to\mathbb{R}$ and
consider the initial-valued problem.
\begin{equation}
\left\{ \begin{aligned} & \partial_{t}u_{\epsilon}\,=\,\mathscr{L}_{\epsilon}u_{\epsilon}\;,\\
 & u_{\epsilon}(0,\cdot)=u_{0}(\cdot)\;.
\end{aligned}
\right.
\label{44}
\end{equation}
The tools developed in \cite{BL1,BL2,LLM,LMS2,LMS} permit to describe
the solution of the parabolic equation \eqref{44}, in the domain of
attraction of a local minimum $\boldsymbol{m}$, at the time-scale in
which the solution is transformed from the value of the initial
condition $u_{0}(\cdot)$ at the local attractor $\boldsymbol{m}$ of
the field $\boldsymbol{b}=-(\nabla U+\boldsymbol{\ell})$ to a convex
combination of the initial condition calculated at several different
local attractors $\boldsymbol{m}'$. A similar result appeared in
\cite{bgl2} in the case of sequences of continuous-time Markov chains
on a fixed finite state space.

\subsection*{The first critical time scale}

Let us now explain our main result in more detail. For two positive
sequences $(\alpha_{\epsilon}:\epsilon>0)$,
$(\beta_{\epsilon}:\epsilon>0)$, we denote by
${\color{blue}\alpha_{\epsilon}\prec\beta_{\epsilon}}$,
$\color{blue} \beta_{\epsilon} \succ \alpha_{\epsilon}$ if
$\alpha_{\epsilon}/\beta_{\epsilon}\to0$ as $\epsilon\rightarrow0$.
The main results of the current article and the companion paper
\cite{LLS2} assert that there exist critical times scales
$\theta_{\epsilon}^{(1)}\prec\cdots\prec\theta_{\epsilon}^{(\mathfrak{q})}$
associated with the potential function $U$ at which the asymptotic
behavior of the solution $u_{\epsilon}$ changes dramatically. We not
only characterize these time scales explicitly but also provide
precise asymptotics of $u_{\epsilon}$ along these scales. We also
derive the asymptotics of the solution between these time-scales,
completely analyzing the behavior of $u_{\epsilon}$.

The current article concerns the first time-scale among such a complex
multi-scale structure. We explicitly find a time-scale
$\theta_{\epsilon}^{(1)} \succ 1$, and a Markov semigroup
$\{p_{t}:t\ge0\}$ defined on the set of local minima $\mathcal{M}_{0}$
of $U$ such that, for all local minimum
$\boldsymbol{m}\in\mathcal{M}_{0}$,
\begin{equation}
\lim_{\epsilon\to0}u_{\epsilon}(t\,\theta_{\epsilon}^{(1)},\boldsymbol{x})\;=\;\sum_{\boldsymbol{m}'\in\mathcal{M}_{0}}p_{t}(\boldsymbol{m},\boldsymbol{m}')\,u_{0}(\boldsymbol{m}')\;,\label{43}
\end{equation}
for all $t>0$, $\boldsymbol{x}\in\mathcal{D}(\boldsymbol{m})$. Here,
$\mathcal{D}(\boldsymbol{m})$ represents the domain of attraction of
$\boldsymbol{m}$ for the ODE
$\dot{\boldsymbol{x}}(t)=\boldsymbol{b}(\boldsymbol{x}(t))$, where
$\boldsymbol{b}=-(\nabla U+\boldsymbol{\ell})$.

We also show that for any sequence
$1 \prec \varrho_{\epsilon}\prec\theta_{\epsilon}^{(1)}$
\begin{equation}
\lim_{\epsilon\to0}u_{\epsilon}(\varrho_{\epsilon},
\boldsymbol{x})\;=\;u_{0}(\boldsymbol{m}) \;,
\label{51}
\end{equation}
for all $\boldsymbol{x}\in\mathcal{D}(\boldsymbol{m})$. Hence, the
solution does not change until the time-scale $\theta_{\epsilon}^{(1)}$,
and it starts to change exactly at $\theta_{\epsilon}^{(1)}$ in view
of \eqref{43} and \eqref{51}. The main achievement of the current
article is the verification of \eqref{43} and \eqref{51}. We remark
this scale $\theta_{\epsilon}^{(1)}$ is the scale $\theta_{\epsilon}=e^{D/\epsilon}$
obtained in \cite{fk10a,fk10b}.

To illustrate assertion \eqref{43}, consider a generic
potential $U$, namely a potential for which all critical points are at
different heights. Denote by $\bs m_1$ the local minimum associated to
the shallowest valley and by $\bs m_2$ the unique local minimum
separated from $\bs m_1$ by the shallowest saddle point disconnecting
$\bs m_1$ from the other local minima. The local minimum $\bs m_2$ is
unique because all saddle points are at different heights.  Then
$p_t(\bs m, \bs m') =0$ unless $\bs m=\bs m_1$, $\bs m'=\bs m_2$. In
other words, in the time-scale $\theta^{(1)}_\epsilon$, starting from
the basin of attraction of $\bs m_1$, the process waits an exponential
time and then jumps to a neighbourhood of $\bs m_2$, where it stays
forever.

\subsection*{Multi-scale structure }

The characterization of the remaining scales are the contents of the
companion paper \cite{LLS2}. We briefly explain the main result.

Let us start from the second scale which can be inferred from
\eqref{43}.  The theory of finite-state continuous-time Markov chains
asserts that there exist probability measures $\pi_{j}^{(1)}$,
$1\le j\le\mathfrak{n}_{1}$, on $\mathcal{M}_{0}$ with disjoint
supports, and probability measures
$\omega^{(1)}(\boldsymbol{m},\cdot)$,
$\boldsymbol{m}\in\mathcal{M}_{0}$, on $\{1,\dots,\mathfrak{n}_{1}\}$
such that
\begin{equation}
\lim_{t\to\infty}p_{t}(\boldsymbol{m},\boldsymbol{m}')
\;=\;\sum_{k=1}^{\mathfrak{n}_{1}}
\omega^{(1)}(\boldsymbol{m},k)\,\pi_{k}^{(1)}(\boldsymbol{m}')
\label{eq:52}
\end{equation}
for all $\boldsymbol{m}$, $\boldsymbol{m}'\in\mathcal{M}_{0}$. If
$\boldsymbol{m}'$ is a transient state all terms in the previous sum
vanish. Indeed, the measures $\pi_{j}^{(1)}$ represent the stationary
states of the Markov chain restricted to the closed irreducible
classes, which in turn are the support of the measures
$\pi_{j}^{(1)}$. The weight $\omega^{(1)}(\boldsymbol{m},k)$ stands
for the probability that the Markov chain starting from
$\boldsymbol{m}$ is absorbed at the support of the measure
$\pi_{k}^{(1)}$.

If there is only one stationary state or, equivalently, one closed
irreducible class, namely $\mathfrak{n}_{1}=1$, then for all time-scales
$\varrho_{\epsilon}$ such that $\theta_{\epsilon}^{(1)}\prec\varrho_{\epsilon}$,
we can readily guess from \eqref{43} and \eqref{eq:52} that (note
that $\omega^{(1)}(\boldsymbol{m},1)=1$ in this case)
\[
\lim_{\epsilon\to0}u_{\epsilon}(\varrho_{\epsilon},\boldsymbol{x})\;=\;\sum_{\boldsymbol{m}'\in\mathcal{M}_{0}}\pi_{1}^{(1)}(\boldsymbol{m}')\,u_{0}(\boldsymbol{m}')
\]
for all local minimum $\boldsymbol{m}\in\mathcal{M}_{0}$ and $\boldsymbol{x}\in\mathcal{D}(\boldsymbol{m})$.
Note that the limit does not depend on $\boldsymbol{m}$ or $\boldsymbol{x}$.
This behavior occurs when (a) all the wells associated to local minima
of $U$ which are not global minima have the same depth, and (b) either
there is only one global minimum or there is more than one and the
depth of all wells are the same. In this case of a unique closed irreducible
class, the support of the measure $\pi_{1}^{(1)}$ corresponds to
the set of global minima of $U$. This finishes the description of
multi-scale structure for the case $\mathfrak{n}_{1}=1$.

In contrast, if there are more than one closed irreducible classes,
the limit of $u_{\epsilon}(t\theta_{\epsilon}^{(1)},\boldsymbol{x})$
as $\epsilon\to0$ and then $t\to\infty$ depends on the local minimum
attracting $\boldsymbol{x}$. In this case, there exists a second and
longer time-scale $\theta_{\epsilon}^{(2)}$ such that
$\theta_{\epsilon}^{(1)}\prec\theta_{\epsilon}^{(2)}$ and a Markov
semigroup $\{p_{t}^{(2)}:t\ge0\}$ defined on the set of closed
irreducible classes $\{1,\dots,\mathfrak{n}_{1}\}$ obtained at the
first time-scale such that,
\[
\lim_{\epsilon\to0}u_{\epsilon}(t\,\theta_{\epsilon}^{(2)},\boldsymbol{x})\;=\;\sum_{k=1}^{\mathfrak{n}_{1}}\omega^{(1)}(\boldsymbol{m},k)\,\sum_{\ell=1}^{\mathfrak{n}_{1}}p_{t}^{(2)}(k,\ell)\,\sum_{\boldsymbol{m}'\in\mathcal{M}_{0}}\pi_{\ell}^{(1)}(\boldsymbol{m}')\,u_{0}(\boldsymbol{m}')
\]
for all $t>0$ where $\boldsymbol{m}\in\mathcal{M}_{0}$ is a local
minium of $U$ and $\boldsymbol{x}\in\mathcal{D}(\boldsymbol{m})$.
Mind that we may restrict the sum over $\boldsymbol{m}'$ to local
minima in the support of the measure $\pi_{\ell}^{(1)}$. We can also
verify that, for any sequence $\varrho_{\epsilon}$ such that $\theta_{\epsilon}^{(1)}\prec\varrho_{\epsilon}\prec\theta_{\epsilon}^{(2)}$,
we have
\[
\lim_{\epsilon\to0}u_{\epsilon}(\varrho_{\epsilon},\boldsymbol{x})\;=\;\sum_{k=1}^{\mathfrak{n}_{1}}\omega^{(1)}(\boldsymbol{m},k)\,\sum_{\boldsymbol{m}'\in\mathcal{M}_{0}}\pi_{k}^{(1)}(\boldsymbol{m}')\,u_{0}(\boldsymbol{m}')
\]
for all $\boldsymbol{m}\in\mathcal{M}_{0}$ and $\boldsymbol{x}\in\mathcal{D}(\boldsymbol{m})$.
This is exactly the behavior of the solution $u_{\epsilon}$ in the
time scale $t\theta_{\epsilon}^{(1)}$ as $\epsilon\to0$ and then
$t\to\infty$, and the one in the time scale $t\theta_{\epsilon}^{(2)}$
as $\epsilon\to0$ and then $t\to0$. This completes the description
of the asymptotics of $u_{\epsilon}$ until the second scale $\theta_{\epsilon}^{(2)}$.

More generally, there exist $\mathfrak{q}\ge1$ and time-scales $\theta_{\epsilon}^{(1)}\prec\cdots\prec\theta_{\epsilon}^{(\mathfrak{q})}$
such that
\begin{equation}
\lim_{\epsilon\to0}u_{\epsilon}(t\theta_{\epsilon}^{(p)},\boldsymbol{x})\;=\;\sum_{k=1}^{\mathfrak{n}_{p-1}}\omega^{(p-1)}(\boldsymbol{m},k)\,\sum_{\ell=1}^{\mathfrak{n}_{p-1}}p_{t}^{(p)}(k,\ell)\,\sum_{\boldsymbol{m}'\in\mathcal{M}_{0}}\pi_{\ell}^{(p-1)}(\boldsymbol{m}')\,u_{0}(\boldsymbol{m}')\label{lls2-1}
\end{equation}
for each $1\le p\le\mathfrak{q}$, $t>0$,
$\boldsymbol{m}\in\mathcal{M}_{0}$,
$\boldsymbol{x}\in\mathcal{D}(\boldsymbol{m})$. Furthermore, for each
$1\le p\le\mathfrak{q}+1$, sequence $(\varrho_{\epsilon}:\epsilon>0)$
such that
$\theta_{\epsilon}^{(p-1)}\,\prec\,\varrho_{\epsilon}\,\prec\,\theta_{\epsilon}^{(p)}$,
$\boldsymbol{m}\in\mathcal{M}_{0}$, and
$\boldsymbol{x}\in\mathcal{D}(\boldsymbol{m})$,
\begin{equation}
\lim_{\epsilon\to0}u_{\epsilon}(\varrho_{\epsilon},\boldsymbol{x})\;=\;\sum_{k=1}^{\mathfrak{n}_{p-1}}\omega^{(p-1)}(\boldsymbol{m},k)\,\sum_{\boldsymbol{m}'\in\mathcal{M}_{0}}\pi_{k}^{(p-1)}(\boldsymbol{m}')\,u_{0}(\boldsymbol{m}')\;.\label{lls2-2}
\end{equation}
In this formula, $\theta_{\epsilon}^{(0)}$,
$\theta_{\epsilon}^{(\mathfrak{q}+1)}$ are the constant sequences
equal to $1$, $+\infty$, respectively.  Summing up,

\begin{itemize}
\item Denote by $\color{blue} \mathfrak{n}_{0}$ the number of local
minima of $U$ so that
$\mathfrak{n}_{0}>\mathfrak{n}_{1}>\cdots>\mathfrak{n}_{\mathfrak{q}}=1$.

\item $p_{t}^{(p)}$, $t\ge0$, is a Markov semigroup on
$\{1,\dots,\mathfrak{n}_{p-1}\}$, $1\le p \le \mf q$. Here, the
semigroup $p_t$, introduced in \eqref{43}, has been represented by
$p^{(1)}_t$.

\item For a fixed $1\le p\le\mathfrak{q}$, $\pi_{j}^{(p)}$,
$1\le j\le\mathfrak{n}_{p}$, are probability measures on
$\mathcal{M}_{0}$ with disjoint supports. They correspond to the
extremal invariant probability measures of the Markov chain with
transition probability $p_{t}^{(p)}$.

\item $\omega^{(p)}(\boldsymbol{m},\cdot)$ are probability measures on
$\{1,\dots,\mathfrak{n}_{p}\}$, where $\omega^{(p)}(\boldsymbol{m},j)$
stands for the probability that $\boldsymbol{m}$ is absorbed at the
support of the probability measure $\pi_{j}^{(p)}$.
\end{itemize}

It turns out that all local minima which belong to the support of a
measure $\pi_{j}^{(p)}$ are at the same height:
$U(\boldsymbol{m}')=U(\boldsymbol{m}'')$ if $\boldsymbol{m}'$,
$\boldsymbol{m}''$ belong to the support of the same measure
$\pi_{j}^{(p)}$. On the other hand, the support of a measure
$\pi_{j}^{(p+1)}$ is formed by the union of the supports of measures
$\pi_{k}^{(p)}$, $k\in\{1,\dots,\mathfrak{n}_{p}\}$.  Moreover,
$\pi_{j}^{(p+1)}$ is a convex combination of the corresponding
measures $\pi_{k}^{(p)}$. The rigorous recursive construction of this
multi-scale structure is a delicate and complicated task and will be
done in the companion paper \cite{LLS2}. Assertions \eqref{lls2-1} and
\eqref{lls2-2} will be proven there as well.

\subsection*{Comments on the proof}

The analysis of the asymptotics of the solution $u_{\epsilon}(t,\,\boldsymbol{x})$
of \eqref{44} is closely related to that of the metastable behavior
of the process
\begin{equation}
d\boldsymbol{x}_{\epsilon}(t)\,=\,
-\, (\,\nabla U\,+\,\boldsymbol{\ell}\,)
(\boldsymbol{x}_{\epsilon}(t))\,dt\,+\,\sqrt{2\epsilon}\,dW_{t}
\label{sde0}
\end{equation}
where $\epsilon>0$ denotes a small parameter corresponding to the
temperature of the system, and $W_{t}$ a $d$-dimension Brownian
motion. This relation comes from well-known expression
\[
u_{\epsilon}(t,\,\boldsymbol{x})=\mathbb{E}_{\boldsymbol{x}}^{\epsilon}\left[u_{0}(\boldsymbol{x}_{\epsilon}(t))\right]\;,\;\;\;\;t\ge0,\,\boldsymbol{x}\in\mathbb{R}^{d},
\]
where $\mathbb{E}_{\boldsymbol{x}}^{\epsilon}$ denotes the expectation
with respect to the diffusion process \eqref{sde0} starting at
$\boldsymbol{x}\in\mathbb{R}^{d}$.

The proof of the result described above is purely probabilistic and
relies on the theory of metastable Markov processes developed in
\cite{BL1,LLM,LMS2,RS,LMS}. The metastable behavior of the process
\eqref{sde0} has been recently studied in several articles: \cite{LM}
provided sharp asymptotics on the low-lying spectra which is closely
related with the metastability of the process
$\boldsymbol{x}_{\epsilon}(\cdot)$, \cite{LS-22} established
Eyring-Kramers law precisely estimating the mean transition time from
a local minimum of $U$ to another one, and finally \cite{LS-22b}
investigated the metastability among the global minima (i.e., ground
states) of $U$. The last work can be regarded as the analysis of the
metastability at the final scale $\theta_{\epsilon}^{(\mathfrak{q})}$
described above.

The recursive construction of the multiscale structure
presented here appeared before in different contexts.  L. Michel
\cite{M19} introduced it to study the low-lying eigenvalues of the
semiclassical Witten Laplacian associated to a Morse function.  We
refer to \cite{bl3, lx16, fk17} for the same construction in the
context of finite state Markov chains.  

The analysis of the multi-scale structure is based on the resolvent
approach to metastability developed in \cite{LMS}.  The crucial
point consists in showing that the solution of a resolvent equation is
asymptotically constant in neighborhoods of local minima.  More
precisely, denote by $\mathcal{E}(\boldsymbol{m})$ a small
neighborhood of a local minimum $\boldsymbol{m}$. Fix $\lambda>0$,
$\boldsymbol{g}\colon\mathcal{M}_{0}\rightarrow\mathbb{R}$, and let
$\phi_{\epsilon}$ be the unique solution of the resolvent equation
\[
(\lambda-\theta_{\epsilon}^{(1)}\mathscr{L}_{\epsilon})\,\phi_{\epsilon}\;=\;G\;:=\;\sum_{\boldsymbol{m}\in\mathcal{M}_{0}}\mb g(\boldsymbol{m})\,\chi_{_{\mathcal{E}(\boldsymbol{m})}}\;,
\]
where $\chi_{\mathcal{A}}$, $\mathcal{A}\subset\mathbb{R}^{d}$,
represents the indicator function of the set $\mathcal{A}$. The
function on the right-hand side vanishes at
$(\cup_{\boldsymbol{m}\in\mathcal{M}_{0}}\mathcal{E}(\boldsymbol{m}))^{c}$
and is constant on each well $\mathcal{E}(\boldsymbol{m}')$. One of
the main results of this article asserts that the solution $\phi$ is
asymptotically constant in each well $\mathcal{E}(\boldsymbol{m})$:
\begin{equation}
\lim_{\epsilon\rightarrow0}\,\max_{\boldsymbol{m}\in\mathcal{M}_{0}}\,\sup_{\boldsymbol{x}\in\mathcal{E}(\mb
m)}\vert\,\phi_{\epsilon}(\boldsymbol{x})-\boldsymbol{f}(\mb
m)\,\vert\;=\;0\;,
\label{fx3}
\end{equation}
where $\mb f$ is the solution of the reduced resolvent equation
\begin{equation}
(\lambda-\mathfrak{L}_{1})\,\boldsymbol{f}\;=\;\mb g\;,\label{fx1}
\end{equation}
and $\mathfrak{L}_{1}$ is the generator of the $\mathcal{M}_{0}$-valued,
continuous-time Markov chain whose associated semigroup is the one
appearing in \eqref{43}. Property \eqref{43} of the solution of
the initial-valued problem \eqref{44} is deduced from this property
of the resolvent equation. \smallskip{}

\subsection*{Background}

In a sequence of seminal works, Freidlin and Wentzel, \cite{fw98} and
references therein, investigated the metastable behavior of random
perturbations of dynamical systems, and introduced the notion of
hierarchy of cycles. Assuming that each cycle has only one subsequent
cycle, they described the metastable behavior of the diffusion
$d\bs x_\epsilon (t) = \bs b(\bs x_\epsilon (t))\, dt +
\sqrt{2\epsilon} d W_t$ at all the time scales other than the
critical ones at which the diffusion may jump from one cycle to
another. We refer to \cite{fk10a, fk10b, is15, kt16, is17} for recent
developments of this theory.

With the notation introduced above, the hypothesis that each cycle has
only one subsequent cycle means that for each $1\le p\le \mf q$,
$1\le j\le \mf n_{p-1}$, there exists only one $k$ such that
$p^{(p)}_t(j,k) >0$.

In this article, assuming that the drift $\bs b(\cdot)$ can be written
as $\bs b = -(\nabla U +\bs \ell)$ for a vector field $\bs \ell$
satisfying conditions \eqref{27}, we extend the results in \cite{fw98}
in two directions: describing the behavior of the diffusion (or the
one of the solution of the parabolic equation \eqref{44}) at the
critical time-scales $\theta^{(p)}_\epsilon$, and  removing the
hypothesis that each cycle has only one subsequent cycle.

As mentioned at the beginning of this introduction, we impose the
conditions \eqref{27} on the vector field $\bs \ell (\cdot)$ for the
Gibbs measure
$\mu_\epsilon (d\bs x) = (1/Z_\epsilon)\, e^{-U(\bs x)/\epsilon}\,
d\bs x$ to be the stationary state of the diffusion process
$\bs x_\epsilon (\cdot)$. The precise description of the asymptotic
behavior of the solution $u_\epsilon$ in the time-scale
$\theta^{(1)}_\epsilon$ presented in \eqref{43} relies on explicit
computations which require an explicit formula for the stationary
state and smoothness of its density.

In general (that is, without the hypotheses \eqref{27}), the
quasi-potential, which plays the role of $U(\cdot)$ in the formula for
the stationary state, is not smooth and not known explicitly, making
it impossible to apply the approach proposed here.  We refer to
\cite{fg1, ls19}, for a model (a one-dimensional diffusion on the
torus) where the quasi-potential can be computed and the methods
presented here applied, despite the lack of smoothness of the
quasi-potential.

Estimates of the transition times expectation have been
obtained in \cite{BEGK, LMS2}. In the presence of many wells, these
expectations may not converge for the following reason. With a very
small probability, crossing a saddle point higher than the lowest one,
the diffusion may hit a very deep well and remain there a very long
time. This contribution might be dominant for the expectation, turning
it much larger than predicted.  

Uniform estimates, similar to \eqref{fx3}, for solutions of Dirichlet
problems go back at least to Devinatz and Friedman \cite{DF78}, and
Day \cite{Day82}. The convergence to a constant is called in the
literature the leveling property of the equation. We refer to Leli\`{e}vre,
Le Peutrec and Nectoux \cite{LLP22} for a recent account and further
references.

\subsection*{Organization} The paper is organized as follows. In Section
\ref{sec1}, we state the main results. The proof of Theorem \ref{t01}
is divided in two parts. In Section \ref{sec2}, we prove that the
solution of the resolvent equation is constant on each well, and,
in Section \ref{sec9}, that the solution of the resolvent equation
restricted to the set of local minima of $U$ is asymptotically the
solution of the reduced resolvent equation \eqref{fx1}.

  The proof of the local constancy relies on a diffusion mixing
time estimate presented in Section \ref{sec-ap3}. The proof of the
second property of the resolvent equation solution requires an estimate
of the time it takes to exit a neighborhood of an unstable equilibrium
point, presented in Section \ref{sec5}, estimates on the time needed
to reach a local minimum of $U$, the subject of Section \ref{sec6},
and test functions which approximate the equilibrium potential between
wells, introduced in Section \ref{sec10}. In Section \ref{sec4},
we add the last piece of the proof, extending the results of Section
\ref{sec2} by showing that the solution of the resolvent equation
is actually asymptotically constant in the domain of attractions of
a local minimum. In Section \ref{sec9} we prove Theorem \ref{t01},
and Theorem \ref{t00} in Section \ref{sec3}. In the appendices,
we present some results needed in the proofs.

\section{Model and Main Results}
\label{sec1}


Fix a function $U\colon \bb R^d \to \bb R$ in $C^{3}(\mathbb{R}^{d})$
admitting only a finite number of critical points, all non-degenerate
(hence $U$ is a Morse function, cf. \cite[Definition 1.7]{Nic18}).
Assume that
\begin{equation}
\label{26}
\begin{gathered}
\lim_{n\to\infty}\inf_{|\bm{x}|\geq n}\frac{U(\bm{x})}{|\bm{x}|}
\,=\,\infty\;,\quad
\lim_{|\bm{x}|\to\infty}\frac{\bm{x}}{|\bm{x}|}
\cdot\nabla U(\bm{x})\,=\,\infty\;, \\
\lim_{|\bm{x}|\to\infty} \big\{ \, |\nabla U(\bm{x})| \,-\, 2\, \Delta
U(\bm{x}) \, \big\} \,=\,\infty\;.
\end{gathered}
\end{equation}
In this formula and below, $\color{blue} |\boldsymbol{x}|$ represents
the Euclidean norm of $\bs x\in \mathbb{R}^{d}$. Suppose, without loss
of generality, that $\min_{\bs x\in \bb R^d} U(\bs x) = 0$. Consider a vector
field $\bs \ell\colon \bb R^d \to \bb R$ in $C^2(\bb R^d)$, assumed to
be divergence free and orthogonal to the graduent of $U$ as stated in
\eqref{27}.

\subsection*{Time-scale}

Denote by $\color{blue} \mathcal{M}_{0}$ the set of local minima of
$U$. For each pair
$\boldsymbol{m}' \neq \boldsymbol{m}''\in\mathcal{M}_{0}$, denote by
$\Theta(\boldsymbol{m}',\,\boldsymbol{m}'')$ the communication height
between $\boldsymbol{m}'$ and $\boldsymbol{m}''$:
\begin{equation}
\label{Theta}
{\color{blue} \Theta(\boldsymbol{m}',\,\boldsymbol{m}'') } \;:=\;
\inf_{\substack{\boldsymbol{z}:[0\,1]\rightarrow\mathbb{R}^{d}}}
\max_{t\in[0,\,1]}U(\boldsymbol{z}(t))\;,
\end{equation}
where the minimum is carried over all continuous paths
$\boldsymbol{z}(\cdot)$ such that $\boldsymbol{z}(0)=\boldsymbol{m}'$
and $\boldsymbol{z}(1)=\boldsymbol{m}''$.  Clearly,
$\Theta(\boldsymbol{m}',\,\boldsymbol{m}'') =
\Theta(\boldsymbol{m}'',\,\boldsymbol{m}')$.  Denote by $\Gamma(\bs m)$
the depth of the local minimum $\bs m\in\mc M_0$:
\begin{equation}
\label{53}
{\color{blue}  \Gamma(\bs m)} \, :=\, \min_{\bs m' \neq \bs m}
\Theta(\bs m, \bs m') \,-\,
U(\boldsymbol{m})\;.
\end{equation}
Denote by $d^{(1)}$ the depth of the shallowest well, and by
$\theta^{(1)}_\epsilon$ the corresponding time-scale:
\begin{equation*}
{\color{blue}   d^{(1)} } \;:=\; \min_{\bs m\in \mc M_0} \Gamma(\bs m)\;, \quad
{\color{blue}   \theta^{(1)}_\epsilon } \; :=\; e^{d^{(1)}/\epsilon}\;.
\end{equation*}

\subsection*{Gates}

Denote by $\color{blue} \Upsilon (\bs m)$ the set of gates of
$\bs m\in\mc M_0$. This is the set of points $\bs x\in \bb R^d$ for
which there exist $\bs m'\in \mc M_0$, $\bs m'\neq \bs m$, and a
continuous path $z\colon [0,1]\to \bb R^d$ such that $z(0) =\bs m$,
$z(1)=\bs m'$, $z(1/2) = \bs x$ and
$U(z(t)) < U(\bs x) = U(\bs m) + \Gamma (\bs m)$ for all $t\in [0,1]$,
$t\neq 1/2$.

Mind that there might be more than one local minima $\bs m'$ for the
same gate $\bs x\in \Upsilon (\bs m)$: there might exist
$\bs m_1\neq \bs m_2$, both different from $\bs m$,
$\bs x\in\Upsilon (\bs m)$, and continuous paths
$z_i\colon [0,1]\to \bb R^d$, $i=1$, $2$, such that $z_i(0) =\bs m$,
$z_i(1)=\bs m_i$, $z_i(1/2) = \bs x$ and
$U(z_i(t)) < U(\bs x) = U(\bs m) + \Gamma (\bs m)$ for all
$t\in [0,1]$, $t\neq 1/2$.

Mind also that in the definition of gate, we require $\bs m'$ to be
different from $\bs m$. In this way, we exclude from the set of gates
points $\bs x$ for which there exists a continuous path
$z\colon [0,1]\to \bb R^d$ such that $z(0) =\bs m$, $z(1)=\bs m$,
$z(1/2) = \bs x$ and $U(z(t)) < U(\bs x) = U(\bs m) + \Gamma (\bs m)$
for all $t\in [0,1]$, $t\neq 1/2$. 

Recall that $\color{blue} \bs b \,=\, -\, (\nabla U \,+\, \bs \ell)$
and that a heteroclinic orbit $\phi$ from $\bs x$ to
$\bs y\in \bb R^d$ is a solution $\phi\colon \bb R \to \bb R^d$ of the
ODE
\begin{equation}
\label{31}
\dot {\bs x}  (t) \,=\, \bs b(\bs x (t)) \;,
\end{equation}
such that
\begin{equation*}
\lim_{t\to - \infty} \phi(t) \,=\, \bs x\,,\quad
\lim_{t\to  \infty} \phi(t) \,=\, \bs y \;.
\end{equation*}
We represent this relation by
$\color{blue} \bs x \curvearrowright \bs y$. In other words,
$\bs x \curvearrowright \bs y$ indicates the existence of a
heteroclinic orbit from $\bs x$ to $\bs y$.  We assume also that for
all $\bs m\in \mc M_0$ such that $\Gamma(\bs m) = d^{(1)}$, and
$\bs \sigma \in \Upsilon (\bs m)$, there exists $\bs m'\in \mc M_0$,
$\bs m' \neq \bs m$, such that
\begin{equation}
\label{48}
\bs \sigma \curvearrowright \bs m \;\;\text{and}\;\; \bs \sigma
\curvearrowright \bs m'\;.
\end{equation}

The condition \eqref{48} is crucial. For instance, if there
is a heteroclinic orbit from a saddle point
$\bs \sigma \in \Upsilon (\bs m)$ to a saddle point $\bs \sigma'$, we
are not able to determine which is the local minimum visited after
$\bs m$ when the process starts from a neighbourhood of $\bs m$.  

The assumption \eqref{48} holds when the dynamical system
$\boldsymbol{x}(\cdot)$ defined in \eqref{31} is a Morse-Smale
system. In a Morse-Smale system, for two critical points
$\boldsymbol{c}_{1},\,\boldsymbol{c}_{2}$ of $U(\cdot)$, the unstable
manifold of $\boldsymbol{c}_{1}$ and the stable manifold of
$\boldsymbol{c}_{2}$ intersect transversally and thus when
$\boldsymbol{c}_{1}\curvearrowright\boldsymbol{c}_{2}$, the index (the
number of negative eigenvalue of the Hessian) must strictly decrease
along heteroclinic orbits; hence \eqref{48} follows naturally.

By Proposition \ref{l05a}, any gate
$\bs x \in {\color{blue} \Upsilon := \cup_{\bs m\in \mc M_0} \Upsilon
(\bs m)}$ belongs to the set of critical points of $U$, denoted by
$\color{blue} \mc C_0$:
\begin{equation*}
{\color{blue} \mc C_0} \;:=\; \{\, \bs x \in \bb R^d : (\nabla U)(\bs x) = 0\,\}\;.
\end{equation*}
By \cite[Theorem 2.1]{LS-22}, the divergence-free field $\bs \ell$
vanishes at the critical points of $U$: $\bs \ell(\bs x) =0$ for all
$\bs x\in \mc C_0$.  Denote by $\color{blue} (\nabla^{2}U)( \bs x)$ the
Hessian of $U$ at $\bs x$. Since $U$ is a Morse function, for all
$\bs \sigma \in \Upsilon$,
\begin{equation}
\label{47}
\text{ $(\nabla^{2}U)( \bs \sigma)$ has only
one negative eigenvalue, all the others being strictly positive} \;.
\end{equation}
Indeed, by definition, $\bs\sigma$ can not be a local minimum. On the
other hand, assume that $\bs \sigma$ is a gate between $\bs m$ and
$\bs m'$.  If the number of negative eigenvalues is greater than $1$,
the set $\{\bs x : U(\bs x) < U(\bs \sigma)\}$ would be locally
connected, and there would be a continuous path from $\bs m$ to
$\bs m'$ staying strictly below $ U(\bs \sigma)$, which is a
contradiction.

Denote by $\ms V(\bs m)$ the set of points $\bs m' \in\mc M_0$,
$\bs m'\neq \bs m$, for which \eqref{48} holds for some
$\bs \sigma \in \Upsilon (\bs m)$.  Hence, $\ms V(\bs m)$ is the set
of local minima $\bs m' \in \mc M_0$, $\bs m'\neq \bs m$, for which
there exist a critical point $\bs \sigma \in \Upsilon (\bs m)$ and
heteroclinic orbits from $\bs \sigma$ to $\bs m$ and $\bs \sigma$ to
$\bs m'$:
\begin{gather*}
{\color{blue} \ms V(\bs m) } \; :=\; \big \{ \, \bs m' \in \mc M_0
\setminus \{\bs m\}:
\exists \, \bs \sigma \in \Upsilon (\bs m) \;\; \text{such that}\;\;
\bs \sigma\curvearrowright  \bs m \;,\;\;
\bs \sigma\curvearrowright \bs m' \, \big\}\;.
\end{gather*}
Elements of $\ms V(\bs m)$ are called neighbors of the local minimum
$\bs m$ of $U$.  Denote by $\mc S (\bs m , \bs m')$,
$\bs m' \neq \bs m$, the set of critical points which separate $\bs m$
from $\bs m'$:
\begin{equation}
\label{38a}
{\color{blue} \mc S (\bs m , \bs m')} \; :=\; \big \{ \, \bs \sigma
\in \Upsilon (\bs m) :
\bs \sigma\curvearrowright  \bs m \;,\;\;
\bs \sigma\curvearrowright \bs m' \, \big\}\;.
\end{equation}

\subsection*{Reduced model}

Denote by $\color{blue} (D \bs \ell)(\bs x)$ the Jacobian of
$\bs \ell$ at $\bs x$.  By \eqref{47}, $(\nabla^{2}U)( \bs \sigma)$,
$\bs \sigma\in\Upsilon$, has only one negative eigenvalue.  By
\cite[Lemma 3.3]{LS-22},
$(\nabla^{2}U)( \bs \sigma) + (D \bs \ell)(\bs \sigma)$ has also one
negative eigenvalue, represented by
$\color{blue} -\mu_{\bs \sigma}<0$.  For $\bs m \in \mc M_0$,
$\bs \sigma \in \Upsilon(\bs m)$, let the weights $\nu (\bs m)$,
$\omega (\bs \sigma)$ be given by
\begin{equation}
\label{eq:nu}
{\color{blue}  \nu(\boldsymbol{m}) } \;:=\;
\frac{1}{\sqrt{\det(\nabla^{2}U)(\boldsymbol{m})}} \;, \quad
{\color{blue}  \omega(\boldsymbol{\sigma})} \;:=\;
\frac{\mu_{\bs \sigma}}
{2\pi\sqrt{-\det\nabla^{2}U(\boldsymbol{\sigma})}} \;
\cdot
\end{equation}

Let $\omega (\bs m, \bs m')$, $\bs m \neq \bs m' \in \mc M_0$, be the
weight given by
\begin{equation}
\label{39a}
{\color{blue} \omega (\bs m, \bs m') } \; :=\;
\sum_{\bs \sigma \in \mc S (\bs m , \bs m')}  \omega (\bs \sigma)\;.
\end{equation}
Note that $\omega (\bs m, \bs m')$ vanishes if
$\bs m' \notin \mc V(\bs m)$. Moreover, neither
$\mc S(\,\cdot\,,\, \cdot\,)$ nor $\omega (\,\cdot\,,\, \cdot\,)$ is
symmetric in its arguments. To include the depth of the local minimum
$\bs m$ in the definition of the weight $\omega (\bs m, \bs m') $, set
\begin{equation}
\label{40}
{\color{blue} \omega_1 (\bs m, \bs m') } \; :=\;
\omega (\bs m, \bs m') \, \bs 1 \{ \Gamma (\bs m) = d^{(1)} \, \}\;.
\end{equation}

Denote by $\mf L_1$ the generator of the $\mc M_0$-valued,
continuous-time Markov chain given by
\begin{equation}
\label{28}
(\mf L_1 \bs h)(\bs m) \;=\;
\frac{1}{\nu(\bs m)}\,
\sum_{\bs m'\in \mc M_0} \omega_1 (\bs m, \bs m') \,
[\, \bs h (\bs m') \,-\, \bs h (\bs m) \,] \;.
\end{equation}
As $\omega (\bs m, \bs m')$ vanishes if $\bs m'$ does not belong to
$\mc V(\bs m)$, the sum can be carried over $\mc V(\bs m)$.

\begin{thm}
\label{t00}
Assume that hypotheses \eqref{26}, \eqref{48} are in force.  Fix a
bounded and continuous function $u_0\colon \bb R^d \to \bb R$.  Denote
by $u_\epsilon$ the solution of the parabolic equation
\eqref{44}. Then, \eqref{43} and \eqref{51} hold for all $t>0$, where
$p_t(\cdot,\cdot)$ is the semigroup associated to the generator
$\mf L_1$.
\end{thm}

\subsection*{Resolvent equation}


The proof of Theorem \ref{t00} is based on properties of the resolvent
equation presented in this subsection. Denote by
$\color{blue} B(\bs x, r)$, $\bs x\in \bb R^d$, $r>0$, the open ball
of radius $r$ centered at $\bs x$.  Let
$\color{blue} \mathcal{W}^{r}(\bm{m})$, $\bm{m}\in\mathcal{M}_{0}$,
$r>0$, be the connected component of the set
$\{\boldsymbol{x}\in\mathbb{R}^{d}:U(\boldsymbol{x})\le U(\bm{m})+ r
\}$ containing $\bm{m}$.

Fix $\bs m\in\mc M_0$.  All constants $r_i$ below depend on $\bs m$
and $\bs b (\cdot)$, though this does not appear in the notation.
Equation \eqref{eq:cond1} introduces a positive constant $r_5>0$.
Choose $r_4$ small enough for \eqref{eq:condr_4} to hold with
$r_3=r_5$. By Proposition \ref{pap4} and conditions (1), (2) in
Section \ref{sec-ap3}, $B(\bs m, r_5)$ does not contain critical
points of $U$ besides $\bs m$.

Choose $r_{0}>0$ small enough so that for all $\bs m\in \mc M_0$,
\begin{itemize}
\item[(a)] $\overline{\mathcal{W}^{2r_{0}}(\bm{m})}
\setminus\{\boldsymbol{m}\}$
does not contain critical points of $U$;

\item[(b)] $\mc W^{2r_0}(\bs m)$ is contained in the domain
of attraction of $\bs m$ for the ODE \eqref{31};

\item[(c)] $\bs b(\bs x) \cdot \bs n(\bs x) <0$ for all
$\bs x\in \partial \mc W^{2r_0}(\bs m)$, where $\bs n(\cdot)$ is the
exterior normal of the boundary of $\mc W^{2r_0}(\bs m)$.

\item[(d)] $\mc W^{3r_0}(\bs m) \subset B(\bs m, r_5)$.

\item[(e)] $\mc W^{2r_0}(\bs m) \subset \mc D_{r_4} (\bs m)$
\end{itemize}
Since
$\mc W^{r}(\bs m) =
\{\boldsymbol{x}\in\mathbb{R}^{d}:U(\boldsymbol{x})\le U(\bm{m})+ r
\}$, for $r$ small enough $\bs n = \nabla U$ at the boundary of
$\mc W^{r}(\bs m)$. In particular, as $\bs \ell \cdot \nabla U =0$,
$\bs b(\bs x) \cdot \bs n(\bs x) = - |\nabla U(\bs x)|^2 <0$ for all
$\bs x\in \partial \mc W^{2r_0}(\bs m)$ and $r_0$ small enough.

Set
\begin{equation}
\label{30}
{\color{blue} \mathcal{E}(\boldsymbol{m}) } :=\mathcal{W}^{r_{0}}(\bm{m})\;,
\quad \boldsymbol{m}\in\mathcal{M}_{0}\;.
\end{equation}

For $\lambda>0$, $\bs{g}\colon \mc M_0 \rightarrow\mathbb{R}$,
denote by $\phi_{\epsilon} = \phi_{\epsilon}^{\lambda, \bs{g}}$
the unique solution of the resolvent equation
\begin{equation}
\label{e_res}
(\lambda-\theta_{\epsilon}^{(1)} \mathscr{L}_{\epsilon})\,
\phi_{\epsilon} \;=\; G \;:=\;
\sum_{\bs m\in \mc M_0} \mb g (\bs m)\,
\chi_{_{\mathcal{E}(\bs m)}} \;,
\end{equation}
where $\color{blue} \chi_{\mc A}$, $\mc A\subset \bb R^d$, represents
the indicator function of the set $\mc A$. The function on the
right-hand side vanishes at $(\cup_{\bs m\in\mc M_0} \mc E(\bs m))^c$
and is constant on each well $\mc E(\bs m')$.  The second main result
of this article reads as follows.

\begin{thm}
\label{t01}
For all $\lambda>0$ and
$\bs{g}\colon \mc M_0 \rightarrow\mathbb{R}$,
\begin{equation*}
\lim_{\epsilon\rightarrow0}\, \max_{\bs m\in \mc M_0} \,
\sup_{x\in \mc E(\mb m)}
\vert\, \phi_{\epsilon} (x) - \bs{f}(\mb m) \,
\vert \;=\; 0\;,
\end{equation*}
where $\mb f$ is the solution of the reduced resolvent equation
\begin{equation*}
(\lambda - \mf L_1)\, \bs{f} \;=\; \mb g \;,
\end{equation*}
and $\mf L_1$ is the generator introduced in \eqref{28}.
\end{thm}

\subsection*{Comments and Remarks}

The proofs of Theorems \ref{t00} and \ref{t01} are entirely based on
the metastable behavior of the stochastic differential equation
\begin{equation}
\label{sde}
d\boldsymbol{x}_{\epsilon}(t) \,=\,
\bs b(\bs x_\epsilon (t)) \, dt
\,+\,
\sqrt{2\epsilon}\,dW_{t}\;,
\end{equation}
where $\epsilon>0$ denotes a small parameter corresponding to the
temperature of the system, and $W_t$ a $d$-dimension Brownian motion.

The proof of Theorem \ref{t01} is divided in two parts. We first show
in Section \ref{sec2} that $\phi_{\epsilon}$ is asymptotically
constant on each well $\mc E(\bs m)$. Then, we prove that the average
of the solution $\phi_{\epsilon}$ on a well $\mc E(\bs m)$ converges
to $\mb f$.

In Section \ref{sec3}, we deduce from Theorem \ref{t01} and with ideas
introduced in \cite{LLM}, the convergence of the finite-dimensional
distributions of the process $\boldsymbol{x}_{\epsilon}(\cdot)$. A
similar result has been obtained by Sugiura in \cite{su} with
different ideas in the case $\bs \ell =0$.

\section{Mixing time of diffusions}
\label{sec-ap3}

The main result of this section, Theorem \ref{t_main2}, provides an
estimate on the mixing time of a diffusion on $\bb R^d$. The proof of
this result can be skipped in a first reading as the ideas and
techniques used to derive the bound on the mixing time will not be
used in the next sections.

Fix a function $U_0\colon \bb R^d \to \bb R$ of class $ C^{3}$ and a
vector field $\boldsymbol{\ell}_0 \colon \bb R^d \to \bb R^d$ of class
$C^{2}$ such that
\begin{equation}
\label{eq:orth}
(\nabla U_0) (\bs x) \cdot\boldsymbol{\ell}_0  (\bs x) \,=\,
(\nabla\cdot\boldsymbol{\ell}_0) (\bs x) \, = \, 0 \quad
\text{for all} \; \bs x \in \bb R^d\;.
\end{equation}
Suppose that $U_0$ has a local minimum at $\bs x = \boldsymbol{0}$ and
that it has no other critical point in a neighborhood of the origin.
Furthermore, we assume, for convenience, that $U_0(\boldsymbol{0})=0$
.

Consider a vector field $\bs b_0: \bb R^d \to \bb R^d$ of class $C^1$
such that

\begin{enumerate}
\item $\bs b_0$ vanishes only at the origin, which is a stable
equilibrium point for the dynamical system
\begin{equation}
\label{eq:x_0}
\dot {\boldsymbol{y}} (t) \,=\, \boldsymbol{b}_0(\boldsymbol{y}(t)) \;.
\end{equation}

\item There exists $r_3>0$ such that
\begin{equation*}
\boldsymbol{b}_0(\boldsymbol{x}) \,=\, -\,
(\nabla U_0) (\boldsymbol{x}) \,-\, \boldsymbol{\ell}_0(\boldsymbol{x})
\;, \quad \boldsymbol{x}\in B(0,r_3)\;.
\end{equation*}

\item There exist $R>0$ and a finite constant $C_{1}$ such that
\begin{equation}
\label{growth}
|\boldsymbol{b}_0 (\boldsymbol{x})|\le
C_{1}|\, \boldsymbol{x}|\;\;\;\;\text{and}\;\;\;\;
\left\Vert D\boldsymbol{b}_0(\boldsymbol{x})\right\Vert
\le C_{1}\, |\boldsymbol{x}|
\end{equation}
for all $|\boldsymbol{x}|>R$, where the matrix norm
is defined as
\begin{equation*}
\Vert\mathbb{M}\Vert=\sup_{|\boldsymbol{y}|=1}
|\mathbb{M}\boldsymbol{y}|\;.
\end{equation*}

\item Let $\mathbb{H}_0=(\nabla^{2}U_0)(\boldsymbol{0})$ and
$\mathbb{L}_0= (D\boldsymbol{\ell}_0) (\boldsymbol{0})$.  Assume that
\begin{equation}
-\, \left\langle \boldsymbol{b}_0(\boldsymbol{x}),\,
\mathbb{H}_0 \boldsymbol{x}\right\rangle
\,\ge\, \frac{1}{2}\, |\mathbb{H}_0\boldsymbol{x}|^{2}\;\;\;\;
\text{for all }\boldsymbol{x}\in\mathbb{R}\;.
\label{eq:contraction}
\end{equation}
where $\langle\,\cdot\,,\,\cdot\,\rangle$ represents the scalar
product in $\bb R^d$.
\end{enumerate}

The main result of this section requires some notation. Let
\begin{equation}
\mathbb{A}(\boldsymbol{x}) \,:=\,
( D\boldsymbol{b}_0) (\boldsymbol{x})\;,
\quad
\mathbb{A} \,:=\, \mathbb{A}(\boldsymbol{0}) \;, \quad
\text{so that}\quad
\mathbb{A} \,=\, -\, (\mathbb{H}_0 +\mathbb{L}_0)\;.
\label{eq:A(x)}
\end{equation}
By \cite[Lemmata 4.5 and 4.1]{LS-22}, all the eigenvalues of the
matrix $\bb A$ have negative real parts. Therefore, by \cite[Theorems
2 and 3, p.414]{LT85}, there exists a positive definite matrix
$\mathbb{K}$ such that
\begin{equation}
\mathbb{A}^{\dagger}\mathbb{K}+\mathbb{K}\mathbb{A}
\,= \, -\, \mathbb{I}\;,
\label{eq:ricatti1}
\end{equation}
where $\bb I$ is the identity.

Let $\mathcal{D}_{r}\subset\mathbb{R}^{d}$, $r>0$, be the set given by
\begin{equation}
{\color{blue} \mathcal{D}_{r}} \;:=\; \{\boldsymbol{x}\in\mathbb{R}^{d}:
\left\langle
\boldsymbol{x},\,\mathbb{H}_0\, \boldsymbol{x}\right\rangle \le r^{2}\}\;.
\label{eq:D_r}
\end{equation}
By \eqref{eq:ricatti1}, there exists $r'_4>0$ such that
\begin{equation}
\left\Vert \, (\mathbb{A}(\boldsymbol{x})-\mathbb{A})^{\dagger}
\mathbb{K}+\mathbb{K}(\mathbb{A}(\boldsymbol{x})
-\mathbb{A})\right\Vert \leq\frac{1}{2}
\label{eq:A(x)-A}
\end{equation}
for all $\boldsymbol{x}\in B(0, r'_4)$. By \eqref{eq:D_r},
$\mathcal{D}_{2r_4} \subset B(0, r'_4)$ for some $r_4>0$.
Take $r_4$ small enough so that
\begin{equation}
\mathcal{D}_{2r_4} \,\subset\, B(0, r_3) \;,
\label{eq:condr_4}
\end{equation}
where $r_3$ has been introduced in condition (2) above.

The main result of this section reads as follows.  Denote by
$\color{blue} d_{\textup{TV}}(\mu, \nu)$ the total variation distance
between probability measures $\mu$ and $\nu$.  Let
$\bs y_\epsilon (\cdot)$ be the diffusion given by
\begin{equation}
d\boldsymbol{y}_{\epsilon}(t)
\,=\, \boldsymbol{b}_0(\boldsymbol{y}_{\epsilon}(t))\, dt
\,+\, \sqrt{2\epsilon}\,dW_{t}\;.
\label{eq:x_eps}
\end{equation}
The process $\bs y_\epsilon (\cdot)$ starting at
$\boldsymbol{x}\in\mathbb{R}^{d}$, is represented by
$\boldsymbol{y}_{\epsilon}(t;\boldsymbol{x})$.  Let
$t_{\epsilon}=\epsilon^{- \theta}$ for some $\theta\in (0,\,1/3)$.

\begin{thm}
\label{t_main2}
Denote by $\pi_\epsilon$ the stationary state of the diffusion
$\bs y_\epsilon (\cdot)$.  Then,
\begin{equation*}
\lim_{\epsilon\rightarrow0}
\sup_{\boldsymbol{x}\in\mathcal{D}_{r_4}}
d_{\textup{TV}} \big(\,
\boldsymbol{y}_{\epsilon}(t_{\epsilon};\boldsymbol{x}),\,
\pi_\epsilon\, \big) \,= \, 0\;.
\end{equation*}
\end{thm}

\begin{remark}
\label{rem:posev}
The proof of this result is largely based on
\cite{BJ,LeeRamSeo}.  Theorem \ref{t_main2} follows from
\cite[Theorem 2.2]{BJ} when $\bb A$ is negative definite.  As
mentioned above, all the eigenvalues of matrix $\bb A$ have negative
real parts, but $\bb A$ might not be negative definite. The purpose of
this section is to extend \cite[Theorem 2.2]{BJ} to this
situation.
\end{remark}

\subsection*{Proof of Theorem  \ref{t_main2}}

The main idea of proof is to approximate the difference
$\boldsymbol{y}_{\epsilon}(t)-\boldsymbol{y}(t)$ by a Gaussian
process. Let
\begin{equation*}
\widehat{\boldsymbol{\xi}}(t) \,=\,
\frac{1}{\sqrt{2\epsilon}}\left(\boldsymbol{y}_{\epsilon}(t)
-\boldsymbol{y}(t)\right)\;.
\end{equation*}
By (\ref{eq:x_eps}) and (\ref{eq:x_0}),
\begin{equation*}
d\widehat{\boldsymbol{\xi}}(t) \,=\,
\frac{1}{\sqrt{2\epsilon}}
\big\{\, \boldsymbol{b}_0(\boldsymbol{y}_{\epsilon}(t))
-\boldsymbol{b}_0(\boldsymbol{y}(t))\, \big\}\, dt \,+\, dW_{t}
\simeq (D\boldsymbol{b}_0) (\boldsymbol{y}(t)) \,
\widehat{\boldsymbol{\xi}}(t)dt+dW_{t}\;.
\end{equation*}
Hence, it is natural to conjecture that
$\widehat{\boldsymbol{\xi}}(t)\simeq\boldsymbol{\xi}(t)$ where
$\boldsymbol{\xi}(t)$ is the Gaussian process defined by the SDE
\begin{equation}
d\boldsymbol{\xi}(t)=\mathbb{A}(\boldsymbol{y}(t))
\, \boldsymbol{\xi}(t) \, dt \,+\, dW_{t}
\;, \quad \boldsymbol{\xi}(0)=\boldsymbol{0} \;.
\label{eq:xi}
\end{equation}

Let
\begin{equation}
\boldsymbol{z}_{\epsilon}(t)\, :=\,
\boldsymbol{y}(t) \,+\, \sqrt{2\epsilon}\, \boldsymbol{\xi}(t)\;.
\label{eq:z(t)}
\end{equation}
By the previous discussion, we expect that
$\boldsymbol{y}_{\epsilon}(t)\simeq\boldsymbol{z}_{\epsilon}(t)$.

\begin{lem}
\label{prop:tvm1}
There exists $r_4>0$ such that
\begin{equation}
\lim_{\epsilon\rightarrow0}\sup_{x\in\mathcal{D}_{r_4}}
d_{\mathrm{TV}}\left(\boldsymbol{y}_{\epsilon}
(t_{\epsilon};\boldsymbol{x}),\,\boldsymbol{z}_{\epsilon}
(t_{\epsilon};\boldsymbol{x})\right)=0\;.
\label{eq:tvm1-0}
\end{equation}
\end{lem}

Denote by $\mathcal{N}(\boldsymbol{\mu},\,\Sigma)$ the normal
distribution with mean $\boldsymbol{\mu}$ and covariance $\Sigma$.

\begin{lem}
\label{prop:tvm2}
There exists $r_4>0$ such that
\begin{equation*}
\lim_{\epsilon\rightarrow0}\sup_{x\in\mathcal{D}_{r_4}}
d_{\mathrm{TV}}\left(\boldsymbol{z}_{\epsilon}
(t_{\epsilon};\boldsymbol{x}),\,
\mathcal{N}(0,2\epsilon\mathbb{H}^{-1})\right)=0\;.
\end{equation*}
\end{lem}

\begin{proof}
The proof is presented in \cite[Proposition 3.6]{BJ}, and relies
on the fact that $\boldsymbol{z}_{\epsilon}(\cdot;\boldsymbol{x})$ is
a Gaussian process. In particular,
$\boldsymbol{z}_{\epsilon}(t;\boldsymbol{x})$ is a normal random
variable whose mean and variance can be expressed explicitly. The
assertion is thus reduced to a computation of the total variation
distance between two normal random variables.

Denote by $\lambda>0$ the smallest eigenvalue of $\bb H_0$.  The proof
starts at \cite[display (3.22)]{BJ}, and requires the bound
\begin{equation*}
\left|\boldsymbol{y}(t)\right|^{2}
\, \le\, |\boldsymbol{y}(0)|^{2}\, e^{-\lambda t} \;,
\end{equation*}
and \cite[Lemma B.2]{BJ}.  In the present context, Lemma
\ref{lem:stability} replaces the first estimate, and \cite[Lemma
B.2]{BJ} holds because it only needs all the eigenvalues of
$(D\boldsymbol{b}_0)(\boldsymbol{0})$ to have a positive real part, a
property satisfied by our model as mentioned in Remark
\ref{rem:posev}.
\end{proof}

\begin{proof}[Proof of Theorem \ref{t_main2}]
Denote by $p_{t}^{\epsilon}(\cdot,\,\cdot)$ the transition kernel of
the process $\boldsymbol{y}_{\epsilon}(\cdot)$ and by
$\pi_{\epsilon}(\cdot)$ the density of the measure
$\pi_{\epsilon}(d\bs x)$:
$\pi_{\epsilon}(d\boldsymbol{x}) = \pi_{\epsilon}(\boldsymbol{x})
d\boldsymbol{x}$.  By definition, and since $\pi_\epsilon$ is the
stationary state of the process $\bs y_\epsilon(\cdot)$,
\begin{align*}
\mathrm{d}_{\mathrm{TV}}\left(\boldsymbol{y}_{\epsilon}
(t_{\epsilon};\boldsymbol{x}),\,\pi_{\epsilon}\right)
& =\frac{1}{2}\int_{\mathbb{R}^{d}}\left|p_{t^{\epsilon}}^{\epsilon}
(\boldsymbol{x},\,\boldsymbol{y})
-\pi_{\epsilon}(\boldsymbol{y})\right|d\boldsymbol{y}
\\
& =\frac{1}{2}\int_{\mathbb{R}^{d}}\Big|\,
\int_{\mathbb{R}^{d}} \big[ \, p_{t^{\epsilon}}^{\epsilon}
(\boldsymbol{x},\,\boldsymbol{y})
- p_{t^{\epsilon}}^{\epsilon}(\boldsymbol{x}',\,\boldsymbol{y})\,\big]
\, \pi_{\epsilon}(\boldsymbol{x}')\, d\boldsymbol{x}'\, \Big|
\, d\boldsymbol{y} \;.
\end{align*}
The previous expression is bounded by
\begin{equation*}
\frac{1}{2}\int_{\mathbb{R}^{d}}
\int_{\mathbb{R}^{d}} \big| \, p_{t^{\epsilon}}^{\epsilon}
(\boldsymbol{x},\,\boldsymbol{y})
- p_{t^{\epsilon}}^{\epsilon}(\boldsymbol{x}',\,\boldsymbol{y})\,\big|
\, \pi_{\epsilon}(\boldsymbol{x}')\, d\boldsymbol{x}'
\, d\boldsymbol{y}
\;=\;
\int_{\mathbb{R}^{d}}
\mathrm{d}_{\mathrm{TV}} (\, \boldsymbol{y}_{\epsilon}
(t_{\epsilon};\boldsymbol{x}),\,
\boldsymbol{y}_{\epsilon}(t_{\epsilon};\boldsymbol{x}') \,)
\, \pi_{\epsilon}(\boldsymbol{x}')\, d\boldsymbol{x}'\;.
\end{equation*}
By \eqref{apf1}, the right-hand side is less than or equal to
\begin{equation*}
\int_{\mathcal{D}_{r_4}}\mathrm{d}_{\mathrm{TV}}
\left(\, \boldsymbol{y}_{\epsilon}(t_{\epsilon};
\boldsymbol{x}),\,\boldsymbol{y}_{\epsilon}(t_{\epsilon};\boldsymbol{x}')
\, \right) \, \pi_{\epsilon}(\boldsymbol{x}')
\,d\boldsymbol{x}' \,+\, C_0 \, \epsilon
\end{equation*}
for some finite constant $C_0$.  By Lemma \ref{prop:tvm1}, and a
triangular inequality,
\begin{equation*}
\limsup_{\epsilon\rightarrow0}
\sup_{\boldsymbol{x}\in\mathcal{D}_{r_4}}
\int_{\mathcal{D}_{r_4}}\left|\, \mathrm{d}_{\mathrm{TV}}
\left(\boldsymbol{y}_{\epsilon}(t_{\epsilon};\boldsymbol{x}),\,
\boldsymbol{y}_{\epsilon}(t_{\epsilon};\boldsymbol{x}')\right)
\,-\, \mathrm{d}_{\mathrm{TV}}\left(\boldsymbol{z}_{\epsilon}
(t_{\epsilon};\boldsymbol{x}),\,\boldsymbol{z}_{\epsilon}
(t_{\epsilon};\boldsymbol{x}')\right)\, \right|\,
\pi_{\epsilon}(\boldsymbol{x}')\,d\boldsymbol{x}' \;=\; 0\;.
\end{equation*}

It remains to show that
\begin{equation}
\limsup_{\epsilon\rightarrow0}
\sup_{\boldsymbol{x}\in\mathcal{D}_{r_4}}
\int_{\mathcal{D}_{r_4}}
\mathrm{d}_{\mathrm{TV}}\left(\boldsymbol{z}_{\epsilon}
(t_{\epsilon};\boldsymbol{x}),\,\boldsymbol{z}_{\epsilon}
(t_{\epsilon};\boldsymbol{x}')\right)\,
\pi_{\epsilon}(\boldsymbol{x}')\,d\boldsymbol{x}' \;=\; 0\;.
\label{eq:dtt3}
\end{equation}
Since the integrand is bounded  by
\begin{equation*}
\mathrm{d}_{\mathrm{TV}}\left(
\boldsymbol{z}_{\epsilon}(t_{\epsilon};\boldsymbol{x}),\,
\mathcal{N}(0,2\epsilon\mathbb{H}^{-1})\right)
\,+\, \mathrm{d}_{\mathrm{TV}}\left(
\mathcal{N}(0,2\epsilon\mathbb{H}^{-1}),\,
\boldsymbol{z}_{\epsilon}(t_{\epsilon};\boldsymbol{x}')\right)\;,
\end{equation*}
assertion (\ref{eq:dtt3}) follows from Lemma \ref{prop:tvm2}.
\end{proof}

\subsection*{Proof of Lemma  \ref{prop:tvm1}}
\label{sec52}

The proof is similar to the one presented in \cite[Section
3.3]{BJ}, which is based on conditions (C) or (H) of that
article. These conditions, however, are only used in the proof of
Lemma \ref{prop:tvm1} to derive the estimates presented in Lemmata
\ref{lem:stability}, \ref{lem:mom_est}, \ref{lem:mom_Yt},
\ref{lem:mom_Yt_sup}, and Proposition \ref{prop:main_apprx}.

Fix $\delta_{\epsilon}=\epsilon^{c}$ for some $c>0$. As
\begin{equation*}
\boldsymbol{y}_{\epsilon}(t_\epsilon;\boldsymbol{x})
\;=\; \boldsymbol{y}_{\epsilon}(\delta_{\epsilon};
\boldsymbol{y}_{\epsilon}(t_{\epsilon}-\delta_{\epsilon};
\boldsymbol{x}))\;, \quad
\boldsymbol{z}_{\epsilon}(t_\epsilon;\boldsymbol{x})
\;=\;\boldsymbol{z}_{\epsilon}(\delta_{\epsilon};
\boldsymbol{z}_{\epsilon}(t_{\epsilon}-
\delta_{\epsilon};\boldsymbol{x})) \;,
\end{equation*}
we have that
\begin{equation}
\label{eq:dtvdec}
\begin{aligned}
d_{\mathrm{TV}}\left(\,
\boldsymbol{y}_{\epsilon}(t_{\epsilon};
\boldsymbol{x}),\,\boldsymbol{z}_{\epsilon}(t_{\epsilon};
\boldsymbol{x}) \, \right)
\; & \le \;
d_{\mathrm{TV}}\left(\,
\boldsymbol{y}_{\epsilon}(\delta_{\epsilon};
\boldsymbol{y}_{\epsilon}(t_{\epsilon}-\delta_{\epsilon};
\boldsymbol{x})) \, ,\,\boldsymbol{z}_{\epsilon}(\delta_{\epsilon};
\boldsymbol{y}_{\epsilon}(t_{\epsilon}-\delta_{\epsilon};
\boldsymbol{x})) \, \right) \\
&  + \; d_{\mathrm{TV}}\left(\,
\boldsymbol{z}_{\epsilon}(\delta_{\epsilon};
\boldsymbol{y}_{\epsilon}(t_{\epsilon}-\delta_{\epsilon};
\boldsymbol{x})) \, ,\,\boldsymbol{z}_{\epsilon}(\delta_{\epsilon};
\boldsymbol{z}_{\epsilon}(t_{\epsilon}-\delta_{\epsilon};
\boldsymbol{x})) \, \right)\;.
\end{aligned}
\end{equation}
The first term on the right-hand side is bounded in \cite[Proposition
3.3]{BJ} and the second one in \cite[Proposition 3.4]{BJ}. The proof
relies on the estimate presented in Proposition \ref{prop:main_apprx}
below.

We sketch the proof of these bounds.  For the first one, fix
$\boldsymbol{x}\in\mathcal{D}_{r_4}$ and denote by $\mathbb{P}_{Y}$
and $\mathbb{P}_{Z}$ the law of process
$(\boldsymbol{y}_{\epsilon}(s;\boldsymbol{x}))_{s\in[0,\,\delta_{\epsilon}]}$
and
$(\boldsymbol{z}_{\epsilon}(s;\boldsymbol{x}))_{s\in[0,\,\delta_{\epsilon}]}$,
respectively. By Pinsker inequality,
\begin{equation}
d_{\mathrm{TV}}\left(\boldsymbol{y}_{\epsilon}
(\delta_{\epsilon};\boldsymbol{x}),\,
\boldsymbol{z}_{\epsilon}(\delta_{\epsilon};
\boldsymbol{x})\right)^{2} \;\le\;
-\,2\, \mathbb{E}_{\mathbb{P}_{Y}}
\Big[\log\frac{\textup{d}\mathbb{P}_{Z}}
{\textup{d}\mathbb{P}_{Y}}\Big]\;.
\label{eq:pinsker}
\end{equation}
The SDE describing the process $\boldsymbol{z}_{\epsilon}(\cdot)$ can
be written as
\begin{equation}
d\boldsymbol{z}_{\epsilon}(t) \;=\;
\Big\{\, \boldsymbol{b}_0(\boldsymbol{y}(t))
+D\boldsymbol{b}_0(\boldsymbol{y}(t))\,
[\, \boldsymbol{z}_{\epsilon}(t)-\boldsymbol{y}(t)\,]\, \Big\}\,dt
\,+\, \sqrt{2\epsilon}\,dW_{t}\;.
\label{eq:sdeZ}
\end{equation}
Hence, by Girsanov theorem, (\ref{eq:x_eps}), and (\ref{eq:sdeZ}),
\begin{align*}
\log\frac{\textup{d}\mathbb{P}_{Z}}{\textup{d}\mathbb{P}_{Y}}
= & \frac{1}{\sqrt{2\epsilon}}\int_{0}^{\delta_{\epsilon}}
\big \langle \, \boldsymbol{b}_0(\boldsymbol{y}_{\epsilon}(t))
-\boldsymbol{b}_0(\boldsymbol{y}(t))
+D\boldsymbol{b}_0(\boldsymbol{y}(t))\,
[\boldsymbol{y}_{\epsilon}(t)
-\boldsymbol{y}(t)] \, ,\,dW_{s}\, \big \rangle \\
- & \frac{1}{4\epsilon}\, \int_{0}^{\delta_{\epsilon}}
\big|\, \boldsymbol{b}_0(\boldsymbol{y}_{\epsilon}(t))
-\boldsymbol{b}_0 (\boldsymbol{y}(t))
+D\boldsymbol{b}_0(\boldsymbol{y}(t))
\, [ \boldsymbol{y}_{\epsilon}(t)
-\boldsymbol{y}(t)\,] \,\big |^{2}\, \textup{d}s\;.
\end{align*}
Thus, the left-hand side of (\ref{eq:pinsker}) is bounded by
\begin{equation}
\frac{1}{2\epsilon}\int_{0}^{\delta_{\epsilon}}
\mathbb{E}_{\boldsymbol{x}}\left[ \, \big| \,
\boldsymbol{b}_0 (\boldsymbol{y}_{\epsilon}(t))
-\boldsymbol{b}_0 (\boldsymbol{y}(t))
+D\boldsymbol{b}_0 (\boldsymbol{y}(t))
\, [\, \boldsymbol{y}_{\epsilon}(t)
-\boldsymbol{y}(t) \,] \, \big|^{2}\, \right]\textup{d}s\;.
\label{eq:pnbd1}
\end{equation}
By condition (\ref{growth}) on $D\boldsymbol{b}_0$ (which is milder
than that of \cite{BJ}), and the argument presented in
\cite[Proposition 3.3]{BJ}, we can conclude that
$d_{\mathrm{TV}}\left(\boldsymbol{y}_{\epsilon}
(\delta_{\epsilon};\boldsymbol{x}) \, ,\,
\boldsymbol{z}_{\epsilon}(\delta_{\epsilon}; \boldsymbol{x})\right)
\le \delta_{\epsilon}^{1/2}$.  We emphasize that in order to control
the term
$\boldsymbol{b}_0(\boldsymbol{y}_{\epsilon}(t)) -
\boldsymbol{b}_0(\boldsymbol{y}(t))$, we need the estimate of the
fourth moment stated in Proposition \ref{prop:main_apprx}. In all
other places, a bound of the second moment suffices.

By Lemma \ref{lem:mom_est}, the probability that the starting point
$\boldsymbol{y}_{\epsilon}(t_{\epsilon}-\delta_{\epsilon};\boldsymbol{x})$
does not belong to $\mathcal{D}_{r_4}$ vanishes as $\epsilon\to 0$.
This fact together with the bound obtained in the previous paragraph
yields that
\begin{equation*}
\lim_{\epsilon\rightarrow0}\sup_{\boldsymbol{x}\in\mathcal{D}}
d_{\mathrm{TV}}\left(\boldsymbol{y}_{\epsilon}(\delta_{\epsilon};
\boldsymbol{y}_{\epsilon}(t_{\epsilon}
-\delta_{\epsilon};\boldsymbol{x})) \,,\,
\boldsymbol{z}_{\epsilon}(\delta_{\epsilon};
\boldsymbol{y}_{\epsilon}(t_{\epsilon}
-\delta_{\epsilon};\boldsymbol{x}))\right)=0\;.
\end{equation*}
This completes the estimate of the first term on the right-hand side
of (\ref{eq:dtvdec}).

We turn to the second term. By Proposition \ref{prop:main_apprx} the
starting points
$\boldsymbol{y}_{\epsilon}(t_{\epsilon}-\delta_{\epsilon};\boldsymbol{x})$
and
$\boldsymbol{z}_{\epsilon}(t_{\epsilon}-\delta_{\epsilon};\boldsymbol{x})$
are close. Since the process $\bs z_\epsilon (\cdot)$ is Gaussian, the
distance
\begin{equation*}
d_{\mathrm{TV}}\left(\boldsymbol{z}_{\epsilon}(\delta_{\epsilon};
\boldsymbol{w}),\,\boldsymbol{z}_{\epsilon}(\delta_{\epsilon};
\boldsymbol{w}')\right)
\end{equation*}
is complete determined by $\boldsymbol{w}$ and $\boldsymbol{w}'$, and
one can follow the arguments presented in \cite[Proposition
3.3]{BJ}. All error term appearing in the proof are uniform on
the starting point $\boldsymbol{x}\in\mathcal{D}_{r_4}$ because all
estimates obtained in the next subsections are uniform. Thus,
\begin{equation*}
\lim_{\epsilon\rightarrow0}\sup_{\boldsymbol{x}\in\mathcal{D}_{r_4}}
d_{\mathrm{TV}}\left(\boldsymbol{z}_{\epsilon}(\delta_{\epsilon};
\boldsymbol{y}_{\epsilon}(t_{\epsilon}-\delta_{\epsilon};
\boldsymbol{x})),\,\boldsymbol{z}_{\epsilon}(\delta_{\epsilon};
\boldsymbol{z}_{\epsilon}(t_{\epsilon}
-\delta_{\epsilon};\boldsymbol{x}))\right) \;=\; 0\;.
\end{equation*}
This completes the proof of Lemma \ref{prop:tvm1}.

\subsection*{Exponential Stability}

In this subsection and in the next we provide the estimates used in
the proofs of Lemmata \ref{prop:tvm1} and \ref{prop:tvm2}.  Recall
that we denote by $\lambda>0$ the smallest eigenvalue of
$\mathbb{H}_0$. The following lemma substitutes \cite[display
(2.2)]{BJ}.

\begin{lem}
\label{lem:stability}
For all $t\ge0$,
\begin{equation}
\left\langle
\boldsymbol{y}(t),\,\mathbb{H}_0 \, \boldsymbol{y}(t)\right\rangle
\, \leq\, \mathrm{e}^{-\lambda t}
\left\langle
\boldsymbol{y}(0),\,\mathbb{H}_0\, \boldsymbol{y}(0)\right\rangle
\;.
\label{eq:conv_uld0}
\end{equation}
\end{lem}

\begin{proof}
By (\ref{eq:x_0}) and (\ref{eq:contraction}),
\begin{equation}
\frac{d}{dt}\left\langle \boldsymbol{y}(t),\,
\mathbb{H}_0 \, \boldsymbol{y}(t)\right\rangle
\,=\, 2\, \left\langle \boldsymbol{b}_0(\boldsymbol{y}(t)),\,
\mathbb{H}_0 \, \boldsymbol{y}(t)\right\rangle
\, \le\, -\, \left|\, \mathbb{H}_0 \, \boldsymbol{y}(t) \, \right|^{2}
\, \le\, -\, \lambda\, \left\langle \boldsymbol{y}(t),\,
\mathbb{H}_0 \, \boldsymbol{y}(t)\right\rangle
\label{eq:conv1}
\end{equation}
since $\lambda$ is the smallest eigenvalue of $\mathbb{H}_0$.
\end{proof}

\begin{remark}
\label{rem:stability}
Fix $r>0$.  By the previous lemma,
$\boldsymbol{y}(t)\in\mathcal{D}_{r}$ for all $t\ge0$
provided $\boldsymbol{y}(0)\in\mathcal{D}_{r}$.
\end{remark}

A similar computation for $\boldsymbol{y}_{\epsilon}(t)$ instead of
$\boldsymbol{y}(t)$ yields the moment estimate stated in the next
lemma. This bound plays the role of \cite[condition (H)]{BJ}.  Denote
by $\mathbb{E}_{\boldsymbol{x}}$ the expectation with respect to
$\boldsymbol{y}_{\epsilon}(\cdot)$ starting at $\boldsymbol{x}$.
Moreover, from now on, all estimates presented hold only for
sufficiently small $\epsilon$.

\begin{lem}
\label{lem:mom_est}
Fix $r>0$. For all $n\ge1$, there exists a constant $C(n)>0$ such that
\begin{equation*}
\sup_{t\ge0}\sup_{\boldsymbol{x}\in\mathcal{D}_{r}}
\mathbb{E}_{\bs x} \left\langle \boldsymbol{y}_{\epsilon}(t),\,
\mathbb{H}_0 \, \boldsymbol{y}_{\epsilon}(t)\right\rangle ^{n}
\, \leq \, e^{-(n \lambda/ 4)\, t}
\left\langle \boldsymbol{x},\,\mathbb{H}_0\,
\boldsymbol{x}\right\rangle
\,+ \, C(n)\, \epsilon\;.
\end{equation*}
\end{lem}

\begin{proof}
By Ito's formula, (\ref{eq:contraction}), and similar computation with
(\ref{eq:conv1}), we get
\begin{equation}
d\left\langle \boldsymbol{y}_{\epsilon}(t),\,
\mathbb{H}_0\, \boldsymbol{y}_{\epsilon}(t)\right\rangle
\, \le\, \left[-\lambda\left\langle \boldsymbol{y}_{\epsilon}(t),\,
\mathbb{H}_0 \, \boldsymbol{y}_{\epsilon}(t)\right\rangle
+2\mathfrak{h\epsilon}\right]dt+\sqrt{2\epsilon}\left\langle
2\mathbb{H}_0 \, \boldsymbol{y}_{\epsilon}(t),\,dW_{t}\right\rangle
\;,
\label{eq:itoy}
\end{equation}
where $\mathfrak{h}=\text{tr}(\mathbb{H}_0)$. Thus, by Ito's formula
and (\ref{eq:itoy}),
\begin{align*}
& d\left\langle \boldsymbol{y}_{\epsilon}(t),\,\mathbb{H}_0\,
\boldsymbol{y}_{\epsilon}(t)\right\rangle ^{n} \; \le\;
n\left\langle \boldsymbol{y}_{\epsilon}(t),\,\mathbb{H}_0\,
\boldsymbol{y}_{\epsilon}(t)\right\rangle
^{n-1}\left[\, -\lambda\left\langle
\boldsymbol{y}_{\epsilon}(t),\,\mathbb{H}_0\,
\boldsymbol{y}_{\epsilon}(t)\right\rangle
+2\mathfrak{h\epsilon} \, \right]dt \\
& \qquad\qquad + \; \sqrt{2\epsilon}\left\langle
2\mathbb{H}_0\,  \boldsymbol{y}_{\epsilon}(t),\,dW_{t}\right\rangle
+\frac{n(n-1)}{2}\left\langle
\boldsymbol{y}_{\epsilon}(t),\,\mathbb{H}_0\,
\boldsymbol{y}_{\epsilon}(t)\right\rangle^{n-2}
\times8\epsilon\left|\mathbb{H}_0\,
\boldsymbol{y}_{\epsilon}(t)\right|^{2}dt\;.
\end{align*}
For $\epsilon$ sufficiently small, this expression is bounded by
\begin{equation*}
\left\langle \boldsymbol{y}_{\epsilon}(t),\,
\mathbb{H}_0\, \boldsymbol{y}_{\epsilon}(t)\right\rangle ^{n-1}
\Big[\, -\frac{n\lambda}{2}
\left\langle \boldsymbol{y}_{\epsilon}(t),\,\mathbb{H}_0\,
\boldsymbol{y}_{\epsilon}(t)\right\rangle
+2\mathfrak{h}n\epsilon\, \Big] \, dt
\;+\; \sqrt{2\epsilon}\left\langle 2\mathbb{H}_0\,
\boldsymbol{y}_{\epsilon}(t),\,dW_{t}\right\rangle \;.
\end{equation*}
Since
\begin{equation*}
\left\langle \boldsymbol{y}_{\epsilon}(t),\,\mathbb{H}_0\,
\boldsymbol{y}_{\epsilon}(t)\right\rangle ^{n-1}
\; \le\; \frac{n-1}{n}\left\langle \boldsymbol{y}_{\epsilon}(t),\,
\mathbb{H}_0\, \boldsymbol{y}_{\epsilon}(t)\right\rangle ^{n}+\frac{1}{n}\;,
\end{equation*}
for small enough $\epsilon>0$,
\begin{equation*}
d\left\langle \boldsymbol{y}_{\epsilon}(t),\,
\mathbb{H}_0\, \boldsymbol{y}_{\epsilon}(t)\right\rangle ^{n}
\;\le\; \Big [\, -\frac{n\lambda}{4}\left\langle
\boldsymbol{y}_{\epsilon}(t),\,
\mathbb{H}_0\, \boldsymbol{y}_{\epsilon}(t)\right\rangle
^{n}+c(n)\epsilon\, \Big ]\, dt
\;+\; \sqrt{2\epsilon}\left\langle
2\mathbb{H}_0\,  \boldsymbol{y}_{\epsilon}(t),\,dW_{t}\right\rangle
\end{equation*}
for some finite constant $c(n)$.
Hence, by Gronwall's inequality,
\begin{equation*}
\mathbb{E}_{\boldsymbol{x}}\left\langle
\boldsymbol{y}_{\epsilon}(t),\,\mathbb{H}_0\,
\boldsymbol{y}_{\epsilon}(t)\right\rangle ^{n}\le
e^{- (n\lambda/4) \, t}\left\langle \boldsymbol{x},\,\mathbb{H}_0\,
\boldsymbol{x}\right\rangle +\frac{4c(n)\epsilon}{n\lambda}\;,
\end{equation*}
as claimed
\end{proof}

It follows from the estimates derived in the previous lemma, the
argument presented in \cite[page 1192]{BJ} (cf. the last line of the
proof of \cite[Proposition 3.7]{BJ}) and the dominated and monotone
convergence theorems that there exists a finite constant $C_0$ such
that
\begin{equation}
\label{apf1}
\pi_{\epsilon}((\mathcal{D}_{r_4})^{c}) \,\leq\, C_0 \, \epsilon
\end{equation}
for all $\epsilon$ sufficiently small.

\subsection*{Gaussian Approximation}

Hereafter, we couple the processes $\boldsymbol{y}_{\epsilon}(\cdot)$,
$\xi(\cdot)$, and $\boldsymbol{z}_{\epsilon}(\cdot)$ by using the same
driving Brownian motion $W_{t}$. This coupled probability law and
associated expectation will be denoted by $\mb{P}_{\boldsymbol{x}}$
and $\mb{E}_{\boldsymbol{x}}$.

\begin{prop}[Gaussian approximation]
\label{prop:main_apprx}
There exist constants $\alpha_{1},\,\alpha_{2}>0$ such that
\begin{equation*}
\sup_{t\le t_{\epsilon}}
\sup_{\boldsymbol{x}\in\mathcal{D}_{r_4}}
\mb{E}_{\boldsymbol{x}}\left[\,
|\, \boldsymbol{y}_{\epsilon}(t; \bs x)-
\boldsymbol{z}_{\epsilon}(t; \bs x)\, |^{4}
\, \right] \, \leq\, \alpha_{1}\, \epsilon^{2+\alpha_{2}}\;.
\end{equation*}
\end{prop}

This proposition corresponds to \cite[display (3.12)]{BJ} which plays
a crucial role in the proof of the main result. Since the proof of
\cite[display (3.12)]{BJ} requires conditions (C) and (H) of
\cite{BJ}, and these conditions are not assumed here, we develop an
alternative approach below, based on \cite{LeeRamSeo}.

\begin{lem}
\label{lem:drift}
There exists $c>0$ such that
\begin{equation*}
\langle\mathbb{K}\boldsymbol{x},\,
\mathbb{A}(\boldsymbol{y}(t))\boldsymbol{x}\rangle
\, \le\, -\, c\,
\langle\boldsymbol{x},\,\mathbb{K}\boldsymbol{x}\rangle\;.
\end{equation*}
for all $t\ge0$ and $\boldsymbol{x}\in\mathcal{D}_{2r_4}$.
\end{lem}

\begin{proof}
By (\ref{eq:A(x)-A}) and (\ref{eq:ricatti1}),
\begin{align*}
2\langle\mathbb{K}\boldsymbol{x},\,
\mathbb{A}(\boldsymbol{y}(t))\boldsymbol{x}\rangle &
=\left\langle \boldsymbol{x},\,
\left[\mathbb{A}(\boldsymbol{y}(t))^{\dagger}\mathbb{K}
+\mathbb{K}\mathbb{A}(\boldsymbol{y}(t))\right]
\boldsymbol{x}\right\rangle \\
& \le\left\langle \boldsymbol{x},\,
\left[\mathbb{A}^{\dagger}\mathbb{K}+\mathbb{K}\mathbb{A}\right]
\boldsymbol{x}\right\rangle +\frac{1}{2}
|\boldsymbol{x}|^{2}=-\frac{1}{2}|\boldsymbol{x}|^{2} \;.
\end{align*}
As $\mathbb{K}$ is bounded, the previous term is less than or equal to
$-\, c\, \langle\boldsymbol{x},\,\mathbb{K}\boldsymbol{x}\rangle$ for
some positive contant $c$, as claimed.
\end{proof}

\begin{lem}
\label{lem:mom_Yt}
For all $n\ge1$, there exists a finite constant $C(n)>0$
such that
\begin{equation*}
\sup_{t\ge0}\sup_{\boldsymbol{x}\in\mathcal{D}_{r_4}}
\mb{E}_{\boldsymbol{x}}\left[\left\langle
\boldsymbol{\xi}(t),\,
\mathbb{K}\boldsymbol{\xi}(t)\right\rangle^{n}\right]\leq C(n)\;.
\end{equation*}
\end{lem}

\begin{proof}
By Ito's formula and Lemma \ref{lem:drift},
\begin{align}
d\langle\boldsymbol{\xi}(t),\,\mathbb{K}\boldsymbol{\xi}(t)\rangle
& =\; \left[\, 2\, \langle\, \mathbb{K}\boldsymbol{\xi}(t),\,
\mathbb{A}(\bs y (t) )\boldsymbol{\xi}(t)\, \rangle
+\mathfrak{k}\, \right]dt
+2\, \langle\mathbb{K}\boldsymbol{\xi}(t),\,\mathrm{d}W_{t}\rangle
\nonumber \\
& \le\left[-2c\langle\boldsymbol{\xi}(t),\,
\mathbb{K}\boldsymbol{\xi}(t)\rangle+\mathfrak{k}\right]dt
+2\langle\mathbb{K}\boldsymbol{\xi}(t),\,\mathrm{d}W_{t}\rangle \;,
\label{eq:dxi}
\end{align}
where $\mathfrak{k}=\text{tr}(\mathbb{K}).$ The remainder of the
proof is identical to the proof of Lemma \ref{lem:mom_est}.
\end{proof}

\begin{lem}
\label{lem:mom_Yt_sup}
For all $n\ge1$, there exists
a finite constant $C(n)>0$ such that
\begin{equation*}
\sup_{x\in\mathcal{D}_{r_4}}\mb{E}_{\boldsymbol{x}}
\big[\, \sup_{t\in[0,\,t_{\epsilon}]}
\left\langle \boldsymbol{\xi}(t),\,
\mathbb{K}\boldsymbol{\xi}(t)\right\rangle ^{n}\,\big]
\;\leq\; C(n)\, t_\epsilon^{n}\;.
\end{equation*}
\end{lem}

\begin{proof}
By H\"older's inequality, it is enough to prove the lemma for $n$
even. Assume that this is the case.  Integrating (\ref{eq:dxi}), as
the first term on the right-hand side is negative,
\begin{equation*}
\langle\boldsymbol{\xi}(t),\, \mathbb{K}\boldsymbol{\xi}(t)\rangle
\leq\langle\boldsymbol{x},\,\mathbb{K}\boldsymbol{x}\rangle
+2\int_{0}^{t}\langle\mathbb{K}\boldsymbol{\xi}(t),\,
\mathrm{d}W_{t}\rangle ds+\mathfrak{k}t\;.
\end{equation*}
Therefore,
\begin{equation}
\mb{E}_{\boldsymbol{x}}\Big[\, \sup_{t\in[0, t_\epsilon]}
\left\langle \boldsymbol{\xi}(t),\,\mathbb{K}
\boldsymbol{\xi}(t)\right\rangle ^{n}\,\Big]
\;\leq\;  C(n)\, \Big(\, \mb{E}_{\boldsymbol{x}}
\Big[\, \sup_{t\in[0, t_\epsilon ]}
\left|\int_{0}^{t}\langle\mathbb{K}\boldsymbol{\xi}(t),\,
\mathrm{d}W_{t}\rangle\right|^{n}\,\Big]
\,+\, t_\epsilon^{n} \, \Big)
\label{eq:bdx}
\end{equation}
for some finite constant $C(n)$.  By the Burkholder-Davis-Gundy
inequality and the H\"older inequality, the expectation on the
right-hand side is bounded by
\begin{equation*}
C(n)\, \mb{E}_{\boldsymbol{x}}\Big[\,
\Big(\, \int_{0}^{t_\epsilon}
|\mathbb{K}\boldsymbol{\xi}(t)|^{2}dt\, \Big)^{n/2}\, \Big]
\, \le\, C(n)\, t_\epsilon^{(n/2)-1}\,
\, \mb{E}_{\boldsymbol{x}}\Big[\,
\int_{0}^{t_\epsilon} \left\langle \boldsymbol{\xi}(t),\,
\mathbb{K}\boldsymbol{\xi}(t)\right\rangle ^{n/2}dt \,\Big]\;.
\end{equation*}
By Fubini's theorem and by Lemma \ref{lem:mom_Yt} [since $n$ is even],
this expression is less than or equal to $C(n)\,
t_\epsilon^{n/2}$. Inserting this bound is (\ref{eq:bdx}) completes
the proof of the lemma.
\end{proof}

The proof below is developed in \cite{LeeRamSeo} based on ideas of
\cite{BJ}.

\begin{proof}[Proof of Proposition \ref{prop:main_apprx}]
Fix $\boldsymbol{x}\in\mathcal{D}_{r_4}$, and remember from
\eqref{eq:xi} that $\bs\xi (t)$ depends on $\bs x$ through the
dynamical system $\bs y(\cdot)$ which starts from $\bs x$.  Let
\begin{equation}
\boldsymbol{r}_{\epsilon}(t) \,:=\,
\frac{\boldsymbol{y}_{\epsilon}(t)-\boldsymbol{z}_{\epsilon}(t)}
{\sqrt{2\epsilon}}
\,=\, \frac{\boldsymbol{y}_{\epsilon}(t)-\boldsymbol{y}(t)}
{\sqrt{2\epsilon}}-\boldsymbol{\xi}(t)\;.
\label{eq:wte}
\end{equation}
We need to prove that there exist positive finite constants
$\alpha_{1}$, $\alpha_{2}$ such that
\begin{equation}
\sup_{t\le t_\epsilon} \, \sup_{\bs x\in \mc D_{r_4}}\,
\mb E_{\boldsymbol{x}}\left[\langle\boldsymbol{r}_{\epsilon}(t),\,
\mathbb{K}\boldsymbol{r}_{\epsilon}(t)\rangle^{2}\right]
\, \leq\, \alpha_{1}\epsilon^{\alpha_{2}}\;.
\label{eq:obj0}
\end{equation}

Since $\boldsymbol{y}_{\epsilon}(t)$ and $\boldsymbol{\xi}(t)$ share
the same driving Brownian motion, by (\ref{eq:x_eps}), (\ref{eq:x_0}),
and (\ref{eq:xi}) that
\begin{equation}
\label{apf02}
\frac{\mathrm{d}}{\mathrm{d}t}\langle\boldsymbol{r}_{\epsilon}(t),\,
\mathbb{K}\boldsymbol{r}_{\epsilon}(t)\rangle
\, =\, 2\, \langle\mathbb{K}\boldsymbol{r}_{\epsilon}(t),\,
\mathbb{A}(\boldsymbol{y}(t)) \, \boldsymbol{r}_{\epsilon}(t)\rangle \,
-\, 2\, \langle\mathbb{K}\boldsymbol{r}_{\epsilon}(t),\,
\boldsymbol{q}_{\epsilon}(t)\rangle\;,
\end{equation}
where
\begin{equation*}
\boldsymbol{q}_{\epsilon}(t) \,=\,
\frac{1}{\sqrt{2\epsilon}} \, \big\{\,
\boldsymbol{b}_0(\boldsymbol{y}_{\epsilon}(t))
\,-\, \boldsymbol{b}_0(\boldsymbol{y}(t))
\,-\, (D\boldsymbol{b}_0) (\boldsymbol{y}(t))
\, (\boldsymbol{y}_{\epsilon}(t)-\boldsymbol{y}(t)) \,\big\} \;.
\end{equation*}

Let
\begin{equation*}
\mathcal{A}_{\epsilon}=\mathcal{A}_{\epsilon}(\boldsymbol{x})
:=\big\{\, \boldsymbol{y}_{\epsilon}(t)\in\mathcal{D}_{2r_4}\;
\text{for all }t\in[0,\, t_\epsilon]\,\big\} \;.
\end{equation*}
By Lemma \ref{lem:drift} and since $\bb K$ is positive-definite and
bounded, on the event $\mc A_\epsilon$ the right-hand side of
\eqref{apf02} is bounded by
\begin{align}
-\, c\, \langle\boldsymbol{r}_{\epsilon}(t),\,
\mathbb{K}\boldsymbol{r}_{\epsilon}(t)\rangle
\,-\, 2\, \langle\mathbb{K}\boldsymbol{r}_{\epsilon}(t),\,
\boldsymbol{q}_{\epsilon}(t)\rangle
\,\le\, -\, c_{1}\, \langle\boldsymbol{r}_{\epsilon}(t),\,
\mathbb{K}\boldsymbol{r}_{\epsilon}(t)\rangle
\,+\, C_{2}\, |\boldsymbol{q}_{\epsilon}(t)|^{2}
\label{eq:dW1}
\end{align}
for some finite positive constants $c_{1}$, $C_{2}$.

Fix $t\in[0,\, t_\epsilon]$. Since
$\boldsymbol{b}_0\in C^{2}(B(\bs 0, r_3),\,\mathbb{R}^{d})$, by
(\ref{eq:condr_4}) on the event $\mathcal{A}_{\epsilon}$,
\begin{equation*}
|\boldsymbol{q}_{\epsilon}(t)| \,\leq\,
\frac{C_0}{\sqrt{\epsilon}}\,
|\boldsymbol{y}_{\epsilon}(t)-\boldsymbol{y}(t)|^{2}
\,\le\, C_0\, \sqrt{\epsilon}\,
\big\{ \,
|\boldsymbol{r}_{\epsilon}(t)|^{2}+|\boldsymbol{\xi}(t)|^{2}\,\big\}\;,
\end{equation*}
for some finite constant $C_0$, whose value may change from line to
line. The second inequality follows from (\ref{eq:wte}). Therefore, by
(\ref{eq:dW1}),
\begin{equation*}
\frac{\mathrm{d}}{\mathrm{d}t}\langle\boldsymbol{r}_{\epsilon}(t),\,
\mathbb{K}\boldsymbol{r}_{\epsilon}(t)\rangle
\,\leq\, -\, c_{1}\, \langle\boldsymbol{r}_{\epsilon}(t),\,
\mathbb{K}\boldsymbol{r}_{\epsilon}(t)\rangle
\,+\, C_3\, \epsilon
\left[\langle\boldsymbol{r}_{\epsilon}(t),\,
\mathbb{K}\boldsymbol{r}_{\epsilon}(t)\rangle^{2}+
\left\langle \boldsymbol{\xi}(t),\,
\mathbb{K}\boldsymbol{\xi}(t)\right\rangle ^{2}\right]\;.
\end{equation*}

Let $\mathcal{B}_{\epsilon}=\mathcal{B}_{\epsilon}(\bs x)$ be the
event defined by
\begin{equation*}
\mathcal{B}_{\epsilon} \,:=\,
\Big\{\, \frac{C_3\, \epsilon}{c_{1}}\,
\Big(\, C_3\, \epsilon \, t_\epsilon\,
\sup_{s\in[0,\,t_\epsilon]} \left\langle \boldsymbol{\xi}(s),\,
\mathbb{K}\boldsymbol{\xi}(s)\right\rangle ^{2}\,\Big)
\,\leq\, \frac{1}{2}\, \Big\} \;.
\end{equation*}
By Perov's inequality \cite[Theorem 3.1]{Webb}, as
$\boldsymbol{r}_{\epsilon}(0)=\boldsymbol{0}$, it follows from the
previous inequality that on the event $\mc B_\epsilon$,
\begin{align*}
\langle\boldsymbol{r}_{\epsilon}(t),\,
\mathbb{K}\boldsymbol{r}_{\epsilon}(t)\rangle
\,\leq\, 2\, C_{3}\, \epsilon\, t_\epsilon\, e^{-c_{1}t}\,
\sup_{s\in[0,\,t_\epsilon]} \left\langle
\boldsymbol{\xi}(s),\,
\mathbb{K}\boldsymbol{\xi}(s)\right\rangle ^{2}\;.
\end{align*}
for all $t\in[0,\,t_\epsilon]$.  Hence, by Lemma \ref{lem:mom_Yt_sup},
\begin{equation*}
\sup_{t\in[0,\, t_\epsilon]} \, \sup_{\bs x\in \mc D_{r_4}}\,
\mb{E}_{\boldsymbol{x}}\left[\langle\boldsymbol{r}_{\epsilon}(t),\,
\mathbb{K}\boldsymbol{r}_{\epsilon}(t)\rangle^{2}\,
\bs{1}_{\mathcal{A}_{\epsilon}\cap\mathcal{B}_{\epsilon}}\right]
\,\leq\, C_0\, \epsilon^2\, t_\epsilon^6 \;=\;
C_0\, \epsilon^{2-6\theta}\;.
\end{equation*}
Since $\theta<1/3$, this proves (\ref{eq:obj0}) on the event
$\mathcal{A}_{\epsilon}\cap\mathcal{B}_{\epsilon}$.

We turn to the event
$(\mathcal{A}_{\epsilon}\cap\mathcal{B}_{\epsilon})^{c}$.  By the
Cauchy-Schwarz inequality,
\begin{equation*}
\mb E_{\boldsymbol{x}}\left[\langle\boldsymbol{r}_{\epsilon}(t),\,
\mathbb{K}\boldsymbol{r}_{\epsilon}(t)\rangle^{2} \,
\bs{1}_{(\mathcal{A}_{\epsilon}\cap\mathcal{B}_{\epsilon})^{c}}\right]^{2}
\, \leq\, \mb E_{\boldsymbol{x}}
\left[\langle\boldsymbol{r}_{\epsilon}(t),\,
\mathbb{K}\boldsymbol{r}_{\epsilon}(t)\rangle^{4}\right]
\, \left\{\, \mb P_{\boldsymbol{x}}(\mathcal{A}_{\epsilon}^{c})
+\mb P_{\boldsymbol{x}}(\mathcal{B}_{\epsilon}^{c}) \, \right\}\;.
\end{equation*}
By (\ref{eq:wte}) and the Cauchy-Schwarz inequality,
\begin{equation*}
\mb E_{\boldsymbol{x}}\left[\langle\boldsymbol{r}_{\epsilon}(t),\,
\mathbb{K}\boldsymbol{r}_{\epsilon}(t)\rangle^{4}\right]
\,\leq\, \frac{C_0}{\epsilon^{4}}
\, \left(\mb E_{\bs x} \left[\langle\boldsymbol{y}_{\epsilon}(t),\,
\mathbb{K}\boldsymbol{y}_{\epsilon}(t)\rangle^{4}
\,+\,
\langle\boldsymbol{y}(t),\,\mathbb{K}\boldsymbol{y}(t)\rangle^{4}
\, +\, \epsilon^{4}\, \langle\boldsymbol{\xi}(t),\,
\mathbb{K}\boldsymbol{\xi}(t)\rangle^{4}\right]\right)\;.
\end{equation*}
Hence, by Lemmata \ref{lem:stability}, \ref{lem:mom_est} and \ref{lem:mom_Yt},
\begin{equation*}
\mb E_{\boldsymbol{x}}\left[\langle\boldsymbol{r}_{\epsilon}(t),\,
\mathbb{K}\boldsymbol{r}_{\epsilon}(t)\rangle^{4}\right]
\,\le\, \frac{C_0}{\epsilon^{4}}
\end{equation*}
for all $t\ge0$, $\bs x\in\mc D_{r_4}$.

It remains to show that there exist $c_0>0$ and $C_0<\infty$ such
that
\begin{equation}
\sup_{\bs x\in \mc D_{r_4}}\,
\mb P_{\boldsymbol{x}}(\mathcal{A}_{\epsilon}^{c})
\,\le\, C_0\, \epsilon^{4+c_0}\;\;\;\text{and}\;\;\;
\sup_{\bs x\in \mc D_{r_4}}\,
\mb P_{\boldsymbol{x}}(\mathcal{B}_{\epsilon}^{c}) \,\le\,
C_0 \, \epsilon^{4+c_0}\;.
\label{eq:obj1}
\end{equation}
Consider the event $\mc A_\epsilon$.
On the complement of this set,
\begin{equation*}
\sup_{t\le t_\epsilon} \left\langle \boldsymbol{y}_{\epsilon}(t),\,
\mathbb{H}_0 \boldsymbol{y}_{\epsilon}(t)\right\rangle \,\ge\,
(2r_4)^2\;.
\end{equation*}
By (\ref{eq:itoy}),
\begin{equation*}
\left\langle \boldsymbol{y}_{\epsilon}(t),\,
\mathbb{H}_0 \boldsymbol{y}_{\epsilon}(t)\right\rangle
\,\le\, 2\, \mathfrak{h}\, \epsilon\, t \,+\, 2\, \sqrt{2\epsilon}
\int_{0}^{t}\left\langle
\mathbb{H}_0 \boldsymbol{y}_{\epsilon}(s),\,dW_{s}\right\rangle\;.
\end{equation*}
Thus, as $\epsilon\, t_\epsilon \to 0$,
for $\epsilon$ small enough, by Markov inequality,
\begin{equation*}
\mb P_{\boldsymbol{x}}(\mathcal{A}_{\epsilon}^{c})\,
\leq\, \mb P_{\boldsymbol{x}}
\Big[\, \sup_{t\le t_\epsilon} \int_{0}^{t}
\left\langle
\mathbb{H}_0 \, \boldsymbol{y}_{\epsilon}(s),\,dW_{s}\right\rangle
> \frac{r^2_1}{\sqrt{\epsilon}}\, \Big]
\,\le\, C_0 \, \epsilon^{8}\, \mb E_{\boldsymbol{x}}\Big[\,
\sup_{t\le t_\epsilon} \Big|\, \int_{0}^{t} \left\langle
\mathbb{H}_0\, \boldsymbol{y}_{\epsilon}(s),\,dW_{s}\right\rangle
\, \Big|^{16} \,\Big]\;.
\end{equation*}
By the Burkholder-Davis-Gundy and H\"older inequalities, the
right-hand side is bounded by
\begin{equation*}
C_0\, \epsilon^{8}\,
\mb E_{\boldsymbol{x}}\Big[\, \Big(\,
\int_{0}^{t_\epsilon} \left|\, \mathbb{H}_0 \, \boldsymbol{y}_{\epsilon}(s)
\, \right|^{2} \, \mathrm{d}s\, \Big)^{8} \Big]
\,\le\,
C_0\, \epsilon^{8}\, t^7_\epsilon\,
\mb E_{\boldsymbol{x}}\Big[\,
\int_{0}^{t_\epsilon} \left|\, \mathbb{H}_0\, \boldsymbol{y}_{\epsilon}(s)
\, \right|^{16} \, \mathrm{d}s\,  \Big] \;.
\end{equation*}
Hence, by Lemma \ref{lem:mom_est},
\begin{equation*}
\sup_{\bs x\in \mc D_{r_4}}\,
\mb P_{\boldsymbol{x}}(\mathcal{A}_{\epsilon}^{c})
\,\le\, C_0\,  \epsilon^{8}\, t^8_\epsilon
\,=\, C_0\, \epsilon^{8(1-\theta)}\;.
\end{equation*}
As $\theta<1/3$, the first assertion of (\ref{eq:obj1}) holds.

We turn to the second assertion. By definition, there exists a
positive constant $c_0$ such that
\begin{equation*}
\mathcal{B}_{\epsilon}^{c} \,=\,
\Big\{ \, \sup_{s\in[0, t_\epsilon]}
\left\langle \boldsymbol{\xi}(s),\,
\mathbb{K}\boldsymbol{\xi}(s)\right\rangle
\,\ge\, \frac{c_0}{\epsilon\, \sqrt{t_\epsilon}}\, \Big\} \;.
\end{equation*}
By the Markov inequality and Lemma \ref{lem:mom_Yt_sup}
\begin{equation*}
\sup_{\bs x\in \mc D_{r_4}}\, \mb P_{\boldsymbol{x}}(\mathcal{B}_{\epsilon}^{c})
\,\le\, C_0 \, \epsilon^8 \, t^4_\epsilon \,
\sup_{\bs x\in \mc D_{r_4}}\,  \mb E_{\boldsymbol{x}}\Big[\,
\sup_{s\in[0,t_\epsilon]} \left\langle \boldsymbol{\xi}(s),\,
\mathbb{K}\boldsymbol{\xi}(s)\right\rangle ^{8}\, \Big]
\,\le\,  C_0\, \epsilon^{8-12\theta}\;.
\end{equation*}
This proves the second assertion in (\ref{eq:obj1}) since $\theta<1/3$.
\end{proof}

\section{Local ergodicity}
\label{sec2}

Fix $\lambda>0$, $\bs{g}\colon \mc M_0 \rightarrow\mathbb{R}$, and
recall that we denote by
$\phi_{\epsilon} = \phi_{\epsilon}^{\lambda, \bs{g}}$ the unique
solution of the resolvent equation \eqref{e_res}.  The main result of
this section states that the solution $\phi_{\epsilon}$ is
asymptotically constant on each well $\mc E(\bs m)$.

\begin{thm}
\label{p_flat}
Fix $\lambda>0$ and $\bs{g}\colon \mc M_0 \rightarrow\mathbb{R}$.
For all $\mb m \in \mc M_0$,
\begin{equation*}
\lim_{\epsilon\rightarrow0}\, \sup_{x, y\in \mc E(\mb m)}
\vert\, \phi_{\epsilon} (x) -  \phi_{\epsilon} (y) \,
\vert \;=\; 0\;.
\end{equation*}
\end{thm}

Recall from \eqref{sde} that we represent by $\bs x_\epsilon (\cdot)$
the diffusion process induced by the generator $\mc L_\epsilon$.  The
proof of Theorem \ref{p_flat} is based on mixing properties of
$\bs x_\epsilon (\cdot)$ obtained in \cite{fw98, BJ,
LeeRamSeo}. Denote by
$\color{blue} \mathbb{P}_{\boldsymbol{z}}^{\epsilon}$,
$\bs z\in \bb R^d$, the law of $\bm{x}_{\epsilon}(\cdot)$ starting
from $\boldsymbol{z}$. Expectation with respect to
$\mathbb{P}_{\boldsymbol{z}}^{\epsilon}$, is represented by
$\color{blue} \mathbb{E}_{\boldsymbol{z}}^{\epsilon}$.

We start with elementary facts.  By equation (1.3) in \cite{BEGK},
conditions \eqref{26} guarantees that the partition function
$Z_\epsilon$, defined by
\begin{equation}
\label{e: def_Zeps}
{\color{blue} Z_{\epsilon}} \;:=\;
\int_{\mathbb{R}^{d}}\,e^{-U(\bm{x})/\epsilon}\,
d\bm{x}
\end{equation}
is finite. In particular, the Gibbs measure
\begin{equation*}
\mu_{\epsilon}(d\boldsymbol{x}) \;:=\; Z_{\epsilon}^{-1}
\,e^{-U(\bm{x})/\epsilon}\,d\boldsymbol{x}
\;:=\;  \mu_{\epsilon}(\bm{x}) \,d\boldsymbol{x}
\end{equation*}
is well defined.  Moreover, by Theorem 2.2 and 2.3 in \cite{LS-22},
the diffusion $\bs x_\epsilon (\cdot)$ is positive recurrent and
$\mu_\epsilon$ is its unique invariant measure. On the other hand, as
we assumed that $\min_{\bs x\in\bb R^d} U(\bs x)=0$, by
\cite[Proposition 3.2]{LS-22b} or a straightforward computation, if
$\color{blue} \mc M_\star$ representes the set of absolute minima of
$U$,
\begin{equation}
\label{32}
Z_{\epsilon}  \,=\,[\,1+o_{\epsilon}(1)\,]
\,(2\pi\epsilon)^{d/2}\,\nu_{\star}\;, \quad
\text{where}\quad
{\color{blue} \nu_{\star}} \;=\; \sum_{\bs m\in \mc M_\star}
\frac{1}
{\sqrt{\det\nabla^{2}U(\boldsymbol{\bs m})}}  \;,
\end{equation}
and, for a local minimum $\bs m\in \mc M_0$,
\begin{equation}
\label{55}
\mu_\epsilon (\mc E(\bs m)) \, e^{U(\bs m)/\epsilon}
\;=\; [\,1+o_{\epsilon}(1)\,] \, \frac{\nu(\bs m)}{\nu_{\star}}
\;\cdot
\end{equation}
In this formula, and throughout the article,
$\color{blue} o_{\epsilon}(1)$ represents a remainder which vanishes
as $\epsilon\to0$, and $\nu(\bs m)$ has been introduced in
\eqref{eq:nu}.

Denote
by $\tau_{\mathcal{A}}$, $\mathcal{A} \subset \mathbb{R}^{d}$, the
hitting time of the set $\mathcal{A}$:
\begin{equation}
\label{41}
{\color{blue} \tau_{\mathcal{A}}}
\;:=\; \inf\{\, t\ge0 : \bs x_\epsilon (t) \in \mc
A\,\} \;.
\end{equation}
Recall from \eqref{30} the definition of $\mc W^r(\bs m)$.
Conditions (b) and (c) in the definition of $\mc E(\bs m)$ guarantee
that the hypotheses of Theorem 6.2 in Chapter 6 of \cite{fw98} are
fulfilled. This results asserts that

\begin{prop}
\label{p_FW}
Fix $h<H$, and denote by $\mc A$, $\mc B$ a connected component of the
set $\{\bs x: U(\bs x) < h\}$, $\{\bs x: U(\bs x) < H\}$,
respectively. Assume that $\mc A \subset \mc B$. Suppose that all
critical points $\bs c$ of $U$ in $\mc A$ are such that
$U(\bs c) \le h_0$ for some $h_0 <h$. Then, for all $\eta>0$,
\begin{equation}
\label{60}
\limsup_{\epsilon\to 0} \sup_{\boldsymbol{x}\in \mc A}
\,\mathbb{P}_{\boldsymbol{x}}^{\epsilon}
\left[\,\tau_{\partial \mc B}
<e^{(H-h_0-\eta)/\epsilon}\,\right] \;=\; 0\;.
\end{equation}
In particular, for all $\bs m\in\mc M_0$, $\eta>0$,
\begin{equation*}
\limsup_{\epsilon\to 0} \sup_{\boldsymbol{x}\in \mathcal{E} (\bs m)}
\,\mathbb{P}_{\boldsymbol{x}}^{\epsilon}
\left[\,\tau_{\partial\mathcal{W}^{2r_0}(\bs m)}
<e^{(r_0-\eta)/\epsilon}\,\right] \;=\; 0\;.
\end{equation*}
\end{prop}

The estimate in \cite[Theorem 6.6.2]{fw98} is uniform over initial
points $\bs z$ belonging to neighborhoods of a critical points. We
claim that it holds uniformly over initial points $\bs x\in\mc
A$. Indeed, by \cite[Theorem 2.1.2]{fw98}, since the set $\mc A$ is
bounded, if we denote by $\mc N$ the union of neighborhoods of all
critical points of $U$ in $\mc A$, there exist $T_0<\infty$, such that
\begin{equation}
\label{61}
\liminf_{\epsilon\to 0} \inf_{\boldsymbol{x}\in \mc A}
\,\mathbb{P}_{\boldsymbol{x}}^{\epsilon}
\left[\,\tau_{\partial \mc N} < T_0 \,,\,
\tau_{\partial \mc N} < \tau_{\partial \mc B}
\,\right] \;=\; 1\;.
\end{equation}
Assertion \eqref{60} follows from \eqref{61}, the strong Markov
property and \cite[Theorem 6.6.2]{fw98}. Moreover, we could replace
$h_0$ by the minimal value of $U$ on $\mc A$, but that will not be
needed below.

\subsection*{Mixing times}
Fix $\bs m\in \mc M_0$. All constants, functions, processes which
appear in this subsection depend on $\bs m$, but this dependence is
omitted in the notation.

Let $\bs b_0 \colon \bb R^d \to \bb R^d$ be the field of class $C^1$
defined in Appendix \ref{sec-ap4}. By \eqref{eq:vecb} and condition
(d) in the definition of $r_0$, $\bs b_0(\bs x) = \bs b (\bs x) $ for
$\bs x\in \mc W^{3r_0}(\bs m)$.  By Proposition \ref{pap4}, the vector
field $\bs b_0$ satisfies the hypotheses of Section \ref{sec-ap3}.

Denote by $\color{blue} \boldsymbol{x}^{F}_{\epsilon}(\cdot)$ the
diffusion process \eqref{sde} with the vector field $\bs b_0$
replacing $\bs b$. Let
$\color{blue} \mathbb{P}_{\boldsymbol{z}}^{\epsilon,\,F}$,
$\bs z\in \bb R^d$, be the law of $\bm{x}^{F}_{\epsilon}(\cdot)$
starting from $\boldsymbol{z}$, and
$\color{blue} p_{\epsilon}^{F}(\bm{z},\,\cdot \, ;t)$ its transition
kernel:
\begin{equation*}
p_{\epsilon}^{F}(\bm{z}, B ;t)
\;=\; \mathbb{P}_{\boldsymbol{z}}^{\epsilon,\,F}
\left[\,\bm{x}^{F}_{\epsilon}(t)\in\mathcal{B}\,\right]\;,
\quad \boldsymbol{z}\in \bb R^d  \;,\;\;
\mathcal{B}\subseteq \bb R^d \;.
\end{equation*}
Denote by $\color{blue} \mu_{\epsilon}^{F}$ the stationary state of
the process $\bm{x}^{F}_{\epsilon}(\cdot)$.

\begin{proof}[Proof of Theorem \ref{p_flat}]
Fix $\bm{m}\in\mathcal{M}_{0}$. Let
\begin{equation}
\label{35}
{\color{blue} {\bs f}_{\epsilon}(\bm{m})}
\,:=\,
\int_{\bb R^d}\phi_{\epsilon} (\bm{x})\,
\mu_{\epsilon}^{F} (d\boldsymbol{x})\ .
\end{equation}
It is enough to prove that for all $\bm{m}\in\mathcal{M}_{0}$,
\begin{equation*}
\lim_{\epsilon\to0}\,\sup_{\bm{x}\in\mathcal{E}(\bm{m})}\,
|\, \phi_{\epsilon}(\bm{x})-\bs{f}_{\epsilon}(\bm{m})\,|
\,=\, 0 \;.
\end{equation*}

Recall from \eqref{e_res} the definition of the function
$G\colon \bb R^d \to \bb R$.  By the stochastic representation of the
resolvent equation,
\begin{equation}
\label{exp_phi-1}
\phi_{\epsilon}(\bm{x})
\;=\; \mathbb{E}_{\bm{x}}^{\epsilon}\Big[\,
\int_{0}^{\infty}e^{-\lambda s}\,
G(\bm{x}_{\epsilon}(\theta_{\epsilon}^{(1)}s))\,ds\, \Big]\ .
\end{equation}

Fix $0< a <1/3$, $0<\eta < r_0/2$, and let
$\varrho_\epsilon = \epsilon^{-a}$. By definition $\theta_{\epsilon}^{(1)}$,
\begin{equation}
\label{58}
\varrho_{\epsilon} \prec  e^{(r_0-\eta)/\epsilon}
\prec \theta_{\epsilon}^{(1)} \;.
\end{equation}
Since $\varrho_{\epsilon}\prec\theta_{\epsilon}^{(1)}$ and $G$ is
bounded,
\begin{equation*}
\phi_{\epsilon}(\bm{x}) \;=\;
\mathbb{E}_{\bm{x}}^{\epsilon}
\Big[\, \int_{\varrho_{\epsilon}/\theta_{\epsilon}^{(1)}}^{\infty}
e^{-\lambda s}\,G(\bm{x}_{\epsilon}(\theta_{\epsilon}^{(1)}s))\,ds
\,\Big] \;+\; R_{\epsilon} (\bs x) \;,
\end{equation*}
where, here and below, $R_{\epsilon} (\bs x)$ represents an error whose
value may change from line to line and such that
\begin{equation*}
\limsup_{\epsilon\to 0} \sup_{\bs y \in \mc E(\bs m)}
|\, R_{\epsilon} (\bs y)\,| \;=\; 0\;.
\end{equation*}
By the Markov property,
\begin{equation*}
\begin{aligned}
\phi_{\epsilon}(\bm{x})
\;&=\; [ 1 + R_{\epsilon} (\bs x)] \, \mathbb{E}_{\bm{x}}^{\epsilon}
\Big[\, \mathbb{E}_{\bm{x}_{\epsilon}(\varrho_{\epsilon})}
\Big[\, \int_{0}^{\infty}e^{-\lambda s}\,
G(\bm{x}_{\epsilon}(\theta_{\epsilon}^{(1)}s))\,ds\, \Big]\,\Big]
\;+\; R_{\epsilon} (\bs x)
\\
& =\;
\mathbb{E}_{\bm{x}}^{\epsilon}
\big[\, \phi_{\epsilon} (\bm{x}_{\epsilon}(\varrho_{\epsilon})) \,\big]
\;+\; R_{\epsilon} (\bs x)
\end{aligned}
\end{equation*}
because $G$ is bounded.  As
$\varrho_{\epsilon}\prec e^{(r_0-\eta)/\epsilon}$, by Proposition
\ref{p_FW} and since $\phi_{\epsilon}$ is uniformly bounded by
$(1/\lambda)\, \Vert \bs g\Vert_\infty$,
\begin{equation*}
\mathbb{E}_{\bm{x}}^{\epsilon}
\left[\phi_{\epsilon} (\bm{x}_{\epsilon}(\varrho_{\epsilon}))\right]
\;=\; \mathbb{E}_{\bm{x}}^{\epsilon}\left[
\phi_{\epsilon}(\bm{x}_{\epsilon}(\varrho_{\epsilon})) \,
{\bf 1}\{\varrho_{\epsilon}<\tau_{\left(\mathcal{W}^{2r_0}(\bm{m})\right)^{c}}\}
\right] \,+\, R_{\epsilon} (\bs x)
\end{equation*}
Recall that $\bs b$ and $\bs b_0$ coincide on $\mc W^{3r_0} (\bs m)$.
By coupling the diffusions $\bs x_\epsilon (\cdot)$,
$\bs x^F_\epsilon (\cdot)$, and in view of Proposition \ref{p_FW}, the
previous expectation is equal to
\begin{align*}
\mathbb{E}_{\bm{x}}^{\epsilon, F }\left[\phi_{\epsilon}
(\bm{x}^{F}_{\epsilon} (\varrho_{\epsilon}))\,
{\bf 1}\{\varrho_{\epsilon}<\tau_{\left(\mathcal{W}^{2r_0}(\bm{m})\right)^{c}}\}
\right]  \;
=\, \mathbb{E}_{\bm{x}}^{\epsilon , F}
\left[\phi_{\epsilon}(\bm{x}^F_{\epsilon}
(\varrho_{\epsilon}))\right]  \,+\, R_{\epsilon} (\bs x) \ .
\end{align*}
Mind that we changed the measure. By condition (e) in the definition
of $r_0$,
$\mc E(\bs m) \subset \mc W^{2r_0} (\bs m) \subset \mc D_{r_4} (\bs
m)$. Hence, by Theorem \ref{t_main2} and since $\phi_{\epsilon}$ is
uniformly bounded,
\begin{equation*}
\mathbb{E}_{\bm{x}}^{\epsilon, F}\left[\phi_{\epsilon}
(\bm{x}^F_{\epsilon} (\varrho_{\epsilon}))\right]
\;=\; \int_{\bb R^d}\phi_{\epsilon} (\bm{y})\,
p_{\epsilon}^{F}(\boldsymbol{x},\, d \boldsymbol{y};
\varrho_{\epsilon})
\;= \; \int_{\bb R^d}\phi_{\epsilon}(\bm{y})\,
\mu_{\epsilon}^{F} (d\bm{y}) \,+\, R_{\epsilon} (\bs x)    \;.
\end{equation*}
As the right-hand side is equal to
${\bs f}_{\epsilon}(\bm{m}) \,+\, R_{\epsilon} (\bs x)$, the theorem
is proved.
\end{proof}

Recall the definition of the sequence $\varrho_\epsilon$ introduced in
\eqref{58}.  The proof of Theorem \ref{p_flat} yields the following
result.

\begin{lem}
\label{l14}
Fix $\bm{m}\in\mathcal{M}_{0}$, $b>0$. Then, for all
$\mathcal{A}\subset \bb R^d$.
\begin{equation*}
\limsup_{\epsilon\rightarrow0}\sup_{t\in[2b,\,4b]}
\sup_{\boldsymbol{x}\in\mathcal E(\bm{m})}
\left|\, \bb P_{\boldsymbol{x}}^{\epsilon}
\big[\,\bm{x}_{\epsilon}(t\theta_{\epsilon}^{(1)})\in\mathcal{A}\,\big] \,-\,
\bb P_{\mu_{\epsilon}^{F}}^{\epsilon}
\big[\,\bm{x}_{\epsilon}(t\theta_{\epsilon}^{(1)}
- \varrho_\epsilon)\in\mathcal{A}\,\big]\, \right|=0 \;.
\end{equation*}
\end{lem}

Denote by $\color{blue} \bs x^{\rm R}_\epsilon (\cdot)$ the diffusion
$\bs x_\epsilon (\cdot)$ reflected at the boundary of
$\mc W^{2r_{0}}(\bm{m})$. Note that we omitted the dependence of
$\bs x^{\rm R}_\epsilon (\cdot)$ on $\bs m$. Denote by
$\color{blue} \mu^{\rm R}_\epsilon$ the measure $\mu_\epsilon$
conditioned to $\mc W^{2r_{0}}(\bm{m})$, which is the invariant
measure of the diffusion $\bs x^{\rm R}_\epsilon (\cdot)$. Let finally
$\color{blue} \mathbb{P}^{\epsilon,\rm R}_{\boldsymbol{z}}$,
$\bs z\in \mc W^{2r_{0}}(\bm{m})$, be the law of
$\bm{x}^{\rm R}_{\epsilon}(\cdot)$ starting from $\boldsymbol{z}$.

Recall that we denote by $d_{\rm TV}(\mu, \nu)$ the total
variation distance between two probability measures $\nu$, $\mu$
defined on $\bb R^d$. Let $\mu^{\mc E(\bs m)}_\epsilon$ be the measure
$\mu_\epsilon$ conditioned to $\mc E(\bs m)$.  We claim that
\begin{equation}
\label{59}
\limsup_{\epsilon\rightarrow0} d_{\rm TV}(\mu^{\rm R}_\epsilon ,
\mu^{F}_\epsilon) \,=\, 0\;.
\end{equation}

Indeed, fix $\mc A \subset \bb R^d$, $\bs x\in \mc E(\bs m)$, and the
sequence $\varrho_\epsilon$ introduced in \eqref{58}. By stationarity
and Theorem \ref{t_main2},
\begin{equation*}
\mu^{F}_\epsilon (\mc A) \,=\,
\bb P_{\mu_{\epsilon}^{F}}^{\epsilon, F}
\big[\,\bm{x}^F_{\epsilon}(\varrho_\epsilon)\in\mathcal{A}\,\big]
\,=\,
\bb P_{\bs x}^{\epsilon, F}
\big[\,\bm{x}^F_{\epsilon}(\varrho_\epsilon)\in\mathcal{A}\,\big]
\,+\, R_\epsilon (\bs x)  \;,
\end{equation*}
where adopted the convention established in the proof of Theorem
\ref{p_flat} for the remainder $R_\epsilon (\bs x) $.

As in the proof of Theorem \ref{p_flat}, introduce the event
$\{\tau_{\partial \mathcal{W}^{2r_0}(\bm{m}) } \le \varrho_{\epsilon}\}$
and its complement. On the event
$\{\tau_{\partial \mathcal{W}^{2r_0}(\bm{m})} > \varrho_{\epsilon}\}$
we may replace the set $\mc A$ by $\mc A \cap \mc W^{2r_0}(\bm{m})$,
and couple the process $\bm{x}^F_{\epsilon}(\cdot)$, $\bm{x}^{\rm
R}_{\epsilon} (\cdot)$ up to time $\varrho_{\epsilon}$. Therefore, the
probability on the right-hand side of the previous displayed equation
is equal to
\begin{equation*}
\bb P_{\bs x}^{\epsilon, \rm R}
\big[\,\bm{x}^R_{\epsilon}(\varrho_\epsilon)\in\mathcal{A}
\cap \mc W^{2r_0}(\bm{m})\,\big]
\,+\, R^{(2)}_\epsilon \;=\;
\bb P_{\bs x}^{\epsilon, \rm R}
\big[\,\bm{x}^R_{\epsilon}(\varrho_\epsilon)\in\mathcal{A} \,\big]
\,+\, R^{(2)}_\epsilon \;,
\end{equation*}
where
$|R^{(2)}_\epsilon| \le 2 \sup_{\bs z\in \mc E(\bs m)} \bb
P_{\bs z}^{\epsilon} [\, \tau_{\partial \mathcal{W}^{2r_0}(\bm{m}) } \le
\varrho_{\epsilon} \,]$. Here, we removed the set
$\mc W^{2r_0}(\bm{m})$ because $\bs x^{\rm R}$ takes value on this
set. By Proposition \ref{p_FW}, $R^{(2)}_\epsilon \to 0$.

Since the previous estimates are uniform over $\bs x\in \mc E(\bs m)$,
we may average the probability appearing on the right-hand side of the
previous displayed equation with respect to the measure $\mu^{\mc
E(\bs m)}_\epsilon$ to get that
\begin{equation*}
\mu^{F}_\epsilon (\mc A) \,=\,
\bb P_{\mu^{\mc E(\bs m)}_\epsilon}^{\epsilon, \rm R}
\big[\,\bm{x}^R_{\epsilon}(\varrho_\epsilon)\in\mathcal{A} \,\big]
\,+\, o_\epsilon (1) \;.
\end{equation*}
Rewrite the previous probability as
\begin{equation*}
\bb P_{\mu^{\mc E(\bs m)}_\epsilon}^{\epsilon, \rm R}
\big[\,\bm{x}^R_{\epsilon}(\varrho_\epsilon)\in\mathcal{A} \,\big]
\;=\; \frac{1}{\mu_\epsilon (\mc E(\bs m))}  \,
\int_{\mc E(\bs m)}
\bb P_{\bs y}^{\epsilon, \rm R}
\big[\,\bm{x}^R_{\epsilon}(\varrho_\epsilon)\in\mathcal{A} \,\big]
\mu_\epsilon (d\bs y) \;.
\end{equation*}
The measure $\mu^{\mc E(\bs m)}_\epsilon$ is also the measure
$\mu^{\rm R}_\epsilon$ conditioned to $\mc E(\bm{m})$. Since
$\mu_\epsilon (\mc W^{2r_0}(\bm{m}) \setminus \mc E(\bm{m})) /
\mu_\epsilon (\mc E(\bm{m})) \to 0$, the previous expression is equal
to
\begin{equation*}
\bb P_{\mu^{\rm R}_\epsilon}^{\epsilon, \rm R}
\big[\,\bm{x}^R_{\epsilon}(\varrho_\epsilon)\in\mathcal{A} \,\big]
\,+\, R^{(3)}_\epsilon\;,
\end{equation*}
where $R^{(3)}_\epsilon \to 0$. Since $\mu^{\rm R}_\epsilon$ is the
stationary state, the previous probability is equal to
$\mu^{\rm R}_\epsilon (\mc A)$.

Putting together the previous estimates yields that
\begin{equation*}
\limsup_{\epsilon \to 0} \sup_{\mc A \subset \bb R^d} \,
\big|\, \mu^{F}_\epsilon (\mc A) \,-\, \mu^{\rm R}_\epsilon (\mc
A)\,\big|\; = \; 0  \;.
\end{equation*}
as claimed in \eqref{59}.

Next result follows from Lemma \ref{l14} and \eqref{59}, Note that the
measure $\mu_{\epsilon}^{F}$ has been replaced by
$\mu_{\epsilon}^{\rm R}$.

\begin{corollary}
\label{l15}
Fix $\bm{m}\in\mathcal{M}_{0}$, $b>0$,
$\mathcal{A}\subset \bb R^d$. Then,
\begin{equation*}
\limsup_{\epsilon\rightarrow0}\sup_{t\in[2b,\,4b]}
\sup_{\boldsymbol{x}\in\mathcal E(\bm{m})}
\left|\, \bb P_{\boldsymbol{x}}^{\epsilon}
\big[\,\bm{x}_{\epsilon}(t\theta_{\epsilon}^{(1)})\in\mathcal{A}\,\big] \,-\,
\bb P_{\mu_{\epsilon}^{\rm R}}^{\epsilon}
\big[\,\bm{x}_{\epsilon}(t\theta_{\epsilon}^{(1)}
- \varrho_\epsilon)\in\mathcal{A}\,\big]\, \right|=0 \;.
\end{equation*}
\end{corollary}

\section{Exiting neighborhoods of unstable critical points}
\label{sec5}

The main result of this section, Proposition \ref{p:Kifer}, asserts
that the time necessary for the diffusion $\bs x_\epsilon(\cdot)$ to
leave neighborhoods of unstable critical points is bounded by
$\epsilon^{-1}$. It also characterizes the exiting  sets.

Recall that $\mathcal{C}_{0}$ denotes the set of critical points
of $U$ and set
\begin{equation*}
{\color{blue} \mathcal{Y}_{0}} \,:=\,
\mathcal{C}_{0}\setminus\mathcal{M}_{0}\;,
\end{equation*}
so that $\mathcal{Y}_{0}$ stands for the collection of critical points
of $U$ with index larger than $0$.

By \cite[Theorem 2.1]{LS-22}, $\mathcal{M}_{0}$ and
$\mathcal{Y}_{0}$ are the set of stable and unstable equilibria of the
dynamical system \eqref{31}, respectively. Let
$\color{blue}
\mathbb{H}^{\boldsymbol{c}}=(\nabla^{2}U)(\boldsymbol{c})$,
$\color{blue}
\mathbb{L}^{\boldsymbol{c}}=(\nabla\cdot\boldsymbol{\ell})(\boldsymbol{c})$,
$\boldsymbol{c}\in\mathcal{C}_{0}$, so that
$\mathbb{H}^{\boldsymbol{c}}+\mathbb{L}^{\boldsymbol{c}}$ denotes the
Jacobian of the drift $\bs b$ at the critical point $\bs c$.  Next
result asserts that critical points in $\mathcal{Y}_{0}$ are
hyperbolic.

\begin{lem}
\label{lem:hyper}
Fix $\bm{c}\in\mathcal{Y}_{0}$. Then, the matrix
$\mathbb{H}^{\bm{c}}+\mathbb{L}^{\bm{c}}$ is invertible and does not
have a pure imaginary eigenvalue.
\end{lem}

\begin{proof}
Suppose, by contradiction, that $ai$, $a\in\mathbb{R}$, is an
eigenvalue of $\mathbb{H}^{\bm{c}}+\mathbb{L}^{\bm{c}}$. Denote by
$\bm{v}$ the unit eigenvector corresponding to $ai$ so that
$(\mathbb{H}^{\bm{c}}+\mathbb{L}^{\bm{c}})\bm{v}=ai\bm{v}$.
Thus, if $\color{blue} \bb A^\dagger$ represents the transpose of the
matrix $\bb A$,
\begin{align*}
ai\bm{v}\cdot ai\bm{v} &
=\bm{v}\cdot(\mathbb{H}^{\bm{c}}
+\mathbb{L}^{\bm{c}})^{\dagger}(\mathbb{H}^{\bm{c}}
+\mathbb{L}^{\bm{c}})\bm{v}
\\
& =\bm{v}\cdot\left\{ (\mathbb{H}^{\bm{c}})^{\dagger}
\mathbb{H}^{\bm{c}}+(\mathbb{H}^{\bm{c}})^{\dagger}
\mathbb{L}^{\bm{c}}+(\mathbb{L}^{\bm{c}})^{\dagger}
\mathbb{H}^{\bm{c}}+(\mathbb{L}^{\bm{c}})^{\dagger}
\mathbb{L}^{\bm{c}})\right\} \bm{v}\;.
\end{align*}
By \cite[Lemma 4.5]{LS-22}, the matrix $\bb H^{\bm{c}} \bb L^{\bm{c}}$
is skew-symmetric, so that
\begin{align*}
-\, a^{2}\|\bm{v}\|^{2} & =\|\mathbb{H}^{\bm{c}}\bm{v}\|^{2}
+\|\mathbb{L}^{\bm{c}}\bm{v}\|^{2}\;,
\end{align*}
which is a contradiction if $a\neq0$. If $a=0$,
$\mathbb{H}^{\bm{c}}\bm{v}=0$ which implies that $\boldsymbol{v}=0$
since $\mathbb{H}^{\boldsymbol{c}}$ is invertible. This is also a
contradiction to the fact that $\boldsymbol{v}$ is a unit vector.
\end{proof}

\subsection*{The Hartman-Grobman theorem}

Fix from now on a critical point $\bs c\in \mc Y_0$ of index $k\ge1$.
In this subsection, we use Hartman-Grobman theorem \cite[Theorem
1.47]{Chicone}, \cite[Section 2.8]{Perko}, to define a neighborhood of
$\bs c$.

Denote by $\color{blue} \upsilon_{\bs x} (t)$, $\bm{x} \in\bb R^d$,
$t\ge0$, the solution of the ODE \eqref{31} starting from $\bs x$, and
by
$\color{blue} \upsilon_{L, \bs x} (t) = \upsilon^{\bs c}_{L, \bs x}
(t)$ the solution of the linear ODE
\begin{equation}
\label{34}
\dot {\bs x} (t) \,=\, -\,
( \mathbb{H}^{\bm{c}}+ \mathbb{L}^{\bm{c}}  )\,
(\boldsymbol{x} (t) -\boldsymbol{c} )
\end{equation}
starting from $\bs x$. The Hartman-Grobman theorem, which can be
applied in view of Lemma \ref{lem:hyper}, reads as follows.

\begin{thm}
\label{thm:H-G}
Fix $\boldsymbol{c}\in\mathcal{Y}_{0}$. There exist open neighborhoods
$\mathcal{U}_{\boldsymbol{c}},\,\mathcal{U}^L_{\boldsymbol{c}}$ of
$\boldsymbol{c}$ and a homeomorphism
$\Xi \colon
\mathcal{U}_{\boldsymbol{c}}\rightarrow\mathcal{U}^L_{\boldsymbol{c}}$
such that $\Xi(\boldsymbol{c})=\boldsymbol{c}$ and
$\Xi ( \upsilon_{\bs x} (t) ) = \upsilon_{L,\Xi(\bs x)} (t)$ for all
$(\bs x, t)$ such that $ \upsilon_{\bs x} (t) \in \mc U_{\bs c}$. In
particular, $\boldsymbol{c}$ is the unique critical point of $U$ in
$\mathcal{U}_{\boldsymbol{c}}$.
\end{thm}

Denote by $\color{blue} \mc M_s = \mathcal{M}_s (\boldsymbol{c})$,
$\color{blue} \mc M_u = \mathcal{M}_u (\boldsymbol{c})$ the stable,
unstable manifold of $\boldsymbol{c}$ for the dynamical system
\eqref{31}, respectively. Hence, for all $\bs x \in \mc M_s$,
$\lim_{t\rightarrow \infty} \upsilon_{\bs x} (t) = \boldsymbol{c}$. In
contrast, for all $\bs y \in \mc M_u$ there exists a solution $\bs x(t)$,
$t\le 0$ of \eqref{31} such that
\begin{equation*}
\bs x(0) = \bs y \;, \quad \lim_{t\to - \infty} \bs x(t) = \bs c \;.
\end{equation*}
Let $\color{blue} \mathcal{M}_{L,s}$, $\color{blue} \mathcal{M}_{L,u}$
be the stable, unstable manifold of $\bs c$ for the linear ODE
\eqref{34}. By Theorem \ref{thm:H-G}, on the set $\mc U^L_{\bs c}$,
$\mathcal{M}_{L,s} = \Xi (\mc M_s)$,
$\mathcal{M}_{L,u} = \Xi (\mc M_u)$.

Choose $r_{1}>0$ so that
$\color{blue} B(\bs c,\,r_{1})\subset\mc U^L_{\bs c}$.  Let
$\color{blue} \widehat {\mc N} = \widehat {\mc N} (\bm{c}) :=
\Xi^{-1}(B(\bs c ,\,r_{1}))$. For each
$\boldsymbol{y}\in\widehat{\mc N} \setminus\mc M_s$, let
$t(\boldsymbol{y})=t_{\boldsymbol{c}}(\boldsymbol{y})$ be the exit
time from $\widehat{\mc N}$:
\begin{equation}
\label{eq:overlinet}
{\color{blue}   t(\boldsymbol{y})} \, :=\, \inf \{t\ge0:
\upsilon_{\bs y} (t)  \not\in \widehat{\mc N} \} \;.
\end{equation}
Clearly, $t(\boldsymbol{y}) = t_L (\Xi(\boldsymbol{y}))$ if
$t_L (\bs z)$ represents the exit time from $B(\bs c ,\,r_{1})$ for
the linear ODE \eqref{34} starting from $\bs z$.  Denote by
$\boldsymbol{e}(\boldsymbol{y}) =
\boldsymbol{e}_{\boldsymbol{c}}(\boldsymbol{y})$ the exit location of
the dynamical systems \eqref{31} from the set $\widehat{\mc N}$:
${\color{blue} \boldsymbol{e}(\boldsymbol{y})} \,:=\; \upsilon_{\bs y}
(t(\boldsymbol{y})) $.  Here again,
\begin{equation}
\label{eq:relp}
\Xi(\bm{e}(\bm{y}))
\;=\;  \bs e_L (\Xi(\boldsymbol{y}))
\end{equation}
provided $\color{blue} \bs e_L (\bs z)$ stands for the exit location
from the set $B(\bs c,\,r_{1})$ of the linear dynamical systems
\eqref{34} starting from $\bs z$.

Let
$\mc J^{a}_L = \mc J^{a}_L(\boldsymbol{c})$ be the elements of
$\partial B (\boldsymbol{c},\,r_{1})$ at distance less than $a$ from
$\mc M_{L,u} \cap\partial\mathcal{B}(\boldsymbol{c},\,r_{1})$:
\begin{equation*}
{\color{blue} \mc J^a_L } \,:=\, \big\{
\boldsymbol{x}\in\partial B (\boldsymbol{c},\,r_{1}):
\exists\, \boldsymbol{y}\in \mc M_{L,u}
\cap\partial\mathcal{B}(\boldsymbol{c},\,r_{1})
\text{ such that }\Vert\boldsymbol{x}-\boldsymbol{y}\Vert<a\big\} \;.
\end{equation*}
Next result is an assertion about the linear ODE \eqref{34}.  Its
proof is presented in Appendix \ref{sec:ODE}.

\begin{lem}
\label{lem_esclin}
Fix $\boldsymbol{c}\in\mathcal{Y}_{0}$ and $a>0$. Then, there exists
$0< r(a)<r_1$ such that $e_L(\bs z) \in \mc J^a_L$ for all
$\bs z\in \mathcal{B}(\bm{c},\,r(a)) \setminus \mc M_{L,s} $.
\end{lem}

We turn to the construction of a second neighborhood
$\mc N \subset \widehat {\mc N}$.  Since
$\nabla U\cdot\boldsymbol{\ell}\equiv0$,
$(d/dt) U(\upsilon_{\bs x}(t)) = -\,| \nabla U (\upsilon_{\bs x}(t))
|^{2} < 0$ for all $\bs x \notin \mathcal{C}_{0}$, $t>0$. Therefore,
if $\bs x$ is not a critical point, $U (\upsilon_{\bs x}(t))$ is
strictly decreasing in $t$, and there exists
$\eta_{0}=\eta_{0}(r_1)>0$ such that
\begin{equation}
\label{eq:Ubdr}
\max_{\boldsymbol{x}\in \mc M_{u}
\cap \partial\widehat{\mc N} }
U(\boldsymbol{x})<U(\boldsymbol{c})-3\eta_{0}\;.
\end{equation}
Take $\eta_{0}$ small enough so that there is no critical point
$\boldsymbol{c}' \in \mc C_0$ such that
\begin{equation}
\label{eq:nocr}
U(\boldsymbol{c}')\in[U(\boldsymbol{c})-\eta_{0},\,
U(\boldsymbol{c})) \;.
\end{equation}

\begin{lem}
\label{lem_esc}
For all $\boldsymbol{c}\in\mathcal{Y}_{0}$, there exists
$r_{2}=r_{2}(\boldsymbol{c})>0$ such that,
\begin{equation*}
\sup_{\boldsymbol{y}\in\Xi^{-1}(B(\bm{c},\,r_{2}))
\setminus\mc M_s}
U(\boldsymbol{e}(\boldsymbol{y}))\le U(\boldsymbol{c})-2\eta_{0}\;.
\end{equation*}
\end{lem}

\begin{proof}
For $a>0$, let
\begin{equation*}
\mathcal{J}^{a} \,=\,
\big\{ \, \boldsymbol{x}\in\partial\widehat{\mathcal{N}} :
\exists\, \boldsymbol{y}\in\mc M_u \cap \partial\widehat{\mc N}
\text{ such that }\Vert\boldsymbol{x}-\boldsymbol{y}\Vert<a\, \big\} \;.
\end{equation*}
By \eqref{eq:Ubdr} and the fact that $|\nabla U|$ is bounded on
compact sets, there exists $a_{0}>0$ such that
\begin{equation}
\label{eq:esc1}
\sup_{\boldsymbol{x}\in\mathcal{J}^{a_{0}} }
U (\boldsymbol{x})\le U(\boldsymbol{c})-2\eta_{0}\;.
\end{equation}

Since $\Xi^{-1}\colon \mathcal{U}^L_{\bm{c}}\to\mathcal{U}_{\bm{c}}$ is
continuous, it is uniformly continuous on the compact set
$\overline {B(\bm{c},\,r_{1})}$.  Therefore, there exists $b_{0}>0$
such that $\|\Xi^{-1}(\bm{x})-\Xi^{-1}(\bm{y})\|\le a_{0}$ for all
$\boldsymbol{x},\,\boldsymbol{y}\in B(\bm{c},\,r_{1})$ satisfying
$\|\bm{x}-\bm{y}\|\le b_{0}$, Therefore,
\begin{equation}
\label{esc2}
\Xi^{-1}(\mathcal{J}_L^{b_{0}}) \subset \mathcal{J}^{a_{0}}\ .
\end{equation}

Let $r(b_0)>0$ be the positive constant whose existence is asserted in
Lemma \ref{lem_esclin}. Set $r_2 = r(b_0) \wedge r_1$. By Lemma
\ref{lem_esclin},
$\bs e_L (\Xi(\boldsymbol{y})) \in \mathcal{J}_L^{b_{0}}$ for
$\Xi(\bm{y})\in B(\bm{c},\,r_{2})\setminus\Xi(\mc M_s)$.  Therefore,
by \eqref{eq:relp} and \eqref{esc2}, for
$\bm{y}\in\Xi^{-1}(B(\bm{c},\,r_{2}))\setminus\mc M_s$,
\begin{equation*}
\bm{e}(\bm{y}) =
\Xi^{-1}(\bs e_L (\Xi(\boldsymbol{y}))) \in \mathcal{J}^{a_{0}} \;.
\end{equation*}
This along with \eqref{eq:esc1} imply that
\begin{equation*}
\sup_{\boldsymbol{y}\in\Xi^{-1}(B(\bm{c},\,r_{2}))
\setminus \mc M_s}
U(\boldsymbol{e}(\boldsymbol{y}))\le U(\boldsymbol{c})-2\eta_{0}\;,
\end{equation*}
which completes the proof of the lemma.
\end{proof}

\subsection*{Exit problem from $\widehat{\mc N}$}

Denote by $\color{blue} \mc N = \mc N(\bs c)$ the closure of the set
$\Xi^{-1}(B(\boldsymbol{c},\,r_{2}))$, where $r_{2}$ has been
introduced in Lemma \ref{lem_esc}.  As the set $\widehat{\mc N}$
contains an unstable equilibrium $\boldsymbol{c}$, the exit problem
from $\widehat{\mc N}$ does not follow from the Friedlin-Wentzell
theory, but has been investigated in \cite{Kif}.

\begin{prop}
\label{p:Kifer}
Fix $\bm{c}\in\mathcal{Y}_{0}$.  Then,
\begin{equation*}
\limsup_{\epsilon\rightarrow0}
\sup_{\bm{z}\in\mc N }
\mathbb{P}_{\bm{z}}^{\epsilon} \big[\, U(\bm{x}_{\epsilon}
(\tau_{\partial\widehat{\mc N}}))>U(\bm{c})-\eta_{0}\, \big]=0\;.
\end{equation*}
Moreover, for all $C>0$,
\begin{equation*}
\limsup_{\epsilon\rightarrow0}
\sup_{\bm{z}\in\mc N } \mathbb{P}_{\bm{z}}^{\epsilon}
\Big[\, \tau_{\partial\widehat{\mc N}}
>\frac{C}{\epsilon}\, \Big]=0\;.
\end{equation*}
\end{prop}

\begin{proof}
Since the set $\widehat{\mc N}$ contains only one unstable
equilibrium, the second assertion of the proposition follow from
\cite[Theorem 2.1]{Kif}, which presents an estimate for a fixed
starting point in the interior of $\widehat{\mc N}$. However, a
careful reading of the proof reveals that all estimates hold uniformly
on compact subsets of the interior of $\widehat{\mc N}$, such as
$\mc N$.

We turn to the first assertion of the proposition.
Let
$\mathcal{Q} \subset\partial\widehat{\mc N}$ be given by
\begin{equation*}
\mathcal{Q} = \{ \boldsymbol{e}(\boldsymbol{y}):
\boldsymbol{y}\in\mc N \setminus\mathcal{M}_s \}
\cup ( \mathcal{M}_u \cap \partial\widehat{\mc N} ) \;.
\end{equation*}
By \cite[Theorem 2.3]{Kif}, for any open neighborhood
$\mathcal{U}\subset\partial\widehat{\mc N}$ of $\mathcal{Q}$ in
$\partial\widehat{\mc N}$,
\begin{equation*}
\limsup_{\epsilon\rightarrow0}
\sup_{\bm{z}\in\mc N}
\mathbb{P}_{\bm{z}}^{\epsilon}
\big[\, \bm{x}_{\epsilon}(\tau_{\partial\widehat{\mc N}})
\notin\mathcal{U} \, \big] = 0\;.
\end{equation*}
Note that \cite[Theorem 2.3]{Kif} is stated for a fixed starting point
in the interior of $\widehat{\mc N}$, but as in the first part of the
proof, all estimates in the proof of this result hold unifomly on
compact subsets of the interior of $\widehat{\mc N}$.

By \eqref{eq:Ubdr} and Lemma \ref{lem_esc},
\begin{equation*}
\sup_{\boldsymbol{x}\in\mathcal{\mathcal{Q}} }
U(\boldsymbol{x})\le U(\boldsymbol{c})-2\eta_{0}\;.
\end{equation*}
To complete the proof, it remains to choose a neighborhood
$\mathcal{U}$ small enough so that
\begin{equation*}
\sup_{\boldsymbol{x}\in\mathcal{U}}
U(\boldsymbol{x})\le U(\boldsymbol{c})-\eta_{0}\;.
\end{equation*}

\end{proof}

\section{Hitting wells}
\label{sec6}

The main result of this section, Theorem \ref{t_hitting2} below,
asserts that starting from a compact set, the diffusion
$\bs x_\epsilon (\cdot)$ hits some well $\mc E(\bs m)$ in a time
bounded by $\epsilon^{-1}$.

Denote by $\mc E(\mc A)$, $\mathcal{A}\subset\mathbb{R}^{d}$, the
union of the wells in $\mc A$:
\begin{equation*}
\mathcal{E}(\mathcal{A})=
\bigcup_{\bm{m}\in\mathcal{M}_{0}\cap\mathcal{A}}\mathcal{E}(\bm{m})\;.
\end{equation*}
Let
\begin{equation*}
\Lambda_H = \{\boldsymbol{x}\in\mathbb{R}^{d}:U(\boldsymbol{x})\le H\}
\;, \quad H \in \bb R \;.
\end{equation*}

\begin{thm}
\label{t_hitting2}
Fix $H>\min_{\bm{x}\in\mathbb{R}^{d}}U(\bm{x})$.  Suppose that there
is no critical point $\boldsymbol{c}\in\mathcal{C}_{0}$ such that
$U(\boldsymbol{c})=H$. Then, for all $C>0$,
\begin{equation*}
\limsup_{\epsilon\rightarrow0}
\sup_{\bm{z}\in \Lambda_H}
\mathbb{P}_{\bm{z}}^{\epsilon}
\Big[\, \tau_{\mathcal{E}(\Lambda_H)}>
\frac{C}{\epsilon}\, \Big]=0\;.
\end{equation*}
Fix $h_0<h_1$, and denote by $\mc A$, $\mc B$ a connected component of
the set $\{\bs x\in\bb R^d : U(\bs x) < h_0\}$,
$\{\bs x \in\bb R^d : U(\bs x) < h_1\}$, respectively. Assume that
$\mc A \subset \mc B$, and that there are no critical points $\bs c$
of $U$ in $\mc B \setminus \mc A$. Then, for all $C>0$,
\begin{equation*}
\limsup_{\epsilon\rightarrow0}
\sup_{\bm{z}\in \mc A}
\mathbb{P}_{\bm{z}}^{\epsilon}
\Big[\, \tau_{\mathcal{E}(\mc  B)}>
\frac{C}{\epsilon}\, \Big]=0\;.
\end{equation*}
\end{thm}

\begin{corollary}
\label{t_hitting}
Fix $R>0$ large enough for $\Lambda_R$ to contain all the local minima
of $U$.  For all constant $C>0$,
\begin{equation*}
\limsup_{\epsilon\rightarrow0}
\sup_{\bm{x}\in \Lambda_R}
\mathbb{P}_{\bm{x}}^{\epsilon}
\Big[ \, \tau_{\mathcal{E}(\mathcal{M}_{0})}>
\frac{C}{\epsilon}\,\Big]=0\ .
\end{equation*}
\end{corollary}

Next result follows from Proposition \ref{p_FW} and Theorem
\ref{t_hitting2}.

\begin{corollary}
\label{l16}
Fix $h_0<h_1$, and denote by $\mc A$, $\mc B$ a connected component of
the set $\{\bs x \in\bb R^d: U(\bs x) < h_0\}$,
$\{\bs x \in\bb R^d : U(\bs x) < h_1\}$, respectively. Assume that
$\mc A \subset \mc B$, and that there are no critical points $\bs c$
of $U$ in $\mc B \setminus \mc A$. Then,
\begin{equation*}
\limsup_{\epsilon\rightarrow0}
\sup_{\bm{x}\in \mc A}
\mathbb{P}_{\bm{x}}^{\epsilon}
\big[ \, \tau_{\partial \mc B} < \tau_{\mathcal{E}(\mc B)}
\,\big] \;=\; 0\ .
\end{equation*}
\end{corollary}

\begin{remark}
We expect  the optimal time scale to be $O(\log \epsilon^{-1})$
instead of $O(\epsilon^{-1})$.
\end{remark}

Denote by $\mc N(\mc A)$, $\mathcal{A}\subset\mathbb{R}^{d}$, the
union, carried over all critical points $\bs c$ in
$\mc Y_0 \cap \mc A$, of the neighborhoods $\mc N (\bs c)$ introduced
in the previous section:
\begin{equation*}
\mathcal{N}(\mathcal{A})=
\bigcup_{\bm{c}\in\mathcal{Y}_{0}\cap\mathcal{A}}\mathcal{N}(\bm{c})\;.
\end{equation*}

\begin{lem}
\label{p: FW2}
Under the hypotheses of Theorem \ref{t_hitting2},
for all $C>0$,
\begin{equation*}
\begin{gathered}
\limsup_{\epsilon\rightarrow0} \sup_{\bm{z}\in \Lambda_H}
\mathbb{P}_{\bm{z}}^{\epsilon} \Big[\,
\tau_{\mc N(\Lambda_H) \cup  \mc E(\Lambda_H)}
>\frac{C}{\epsilon} \, \Big] \,=\, 0\;,
\\
\limsup_{\epsilon\rightarrow0} \sup_{\bm{z}\in \mc A}
\mathbb{P}_{\bm{z}}^{\epsilon} \Big[\,
\tau_{\mc N (\mc B) \cup  \mc E(\mc B)}
>\frac{C}{\epsilon} \, \Big] \,=\, 0\;.
\end{gathered}
\end{equation*}
\end{lem}

\begin{proof}
For each $\boldsymbol{z}\in\Lambda_H$, $\upsilon_{\bs z}(t)$ reaches
the set $\mc N(\Lambda_H)\cup\mathcal{E}(\Lambda_H)$ in finite time.
Therefore, the assertion of the lemma follows from
\cite[Theorem 2.1.2]{fw98}.
\end{proof}

\subsection*{Proof of Theorem \ref{t_hitting2}}

We prove the first assertion of the theorem. The arguments for the
second one are similar.  Recall the definition of
$\eta_{0}=\eta_{0}(\boldsymbol{c})$, $\bm{c}\in\mathcal{Y}_{0}$,
introduced at \eqref{eq:nocr}, and let
\begin{equation*}
\mathcal{H}(\bm{c})=
\left\{ \boldsymbol{x}:U(\boldsymbol{x})\le
U(\bm{c})-\eta_{0}\right\}\;.
\end{equation*}
By definition of $\eta$, there is no critical point $\boldsymbol{c}'$
of $U$ such that $U(\boldsymbol{c}')=U(\bm{c})-\eta_{0}$.

The proof is carried out by induction on
$| \mathcal{Y}_{0}\cap\Lambda_H |$, the number of critical point
$\bs c \in \mc Y_0$ which belong to $\Lambda_H$. If there are no such
critical points in $\Lambda_H$ the assertion of the theorem follows
from Lemma \ref{p: FW2}.

Consider the general case.  Decompose the probability
$\mathbb{P}_{\bm{z}}^{\epsilon} [\, \tau_{\mathcal{E}(\Lambda_H)}> C
{\epsilon}^{-1} \, ]$ into
\begin{equation*}
\mathbb{P}_{\bm{z}}^{\epsilon}\Big[\,
\tau_{\mathcal{E}(\Lambda_H)}>\frac{C}{\epsilon},\,
\tau_{\mathcal{E}(\Lambda_H)}=
\tau_{\mathcal{E}(\Lambda_H)\cup\mc N(\Lambda_H)}\, \Big]
+\mathbb{P}_{\bm{z}}^{\epsilon}\Big[\, \tau_{\mathcal{E}(\Lambda_H)}
>\frac{C}{\epsilon},\,\tau_{\mathcal{E}(\Lambda_H)}
>\tau_{\mathcal{E}(\Lambda_H)\cup\mc N(\Lambda_H)}\, \Big]\;.
\end{equation*}
The first probability is bounded above by
$\mathbb{P}_{\bm{z}}^{\epsilon} [\,
\tau_{\mathcal{E}(\Lambda_H)\cup\mc N(\Lambda_H)}> C {\epsilon}^{-1}
\, ]$. By Lemma \ref{p: FW2} this expression vanishes as
$\epsilon \to 0$. In view of this result, it remains to show that
\begin{equation*}
\limsup_{\epsilon\rightarrow0}\sup_{\bm{z}\in\Lambda_H}
\mathbb{P}_{\bm{z}}^{\epsilon} \Big[\, \tau_{\mathcal{E}(\Lambda_H)}
>\frac{C}{\epsilon},\;\tau_{\mathcal{E}(\Lambda_H)}
>\tau_{\mathcal{E}(\Lambda_H)\cup\mc N(\Lambda_H)},\;
\tau_{\mathcal{E}(\Lambda_H)\cup\mc N(\Lambda_H)}
\le\frac{C}{2\epsilon}\, \Big]=0\;.
\end{equation*}
By the strong Markov property, the last display is bounded by
\begin{equation*}
\limsup_{\epsilon\rightarrow0}
\sup_{\bm{z}\in\mathcal{E}(\Lambda_H)\cup\mc N(\Lambda_H)}
\mathbb{P}_{\bm{z}}^{\epsilon} \Big[\, \tau_{\mathcal{E}(\Lambda_H)}
>\frac{C}{2\epsilon}\,\Big]\;.
\end{equation*}
Since $\mathbb{P}_{\bm{z}}^{\epsilon}
[\, \tau_{\mathcal{E}(\Lambda_H)}>\frac{C}{2\epsilon}\,]=0$
if $\boldsymbol{z}\in\mathcal{E}(\Lambda_H)$, it suffices show
that, for each $\boldsymbol{c}\in\mathcal{Y}_{0} \cap \Lambda_H$,
\begin{equation*}
\limsup_{\epsilon\rightarrow0}
\sup_{\bm{z}\in\mc N(\boldsymbol{c})}
\mathbb{P}_{\bm{z}}^{\epsilon}\Big[\,
\tau_{\mathcal{E}(\Lambda_H)}>\frac{C}{2\epsilon}\,\Big]=0\;.
\end{equation*}
By Proposition \ref{p:Kifer}, it is enough  to prove that
\begin{equation*}
\limsup_{\epsilon\rightarrow0}
\sup_{\bm{z}\in\mc N(\boldsymbol{c})}
\mathbb{P}_{\bm{z}}^{\epsilon} \Big[\,
\tau_{\mathcal{E}(\Lambda_H)}>\frac{C}{2\epsilon} \;, \;\;
\tau_{\partial\widehat{\mc N}(\bm{c})}
\le\frac{C}{4\epsilon} \; ,\;\;
\bm{x}_{\epsilon}(\tau_{\partial\widehat{\mc N}(\bm{c})})
\in \mc H (\bm{c})\, \Big]=0\;.
\end{equation*}
By the strong Markov property, the left-hand side is bounded from
above by
\begin{equation*}
\limsup_{\epsilon\rightarrow0}
\sup_{\bm{z}\in\mc H(\bm{c})}\mathbb{P}_{\bm{z}}^{\epsilon}
\Big[\, \tau_{\mathcal{E}(\Lambda_H)}>\frac{C}{4\epsilon}\,\Big]=0\;.
\end{equation*}
As $\bs c$ belongs to $\Lambda_H$, $U(\bs c) \le H$ and
$\mc H(\bs c) \subset \Lambda_H$. Thus,
$\tau_{\mathcal{E}(\Lambda_H)} \le \tau_{\mathcal{E}(\mc H(\bm{c}))}$, and
it is enough to prove that
\begin{equation*}
\limsup_{\epsilon\rightarrow0}\sup_{\bm{z}\in\mc H(\bm{c})}
\mathbb{P}_{\bm{z}}^{\epsilon}\Big [ \,
\tau_{\mathcal{E}(\mc H(\bm{c}))}>\frac{C}{4\epsilon} \, \Big]=0 \;.
\end{equation*}
This identity follows from the induction hypothesis. Indeed, as
the critical point $\bs c$ belongs to $\Lambda_H$ and not to
$\mc H(\boldsymbol{c})$, the number of critical points in
$\mc Y_0 \cap \Lambda_H$ is strictly greater than the one in
$\mc Y_0 \cap \mc H(\boldsymbol{c})$. \qed \smallskip

We conclude this section with two results on hitting times of wells.
The first one follows from Theorem \ref{t01} and \cite[Lemma
4.2]{LMS}. It will be used in Section \ref{sec3} in the proof of
Theorem \ref{t00}. We state it here, before the proof of Theorem
\ref{t01}, to have all hitting time estimates of wells in the same
section.

\begin{lem}
\label{l17}
For all $\bs m \in \mathcal{M}_0$,
\begin{equation*}
\limsup_{a\rightarrow0}\,
\limsup_{\epsilon\rightarrow0}\,
\sup_{\boldsymbol{x}\in\mathcal{E}(\bs m)}\,
\mathbb{P}_{\boldsymbol{x}}^{\epsilon}
\big[\,\tau_{\mathcal{E} (\mc M_0) \setminus\mathcal{E}(\bs m)}
\le a\,\theta^{(1)}_\epsilon \, \big]=0\;.
\end{equation*}
\end{lem}

The last result asserts that starting from the domain of attraction of
a local minima the well associated to this local minimum is attained
before the other ones.  Recall that we denote by
$\upsilon_{\bs x} (t)$, $\bm{x} \in\bb R^d$, $t\ge0$, the solution of
the ODE \eqref{31} starting from $\bs x$. Denote by
$\color{blue} \mathcal{D}(\boldsymbol{m})$,
$\boldsymbol{m}\in\mathcal{M}_{0}$, the domain of attraction of
$\boldsymbol{m}$:
\begin{equation*}
\mathcal{D}(\boldsymbol{m})=
\big\{ \boldsymbol{x}\in\mathbb{R}^{d}:
\lim_{t\rightarrow\infty} \upsilon_{\bs x} (t) =\boldsymbol{m}
\big\} \;.
\end{equation*}

\begin{lem}
\label{lem_exit}
Let $\boldsymbol{m}\in\mathcal{M}_{0}$ and $\mathcal{K}$
be a compact subset of $\mathcal{D}(\boldsymbol{m})$. Then,
\begin{equation*}
\liminf_{\epsilon\rightarrow0}
\inf_{\boldsymbol{x}\in\mathcal{K}}
\mathbb{P}_{\boldsymbol{x}}^{\epsilon}
\left[\tau_{\mathcal{E}(\mathcal{M}_{0})}
=\tau_{\mathcal{E}(\boldsymbol{m})}\right]=1\;.
\end{equation*}
\end{lem}

\begin{proof}
Let
$\mathcal{F}(\boldsymbol{m})
:=\mathcal{D}(\boldsymbol{m})\setminus\mathcal{E}(\boldsymbol{m})$
so that
$\partial\mathcal{F}(\boldsymbol{m})
=\partial\mathcal{D}(\boldsymbol{m})\cup\partial\mathcal{E}(\boldsymbol{m})$.
Then,
\begin{equation*}
\mathbb{P}_{\boldsymbol{x}}^{\epsilon}
\left[\tau_{\mathcal{E}(\mathcal{M}_{0})}
=\tau_{\mathcal{E}(\boldsymbol{m})}\right]
\ge\mathbb{P}_{\boldsymbol{x}}^{\epsilon}
\left[\tau_{\partial\mathcal{F}(\boldsymbol{m})}
=\tau_{\partial\mathcal{E}(\boldsymbol{m})}\right]\;.
\end{equation*}
Therefore, it suffices to show that
\begin{equation}
\label{P_exit}
\liminf_{\epsilon\rightarrow0}
\inf_{\boldsymbol{x}\in\mathcal{K}}
\mathbb{P}_{\boldsymbol{x}}^{\epsilon}
\left[\tau_{\partial\mathcal{F}(\boldsymbol{m})}
=\tau_{\partial\mathcal{E}(\boldsymbol{m})}\right]=1\;.
\end{equation}
Since $\mathcal{K}$ is contained in the domain of attraction of
$\bs m$, the solution $\upsilon_{\boldsymbol{x}}(t)$ of the ODE
\eqref{31} starting from $\boldsymbol{x}\in\mathcal{K}$ exits the
domain $\mathcal{F}(\boldsymbol{m})$ at
$\partial\mathcal{E}(\boldsymbol{m})$. Thus, by \cite[Chapter 2,
Theorem 1.2]{fw98}, \eqref{P_exit} holds.

The estimate in \cite[Chapter 2, Theorem 1.2]{fw98} is not uniform in
$\bs x$, and just asserts that
\begin{equation*}
\liminf_{\epsilon\rightarrow0}
\mathbb{P}_{\boldsymbol{x}}^{\epsilon}
\left[\tau_{\partial\mathcal{F}(\boldsymbol{m})}
=\tau_{\partial\mathcal{E}(\boldsymbol{m})}\right]=1
\end{equation*}
for all $\boldsymbol{x}\in\mathcal{D}(\boldsymbol{m})$.  However, the
bound \cite[Chapter 2, Theorem 1.2]{fw98} holds uniformly over
$\boldsymbol{x}\in\mathcal{K}$ (The variable $a(t)$ in the proof
depends on $\boldsymbol{x}$ but can be bounded uniformly on any
compact set of $\mc D(\bs m)$, see the displayed equation above (1.6)).
This completes the proof of the lemma.
\end{proof}

\section{Test functions}
\label{sec10}

In this section, we construct the test functions used in Section
\ref{sec9} to estimate the solution of the resolvent equation. This
test function appeared before in \cite{BEGK, LMS2, LS-22}. For this
reason we just present its definition and main properties, and refer
the reader to \cite{LS-22b} for proofs.

Fix a local minimum $\bs m$.  The test function proposed
below and denoted by $Q_\epsilon$ is an approximation of the
equilibrium potential $h\colon \bb R^d \to [0,1]$ defined by
$h(\bs x) = \bb P^\epsilon_{\bs x} [ \tau_{\mc E(\bs m)} < \tau_{\mc
E(\mc M_0) \setminus \mc E(\bs m)} \,]$. In particular, inside the
wells the test function will be either very close to $1$ or very close
to $0$. In contrast, in very small neighborhoods of saddle points it
will change from $0$ to $1$. To capture this behavior, we linearize
the generator at the saddle point and set $Q_\epsilon$ to be close to
the equilibrium potential for the linearized generator.

\subsection*{Around a saddle point}

Fix a saddle point $\boldsymbol{\sigma}$ of $U$ such that
$\bs m \curvearrowleft \bs \sigma \curvearrowright \bs m'$ for
distinct local minima $\bs m$, $\bs m'$ of $U$. Let
$\color{blue} \mathbb{H}^{\boldsymbol{\sigma}} =
\nabla^{2}U(\boldsymbol{\sigma})$,
$\color{blue} \mathbb{L}^{\boldsymbol{\sigma}} = (D \bs \ell)(\bs
\sigma)$. By \eqref{47}, $\mathbb{H}^{\bs \sigma}$ has a unique
negative eigenvalue. Denote by
$\color{blue} -\lambda_{1},\,\lambda_{2},\,\dots,\,\lambda_{d}$ the
eigenvalues of $\mathbb{H}^{\bs \sigma}$, where $-\lambda_{1}$
represents the unique negative eigenvalue. Mind that we omit the
dependence on $\bs \sigma$ which is fixed.  Let
$ \color{blue} \boldsymbol{e}_{1}$, $\color{blue} \boldsymbol{e}_{k}$,
$k\ge2$, be the unit eigenvector associated with the eigenvalue
$-\lambda_{1}$, $\lambda_{k}$, respectively.  Choose $\bs e_1$
pointing towards $\bs m$: for all sufficiently small $a>0$,
$\bm{\sigma}+a\boldsymbol{e}_{1}$ belongs to the domain of attraction
of $\bs m$.  For $\boldsymbol{x}\in\mathbb{R}^{d}$ and
$k=1,\,\dots,\,d$, write
$\color{blue} \hat x_{k}=(\boldsymbol{x}-\boldsymbol{\sigma}) \cdot
\boldsymbol{e}_{k}$, so that
$\boldsymbol{x}=\bm{\sigma}+\sum_{m=1}^{d} \hat x_{m}\bm{e}_{m}$.

Let
\begin{equation*}
{\color{blue} \delta} = \delta(\epsilon) :=
(\epsilon\log\frac{1}{\epsilon})^{1/2}\ .
\end{equation*}
Fix a large constant $J>0$ to be chosen later, and denote by
$\mc A^\pm_\epsilon$, $\mathcal{C}_{\epsilon}$ the $d$-dimensional
rectangles defined by
\begin{gather*}
{\color{blue} \mathcal{A}^-_{\epsilon}}
\,:= \, \Big\{\,\boldsymbol{x}\in\mathbb{R}^{d}\,:\,
\hat x_{1}\in \Big[\,-\frac{J\delta}{\sqrt{\lambda_{1}}} - \epsilon^2,\,
- \frac{J\delta}{\sqrt{\lambda_{1}}}\,\Big] \,,\,
\hat x_{k}\in \Big[\,-\frac{2J\delta}{\sqrt{\lambda_{k}}},\,
\frac{2J\delta}{\sqrt{\lambda_{k}}}\,\Big] \,,\,
2\leq k\leq d\,\Big\} \\
{\color{blue} \mathcal{C}_{\epsilon}}
\,:= \, \Big\{\,\boldsymbol{x}\in\mathbb{R}^{d}\,:\,
\hat x_{1}\in \Big[\,-\frac{J\delta}{\sqrt{\lambda_{1}}},\,
\frac{J\delta}{\sqrt{\lambda_{1}}}\,\Big] \,,\,
\hat x_{k}\in \Big[\,-\frac{2J\delta}{\sqrt{\lambda_{k}}},\,
\frac{2J\delta}{\sqrt{\lambda_{k}}}\,\Big] \,,\,
2\leq k\leq d\,\Big\} \\
{\color{blue} \mathcal{A}^+_{\epsilon}}
\,:= \, \Big\{\,\boldsymbol{x}\in\mathbb{R}^{d}\,:\,
\hat x_{1}\in \Big[\, \frac{J\delta}{\sqrt{\lambda_{1}}} ,\,
 \frac{J\delta}{\sqrt{\lambda_{1}}} + \epsilon^2 \,\Big] \,,\,
\hat x_{k}\in \Big[\,-\frac{2J\delta}{\sqrt{\lambda_{k}}},\,
\frac{2J\delta}{\sqrt{\lambda_{k}}}\,\Big] \,,\,
2\leq k\leq d\,\Big\}\ .
\end{gather*}
Figure \ref{fig1} illustrates these definitions and the next ones.


\begin{figure}
\includegraphics[scale=0.2]{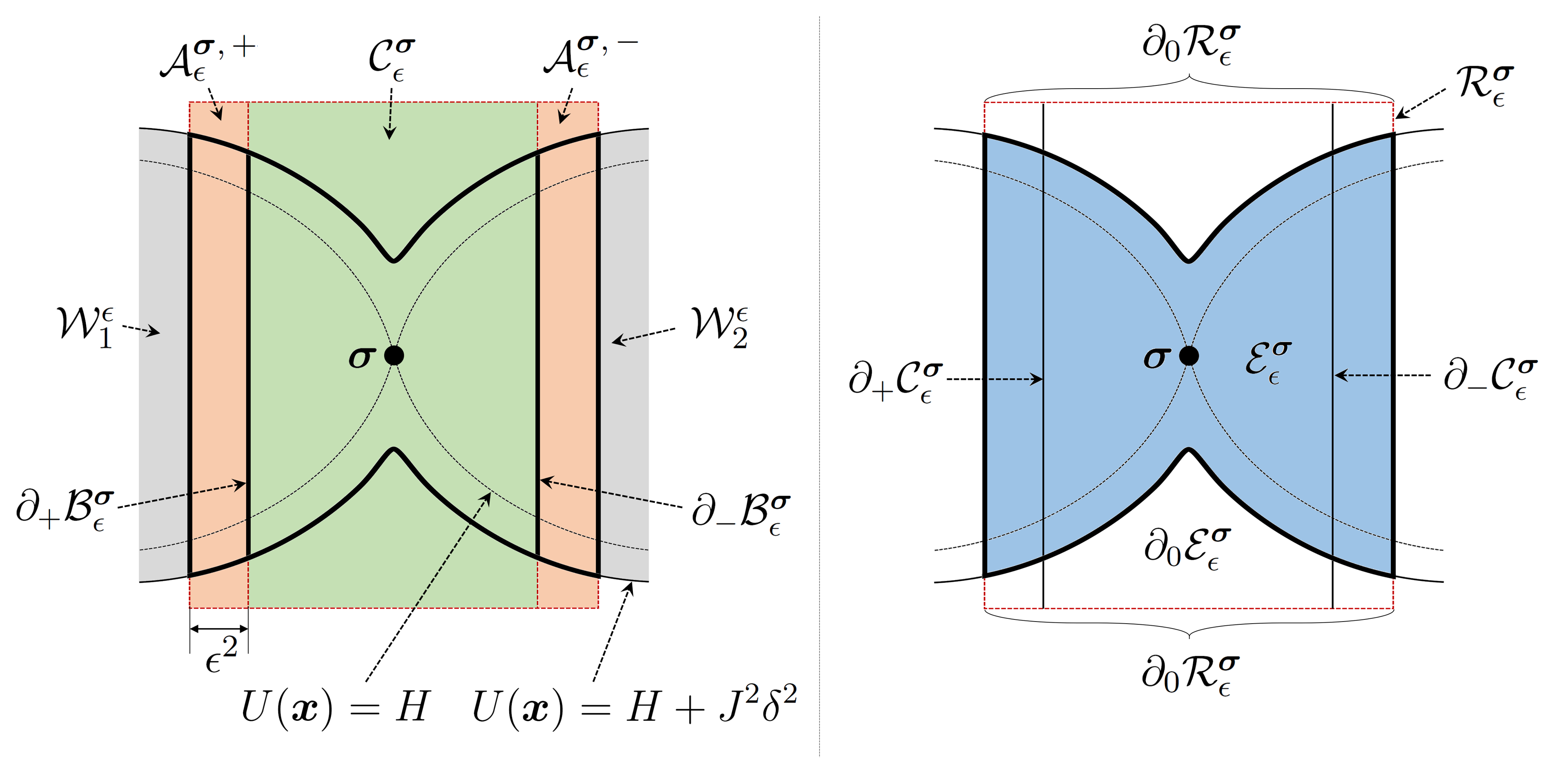}
\caption{The sets around a saddle point $\boldsymbol{\sigma}$
appearing in the definition of the test function.}
\label{fig1}
\end{figure}

Recall from \eqref{eq:nu} that
$\mathbb{H}^{\boldsymbol{\sigma}}+\mathbb{L^{\boldsymbol{\sigma}}}$
has a unique negative eigenvalue, denoted by $-\mu$. Denote by
$\mathbb{A}^{\dagger}$ the transpose of a matrix $\mathbb{A}$.  By
\cite[display (8.1)]{LS-22}, the matrix
$\mathbb{H}^{\boldsymbol{\sigma}}-(\mathbb{L^{\boldsymbol{\sigma}}})^{\dagger}$
also has a unique negative eigenvalue equal to $-\mu$.  Denote by
$\boldsymbol{v}$ the unit eigenvector of
$\mathbb{H}^{\boldsymbol{\sigma}}-(\mathbb{L^{\boldsymbol{\sigma}}})^{\dagger}$
associated with $-\mu$. By \cite[Lemma 8.1]{LS-22},
$\boldsymbol{v}\cdot\boldsymbol{e}_{1}\neq0$. We assume that
$\boldsymbol{v}\cdot\boldsymbol{e}_{1}>0$, as we can take
$-\boldsymbol{v}$ instead of $\boldsymbol{v}$ if this inner product is
negative.

Let
$p_{\epsilon}\colon
\mathcal{C}_{\epsilon} \to \mathbb{R}$ be given
by
\begin{equation}
\label{e_pesB}
p_{\epsilon}(\bm{x})\,:=\,
\frac{1}{M_{\epsilon}}
\int_{-\infty}^{(\bm{x}-\bm{\sigma})\cdot\bm{v}}\,
e^{-\frac{\mu}{2\epsilon}t^{2}}
\,dt \;,
\end{equation}
where the normalizing constant $M_{\epsilon}$ is
given by
\begin{equation}
\label{e_Ces}
M_{\epsilon}\,=\,
\int_{-\infty}^{\infty}\,e^{-\frac{\mu}{2\epsilon}t^{2}}\,dt
\,=\,\sqrt{\frac{2\pi\epsilon}{\mu}}\;.
\end{equation}

We extend continuously the function $p_{\epsilon}$ to the
$d$-dimensional rectangle
$\color{blue} \mc R_\epsilon = \mc A^-_\epsilon \cup \mc C_\epsilon
\cup \mc A^+_\epsilon$ as follows.  For
$\widehat{\bm{x}}=\bm{\sigma}+\sum_{k=1}^{d} \widehat x_{k}\bm{e}_{k}
\in \mc A^+_\epsilon$, let
\begin{equation}
\label{33}
\overline{\boldsymbol{x}}_r \,=\,
\bm{\sigma}\, + \,\frac{J\delta}{ \sqrt{\lambda_{1}}}\,
\bm{e}_{1} \,+\, \sum_{k=2}^{d} \widehat{x}_{k}\,\bm{e}_{k} \;.
\end{equation}
We define $\overline{\boldsymbol{x}}_l$ similarly for
$\bm x\in \mc A^-_\epsilon$, replacing on the right-hand side of the
previous formula the first plus sign by a minus sign. Clearly,
$\overline{\bs x}_r$ and $\overline{\bs x}_l$ belong to
$\mc C_\epsilon$. We extend the definition of $p_{\epsilon}$ to
$\mc R_\epsilon$ by setting
$p_{\epsilon} \colon \mc A^-_\epsilon \cup \mc A^+_\epsilon \to \bb R$
as
\begin{equation}
\label{e_pes}
\begin{gathered}
p_{\epsilon}(\bm{x})\,=\,
1\,+\, \epsilon^{-2} \,\Big[\, \hat x_1 - \frac{J\delta}{ \sqrt{\lambda_{1}}}
-\epsilon^2 \,\Big]\, [\, 1-p_{\epsilon} (\overline{\boldsymbol{x}}_r)\,]
\;, \quad
\bm{x} \in \mathcal{A}^+_{\epsilon}\;,
\\
p_{\epsilon}(\bm{x})\,=\,
\epsilon^{-2} \,
\Big[\, \hat x_1 +\frac{J\delta}{\sqrt{\lambda_{1}}}+ \epsilon^2\,\Big]\,
p_{\epsilon}(\overline{\boldsymbol{x}}_l)
\;, \quad \bm{x}\in \mathcal{A}^-_{\epsilon} \;.
\end{gathered}
\end{equation}

The function $p_{\epsilon}$ is an approximating solution of the
Dirichlet problem $\mathscr{L}_{\epsilon}^\dagger f = 0$ in
$\mathcal{R}_{\epsilon}$ with boundary conditions $f = 1$ on the
points of $\mc R_\epsilon$ where
$\hat x_1 = J\delta/ \sqrt{\lambda_{1}} + \epsilon^2$ and $f = 0$ on
the ones such that
$\hat x_1 = - J\delta/ \sqrt{\lambda_{1}} - \epsilon^2$.  This is the
content of \cite[Proposition 6.2]{LS-22b}, which states that the
integral of
$\theta^{(1)}_\epsilon \,|\mathscr{L}_{\epsilon}^\dagger f|$ on a set
slightly smaller than $\mathcal{R}_{\epsilon}$ vanishes as
$\epsilon\to 0$. This result also justifies the definition of the test
function $p_{\epsilon}$.

The test function $p_{\epsilon}(\cdot)$ constructed above depends on
$\bs \sigma$ and $\bs m$.  To stress this fact, it is sometimes
represented by $\color{blue} p^{\bs \sigma, \bs m}_{\epsilon}(\cdot)$.

\subsection*{Inside a well}

In this subsection we define a test function $Q_\epsilon$ on $\bb R^d$
with the help of the test functions $p^{\bs \sigma, \bs m}_\epsilon$
introduced in the previous subsection. Recall that we denote by
$B(\bs x, r)$ the open ball of radius $r$ centered at $\bs x$.

Fix a height $H$ such that $U(\bs \sigma) = H$ for some saddle point
$\bs \sigma$ of $U$. Denote by $\mc W$ a connected component of the
set $\{\bs x\in \bb R^d: U(\bs x) < H \}$. Assume that there exists a
saddle point $\bs \sigma' \in \partial \mc W$ satisfying condition (a)
below and that condition (b) is fulfilled for all saddle points
$\bs \sigma' \in \partial \mc W$ satisfying (a). Here,

\begin{enumerate}
\item[(a)] There exists $\delta_0>0$ such that
$B(\bs\sigma', \delta) \cap \{\bs x\in \bb R^d: U(\bs x) < H \}$ is not
contained in $\mc W$ for all $0<\delta <\delta_0$;

\item[(b)] There exists $\bs m$, $\bs m'\in \mc M_0$ such that
\begin{equation}
\label{56}
\bs \sigma' \curvearrowright \bs m \;\;\text{and}\;\; \bs \sigma'
\curvearrowright \bs m'\;.
\end{equation}
\end{enumerate}

Condition (b) prevents the existence of a heteroclinic orbit from
$\bs\sigma$ to a critical point which is not a local minimum. Clearly,
if conditions (a) and (b) hold, $\bs m\in \mc W$ and
$\bs m'\not\in \mc W$ or the contrary.

Let
$\color{blue} \mc S_H(\mc W) = \{ \bs \sigma_1, \dots \bs \sigma_p\}$
be the (non-empty) set of saddle points
$\bs \sigma \in \partial \mc W$ satisfying (a) and (b).  Note that
there might be saddle points $\bs \sigma$ in $\partial \mc W$ which do
not belong to $\mc S_H(\mc W)$ because they lead to critical points in
$\mc W$: $\bs \sigma \curvearrowright \bs x$,
$\bs \sigma \curvearrowright \bs y$ for $\bs x$,
$\bs y\in \mc W \cap \mc C_0$ . Figure \ref{fig2} illustrates this
possibility.  On the other hand, as $\bs \sigma \in \partial \mc W$,
$U(\bs \sigma) = H$ for all $\bs\sigma \in \mc S_H(\mc W)$.

Write, whenever needed to clarify, $\mc W$ as $\mc W_1$, and denote by
$\mc W_j$, $2\le j\le m$, the connected components of the set
$\{\bs x\in \bb R^d: U(\bs x) < H \}$ which share with $\mc W$ a
common saddle point in $\mc S_H(\mc W)$. Hence, for each $\mc W_j$,
$2\le j\le m$, there exist
$\bs \sigma \in \mc S_H(\mc W) \cap \overline{\mc W_j}$,
$\bs m\in \mc W \cap \mc M_0$, $\bs m'\in \mc W_j \cap \mc M_0$ such
that $\bs \sigma \curvearrowright \bs m$,
$\bs \sigma \curvearrowright \bs m'$.

\begin{figure}
\includegraphics[scale=0.2]{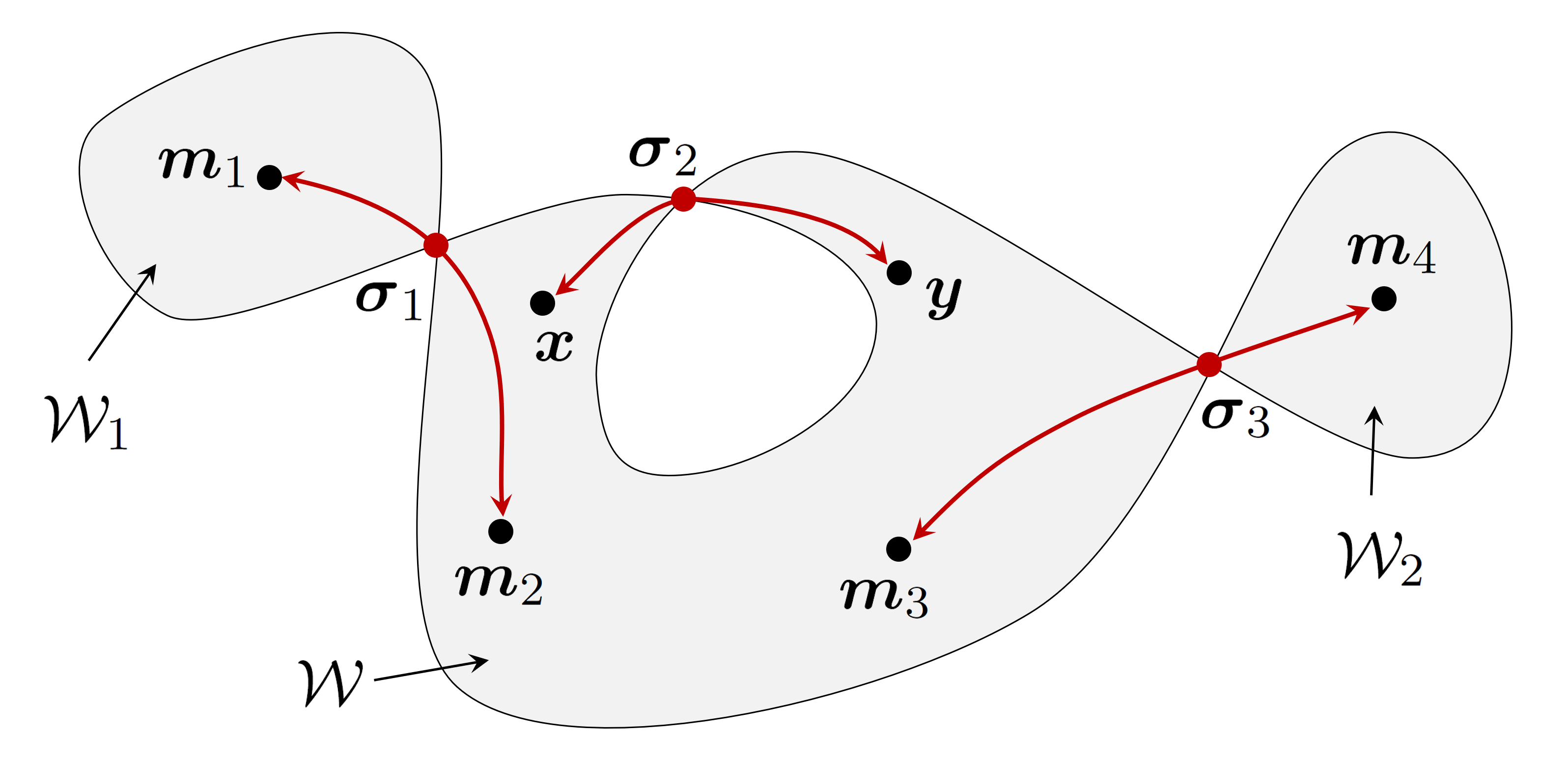}
\caption{The saddle points $\bs \sigma_2$ in $\partial \mc W$ does not
belong to $\mc S_H(\mc W)$ because it leads to critical points in
$\mc W$}
\label{fig2}
\end{figure}

Fix $\eta>0$ small enough so that there is no critical point $\bs x$
with height in the interval $(H,H+2\eta)$.  Let $\color{blue}\Omega$
be the connected component of the set
$\{\bs x\in \bb R^d: U(\bs x) \le H +\eta \}$ which contains $\mc W$
(and thus all connected components $\mc W_j$), and set
\begin{gather*}
{\color{blue} \mathcal{K}_{\epsilon} } \,:=\,
\big \{ \bm{x}\in\mathbb{R}^{d}\,:\,U(\bm{x})\le
H +J^{2}\delta^{2}\big\} \cap \Omega\;.
\end{gather*}

Denote by $\partial_{0}\mathcal{R}^{\bs \sigma}_{\epsilon}$,
$\bs\sigma \in \mc S_H(\mc W)$, the boundary of the set
$\mathcal{R}^{\bs \sigma}_{\epsilon}$, introduced in the previous
subsection, given by
\begin{gather*}
\partial_{0}\mathcal{R}^{\bs \sigma}_{\epsilon}
=\Big\{ \,\boldsymbol{x}\in\mathcal{R}^{\bs \sigma}_{\epsilon}\,:\,
\hat x_{k}=\pm\frac{2J\delta}{\sqrt{\lambda_{k}}}\ \
\text{for some}\ 2\leq k\leq d\,\Big\} \;.
\end{gather*}
By the proof of \cite[Lemma 8.3]{LS-22},
\begin{equation}
\label{23}
U(\bm{x})\geq U(\bm{\sigma})+
\frac{3}{2}\,J^{2}\,\delta^{2}\,[\,1+o_{\epsilon}(1)\,]
\end{equation}
for all $\bm{x}\in\partial_{0}\mathcal{R}_{\epsilon}$. In particular,
$\partial_{0}\mathcal{R}_{\epsilon}$ is contained in the complement of
$\mathcal{K}_{\epsilon}$ provided that $\epsilon$ is sufficiently
small.

Let
$\color{blue} \mathcal{E}_{\epsilon}^{\bm{\sigma}} \,:=\,
\mathcal{R}_{\epsilon}^{\bm{\sigma}}\cap \mathcal{K}_{\epsilon}$,
$\bs \sigma \in \mc S_H(\mc W)$.  Denote by $\mathcal{W}^{\epsilon}_1$
the connected component of
$\mathcal{K}_{\epsilon}\setminus( \bigcup_{\bm{\sigma}\in\mc
S_H(\mc W)}\mathcal{E}_{\epsilon}^{\bm{\sigma}})$ which intersects
$\mathcal{W}_1$, and let
$\mathcal{W}^{\epsilon}_2 = \mathcal{K}_{\epsilon} \setminus(
\mathcal{W}^{\epsilon}_1 \cup \bigcup_{\bm{\sigma}\in\mc S_H(\mc
W)}\mathcal{E}_{\epsilon}^{\bm{\sigma}})$. With this notation,
\begin{equation}
\label{36}
\Omega \;=\; \bigcup_{\bm{\sigma}\in \mc S_H(\mc W)}
\mathcal{E}_{\epsilon}^{\bm{\sigma}} \,\cup\,
\mathcal{W}^{\epsilon}_{1} \,\cup\,
\mathcal{W}^{\epsilon}_{2}  \, \cup\,
\big( \Omega \setminus\mathcal{K}_{\epsilon} \,\big)\;.
\end{equation}

For each $\bs \sigma \in \mc S_H(\mc W)$, denote by
$\color{blue} \bs m_{\bs \sigma}$ the local minimum $\bs m$ in $\mc W$
such that $\bs \sigma \curvearrowright \bs m$. Recall the notation
introduced at the end of the previous subsection, and let
$\color{blue} q^{\bs \sigma} = p^{\bs \sigma, \bs m_{\bs \sigma}}$.
Consider the test function
$Q_{\epsilon} \colon \mathcal{K}_{\epsilon} \to\mathbb{R}$ given by
\begin{gather}
\label{e: def_Q}
Q_{\epsilon} (\boldsymbol{x}) \,=\,
1 \;, \;\; \bm{x}\in\mathcal{W}_1^{\epsilon} \;; \quad
Q_{\epsilon} (\boldsymbol{y}) \,=\,
0 \;, \;\; \bm{y}\in\mathcal{W}_2^{\epsilon} \;;
\\
Q_{\epsilon} (\boldsymbol{x}) \,=\,
q_{\epsilon}^{\bm{\sigma}}(\bm{x})  \;, \;\;
\bm{x}\in\mathcal{E}_{\epsilon}^{\bm{\sigma}},\,
\bm{\sigma}\in \mc S_H(\mc W)  \;.
\nonumber
\end{gather}

By \eqref{e_pes}, the function $Q_{\epsilon}$ is continuous on
$\mathcal{K}_{\epsilon}$.  Moreover, if $\mc G_\epsilon$
represents the open set formed by the union of the interiors of
the set $\mc {E}_{\epsilon}^{\bm{\sigma}}$, $\bm{\sigma} \in \mc
S_H(\mc W)$, and the interior of the sets
$\mathcal{W}_i^{\epsilon}$, $i=1$, $2$,
\begin{equation*}
\lVert\nabla Q_{\epsilon}
\rVert_{L^{\infty}(\mathcal{G}_{\epsilon} )}
\, =\,  O(\epsilon^{-1/2})\;\ \text{and}\;\
\|\Delta Q_{\epsilon}\|_{L^{\infty}(\mathcal{G}_{\epsilon} )}
=O(\epsilon^{-3/2})\ .
\end{equation*}
We can extend $Q_{\epsilon}$ to $\Omega$ keeping these bounds
out of a $(d-1)$ dimensional manifold:
\begin{equation}
\label{18}
\lVert Q_{\epsilon}\rVert_{L^{\infty}(\Omega_0)} \, \le\ 1\;,\ \
\lVert \nabla Q_{\epsilon} \rVert_{L^{\infty}(\Omega_0)}\,=
O(\epsilon^{-1/2})\;,\ \;\text{and\;\;}
\|\Delta Q_{\epsilon} \|_{L^{\infty}(\Omega_0)}=
O(\epsilon^{-3/2})\ ,
\end{equation}
where $\Omega_0 = \Omega \setminus \mf M$, and $\mf M$ is $(d-1)$
dimensional manifold at which the gradient of $Q_{\epsilon}$ is
discontinuous. We further impose the condition that $Q_{\epsilon}$
vanishes away from $\Omega$:
\begin{equation}
\label{e: boundary_Q}
Q_{\epsilon}\equiv0 \,\, \text{ on }\,\,
\{\boldsymbol{x}\in\mathbb{R}^{d}:U(\boldsymbol{x})
>H +\frac{\eta}{2}\}\;,
\end{equation}
respecting the previous bounds. The function $Q_\epsilon$ is the test
function associated to the well $\mc W$ and height $H$.

\subsection*{Main estimate}

Next lemma is a crucial step in the proof of Theorem \ref{t01}. To
stress below the dependence of the set $\mathcal{C}_{\epsilon}$,
$\mathcal{A}^{\pm}_{\epsilon}$ on a saddle point $\bs \sigma$, we add
a superscript $\bs \sigma$ in the notation.  Denote by
$\partial_{\pm}\mathcal{C}^{\bs \sigma}_{\epsilon}$ the boundary of
the set $\mathcal{C}^{\bs \sigma}_{\epsilon}$ given by
\begin{gather*}
\partial_{\pm}\mathcal{C}^{\bs \sigma}_{\epsilon} \,=\, \big\{
\,\boldsymbol{x}\in\mathcal{C}^{\bs \sigma}_{\epsilon}\,:\,
\hat x_{1}
=\pm\frac{J\delta}{\sqrt{\lambda_{1}}}\,\big\} \;, \;\;
\text{and let}\;\;
{\color{blue} \mathcal{B}_{\epsilon}^{\bm{\sigma}}} \, :=\,
\mathcal{C}_{\epsilon}^{\bm{\sigma}}\cap
\mathcal{K}_{\epsilon}\;,\;\;
{\color{blue} \partial_{\pm}\mathcal{B}_{\epsilon}^{\bm{\sigma}}} \,:=\,
\partial_{\pm}\mathcal{C}_{\epsilon}^{\bm{\sigma}}
\cap\mathcal{K}_{\epsilon}\;.
\end{gather*}

Recall from \eqref{32} the definition of $\nu_\star$, and from
\eqref{33} the definition of $\overline{\bs x}_r$,
$\overline{\bs x}_l$.  In the statement of Lemma \ref{p03}, the
vectors $\bs v$ and $\bs e_1$ depend on $\bs\sigma$, but this
dependence is omitted from the notation.  For $c>0$, let
\begin{equation*}
{\color{blue} \Lambda_{c,\epsilon} } \,:=\, \big \{
\bm{x}\in\mathbb{R}^{d}\,:\,U(\bm{x})\le H - c\,
J^{2} \delta^{2} \big\} \;.
\end{equation*}

\begin{lem}
\label{p03}
There exists $c_0>0$, such that for all $0<c<c_0$,
${\bf g}\colon \mc M_0 \to \bb R$,
\begin{equation}
\label{24}
e^{H/\epsilon} \, \int_{\Omega}\,
Q_{\epsilon} \,(-\mathscr{L}_{\epsilon}\,\phi_{\epsilon})
\,d\mu_{\epsilon}
\; =\; \sum_{\bm{\sigma}\in \mc S_H(\mc W) }
J (\bs \sigma) \,+\, o_{\epsilon}(1) \;,
\end{equation}
where
$J (\bs \sigma) = J_+ (\bs \sigma) - J_- (\bs \sigma) $, and
\begin{equation*}
\begin{aligned}
& J_{+} (\bs \sigma) \, =\, -\,
[\,1+   o_{\epsilon}(1) \,]\,
\frac{\epsilon\, \sqrt{\mu^{\bs \sigma}}
\,(\bm{v}\cdot\bm{e}_{1})}{(2\pi\epsilon)^{(d+1)/2}\,
\nu_\star }\,\int_{\partial_{+}\mathcal{B}^{\bs \sigma}_{\epsilon}
\cap \Lambda_{c,\epsilon}}\,
e^{-\frac{1}{2\epsilon}\bm{x}\cdot\,(\mathbb{H}^{\bs \sigma} +
\mu^{\bs \sigma}\,\bm{v}\otimes\bm{v})\,\bm{x}}\,
\phi_{\epsilon} (\bm{x})\, {\rm S} (d\bm{x})
\\
&\quad  -\, [\,1+  o_{\epsilon}(1) \,]\,
\frac{1} {(2\pi)^{(d+1)/2}\ \nu_\star \, \sqrt{\mu^{\bs \sigma}}
\,\epsilon{}^{(d+3)/2}}\,
\int_{\mc A^{\bs \sigma, +}_{\epsilon} \cap \Lambda_{c,\epsilon} }\,
\phi_{\epsilon} (\bm{x}) \,\frac{\mathbb{L}^{\bs \sigma}
\overline{\boldsymbol{x}}_r \cdot\boldsymbol{e}_{1}}
{\overline{\boldsymbol{x}}_r \cdot\bm{v}}\,
e^{-\frac{1}{2\epsilon}\overline{\boldsymbol{x}}_r
\cdot\,(\mathbb{H}^{\bs \sigma} +\mu^{\bs \sigma}\,\bm{v}\otimes\bm{v})\,
\overline{\boldsymbol{x}}_r}\,d\bm{x}\;.
\end{aligned}
\end{equation*}
In this formula, ${\rm S} (d\bm{x})$ represents the surface measure on
the $(d-1)$-dimensional manifold
$\partial_{+}\mathcal{B}^{\bs \sigma}_{\epsilon} \cap
\Lambda_{c,\epsilon}$, and $J_{-}(\bs \sigma)$ is obtained from
$J_{+}(\bs \sigma)$ by removing the minus sign and replacing
$\partial_{+}\mathcal{B}^{\bs \sigma}_{\epsilon}$,
$\mc A^{\bs \sigma, +}_{a,\epsilon}$ by
$\partial_{-}\mathcal{B}^{\bs \sigma}_{\epsilon}$,
$\mc A^{\bs \sigma, -}_{a,\epsilon}$, respectively.
\end{lem}

The proof of this result is omitted as it is the content of
\cite[Section 7]{LS-22b}.

\section{Domain of attraction}
\label{sec4}

Fix $\bs\sigma \in \mc S_H(\mc W)$.  Denote by
$\color{blue} \bs n_{\bs \sigma}$ the local minimum $\bs m$ of $U$
which does not belong to $\mc W$ and such that
$\bs \sigma \curvearrowright \bs m$.  The main result of this section
asserts that we may replace $\phi_{\epsilon} (\bm{x})$ in the formula
for $J_+(\bs\sigma)$, $J_-(\bs\sigma)$ by
$\phi_{\epsilon} (\bm{m}_{\bs\sigma})$,
$\phi_{\epsilon} (\bm{n}_{\bs\sigma})$, respectively.

\begin{prop}
\label{l10}
There exists $c_0>0$, such that for all $0<c<c_0$,
\begin{equation*}
\lim_{\epsilon\to0}\,
\sup_{\bm{x}\in\partial_{+}\mathcal{B}_{\epsilon}^{\bm{\sigma}}\cap
\Lambda_{c,\epsilon}}
\,|\phi_{\epsilon} (\bm{x}) - \phi_{\epsilon} (\bm{m}_{\bs\sigma})|=0\;, \quad
\lim_{\epsilon\to0}\,
\sup_{\bm{x}\in\partial_{-}\mathcal{B}_{\epsilon}^{\bm{\sigma}}\cap
\Lambda_{c,\epsilon}}
\,|\phi_{\epsilon} (\bm{x})-\phi_{\epsilon} (\bm{n}_{\bs\sigma})|=0
\end{equation*}
for all $\bs\sigma \in \mc S_H(\mc W)$. A similar result holds if we
replace $\partial_{+}\mathcal{B}_{\epsilon}^{\bm{\sigma}}$,
$\partial_{-}\mathcal{B}_{\epsilon}^{\bm{\sigma}}$ by
$\mc A^{\bs \sigma, +}_{\epsilon}$,
$\mc A^{\bs \sigma, -}_{\epsilon}$, respectively.
\end{prop}

The proof of Proposition \ref{l10} is based on the following general
result.  Recall that we denote by $\mathcal{D}(\boldsymbol{m})$,
$\boldsymbol{m}\in\mathcal{M}_{0}$, the domain of attraction of
$\boldsymbol{m}$.

\begin{prop}
\label{prop:DoC}
Fix $\boldsymbol{m}\in\mathcal{M}_{0}$, and a sequence
$(\mathcal{K}_{\epsilon})_{\epsilon>0}$ of subsets of
$\mathcal{D}(\boldsymbol{m})$. Assume that
$\bigcup_{\epsilon>0}\mathcal{K}_{\epsilon}$ is a bounded set, and
\begin{equation}
\label{P_exit0}
\liminf_{\epsilon\rightarrow0}
\inf_{\boldsymbol{x}\in\mathcal{K}_{\epsilon}}
\mathbb{P}_{\boldsymbol{x}}^{\epsilon}
\left[\tau_{\mathcal{E}(\mathcal{M}_{0})}
=\tau_{\mathcal{E}(\boldsymbol{m})}\right]=1
\end{equation}
Then,
\begin{equation*}
\limsup_{\epsilon\rightarrow0}
\sup_{\boldsymbol{x}\in\mathcal{K}_{\epsilon}}
\big|\, \phi_{\epsilon} (\boldsymbol{x})
-\phi_{\epsilon} (\boldsymbol{m})\, \big|
\,=\, 0\;.
\end{equation*}
\end{prop}

\begin{proof}
Recall the definition of the function $G\colon\bb R^d\to\bb R$
introduced in \eqref{e_res}. By the stochastic representation of the
solution of the resolvent equation,
\begin{equation}
\label{exp_phi}
\phi_{\epsilon}(\bm{x})\,=\,
\mathbb{E}_{\bm{x}}^{\epsilon} \Big[\,
\int_{0}^{\infty}e^{-\lambda s}\,
G(\bm{x}_{\epsilon}(\theta_{\epsilon}^{(1)}s))\,ds
\, \Big]\;.
\end{equation}
As $G$ is bounded, the absolute value of the time integral is bounded
by $\lambda^{-1} \, \Vert \bs g\Vert_\infty$. Therefore, as
$\bigcup_{\epsilon>0}\mathcal{K_{\epsilon}}$ is a bounded set and $U(\bs
x) \to \infty$ as $|\bs x|\to\infty$, taking $R$ sufficiently large in
Corollary \ref{t_hitting},
\begin{equation}
\label{exp_phi2}
\phi_{\epsilon}(\bm{x}) \,=\,
\mathbb{E}_{\bm{x}}^{\epsilon}\Big[\,
\int_{0}^{\infty} e^{-\lambda s}\,G(\bm{x}_{\epsilon}
(\theta_{\epsilon}^{(1)}s))\,ds\;
{\bf 1}\{\tau_{\mathcal{E}(\mathcal{M}_{0})}
\le\frac{C_0}{\epsilon}\}\, \Big] \;+\;
R_{\epsilon}(\bs x)\;,
\end{equation}
where, here and below, $R_{\epsilon}(\bs x)$ is an error term
such that
\begin{equation*}
\lim_{\epsilon\to 0}
\sup_{\boldsymbol{x}\in\bigcup_{\epsilon>0}\mathcal{K_{\epsilon}}}
|\, R_{\epsilon}(\bs x)\,| \;=\; 0\;.
\end{equation*}
Consider the time integral in the interval
$[0, \tau_{\mathcal{E}(\mathcal{M}_{0})}/\theta_{\epsilon}^{(1)}]$.  As
$G$ is bounded and $\epsilon \, \theta_{\epsilon}^{(1)} \to\infty$,
the expectation of this piece is bounded by $R_{\epsilon}(\bs x)$.
By the strong Markov property, the second piece is equal to
\begin{equation*}
\mathbb{E}_{\bm{x}}^{\epsilon}
\Big[\, \mathbb{E}_{\bm{x}_{\epsilon}
(\tau_{\mathcal{E}(\mathcal{M}_{0})})}^{\epsilon}
\Big[\, \int_{0}^{\infty}e^{-\lambda s}\,
G(\bm{x}_{\epsilon}(\theta_{\epsilon}^{(1)}s))\,ds\,
\Big]\, e^{-\lambda\tau_{\mathcal{E}(\mathcal{M}_{0})}/
\theta_{\epsilon}^{(1)}}\,
{\bf 1}\{\tau_{\mathcal{E}(\mathcal{M}_{0})}
\le\frac{C_0}{\epsilon}\}\, \Big]\;.
\end{equation*}
By the same reasons invoked above, this expression is equal to
\begin{equation*}
\mathbb{E}_{\bm{x}}^{\epsilon}
\Big[\, \mathbb{E}_{\bm{x}_{\epsilon}
(\tau_{\mathcal{E}(\mathcal{M}_{0})})}^{\epsilon}
\Big[\, \int_{0}^{\infty}e^{-\lambda s}\,
G(\bm{x}_{\epsilon}(\theta_{\epsilon}^{(1)}s))\,ds\,
\Big]\,
{\bf 1}\{\tau_{\mathcal{E}(\mathcal{M}_{0})}
\le\frac{C_0}{\epsilon}\}\, \Big]
\;+\; R_{\epsilon}(\bs x) \;.
\end{equation*}
In conclusion,
\begin{equation*}
\phi_{\epsilon} (\bm{x}) \,=\,
\mathbb{E}_{\bm{x}}^{\epsilon}
\Big[\, \phi_{\epsilon}(\bm{x}_{\epsilon}
(\tau_{\mathcal{E}(\mathcal{M}_{0})}))\,
{\bf 1}\{\tau_{\mathcal{E}(\mathcal{M}_{0})}\le
\frac{C_0}{\epsilon}\}\, \Big]
\;+\; R_{\epsilon}(\bs x) \;.
\end{equation*}
Applying Corollary \ref{t_hitting} once more, as $\phi_{\epsilon}$ is
uniformly bounded, the right-hand side is equal to
\begin{equation*}
\mathbb{E}_{\bm{x}}^{\epsilon}
\Big[\, \phi_{\epsilon}(\bm{x}_{\epsilon}
(\tau_{\mathcal{E}(\mathcal{M}_{0})})) \, \Big]
\;+\; R_{\epsilon}(\bs x)
\;=\; \mathbb{E}_{\bm{x}}^{\epsilon}
\Big[\, \phi_{\epsilon}(\bm{x}_{\epsilon}
(\tau_{\mathcal{E}(\boldsymbol{m})}))\, \Big]
\;+\; R_{\epsilon}(\bs x) \;,
\end{equation*}
where we used hypothesis \eqref{P_exit0} and the uniform boundedness
of $\phi_{\epsilon}$ in the last step. To complete the proof, it
remains to recall the assertion of Theorem \ref{p_flat}.
\end{proof}

Recall that $\bs\sigma \in \mc S_H(\mc W)$ is fixed and that
$\bs\sigma \curvearrowright \bs m_{\bs \sigma}$,
$\bs\sigma \curvearrowright \bs n_{\bs \sigma}$. Denote by
$\color{blue} B[\bs x, r]$ the closed ball of radius $r$ centered at
$\bs x$, and by $\mc W'$ the connected component of the set
$\{{\bm x} \in\bb R^d : U({\bm x}) < U(\bs \sigma) \}$ whose closure
contains $\bs\sigma$ and $\bs n_{\bs \sigma}$. Next lemma is a
consequence of Theorem \ref{thm:H-G}.

\begin{lem}
\label{l08}
There exist $\delta>0$ such that
$(B[\bs \sigma, \delta] \cap \overline{\mc W} )  \setminus
\{\bs\sigma\} $ is contained in the domain of attraction
$\mc D({\bs m}_{\bs \sigma})$ of ${\bs m}_{\bs \sigma}$, and
$(B[\bs \sigma, \delta] \cap \overline{\mc W'} ) \setminus
\{\bs\sigma\} $ is contained in the domain of attraction
$\mc D({\bs n}_{\bs \sigma})$ of ${\bs n}_{\bs \sigma}$.
\end{lem}

\begin{proof}[Proof of Proposition \ref{l10}]
We prove the first assertion, as the second is similar.  By Lemma
\ref{l08}, there exists $\epsilon_{1}>0$ such that
$\partial_+ \mathcal{B}_{\epsilon}^{\bm{\sigma}} \cap
\Lambda_{c,\epsilon} \subset\mathcal{D}(\boldsymbol{m}_{\bs \sigma})$
for all $\epsilon<\epsilon_1$. Therefore, by Proposition
\ref{prop:DoC}, it suffices to show that
\begin{equation}
\label{eq:flat}
\liminf_{\epsilon\rightarrow0}\,
\inf_{\bm{x}\in\partial_{+}\mathcal{B}_{\epsilon}^{\bm{\sigma}}
\cap\Lambda_{c,\epsilon}}
\mathbb{P}_{\bm{x}}^{\epsilon}
[\, \tau_{\mathcal{E}(\bs m_{\bs \sigma})}
= \tau_{\mathcal{E}(\mathcal{M}_{0})} \,] \,=\, 1\;.
\end{equation}

Recall that $\mc W$ represents the well that contains
$\bs m_{\bs \sigma}$. Let
${\color{blue} \mc G_\epsilon} =
\partial_{+}\mathcal{B}_{\epsilon}^{\bm{\sigma}} \cap \overline{\mc
W}$. By Lemma \ref{l08}, there exists $\epsilon_0>0$ such
that $\mc G_\epsilon \subset \mc D(\bs m_{\bm{\sigma}})$ for all
$\epsilon \le \epsilon_0$. We claim that
\begin{equation}
\label{eq_esc1}
\inf_{\bm{x}\in\partial_{+}\mathcal{B}_{\epsilon}^{\bm{\sigma}}
\cap\Lambda_{c, \epsilon}}
\mathbb{P}_{\bm{x}}^{\epsilon}[
\tau_{\partial\mathcal{C}_{\epsilon_{0}}^{\bm{\sigma}}}
= \tau_{\mc G_{\epsilon_0}}]=1-o_{\epsilon}(1)\ ,
\end{equation}
where $\partial \mathcal{C}_{\epsilon}^{\bm{\sigma}}$ represents the
boundary of $\mc C^{\bm{\sigma}}_\epsilon$.

\begin{figure}
\includegraphics[scale=0.2]{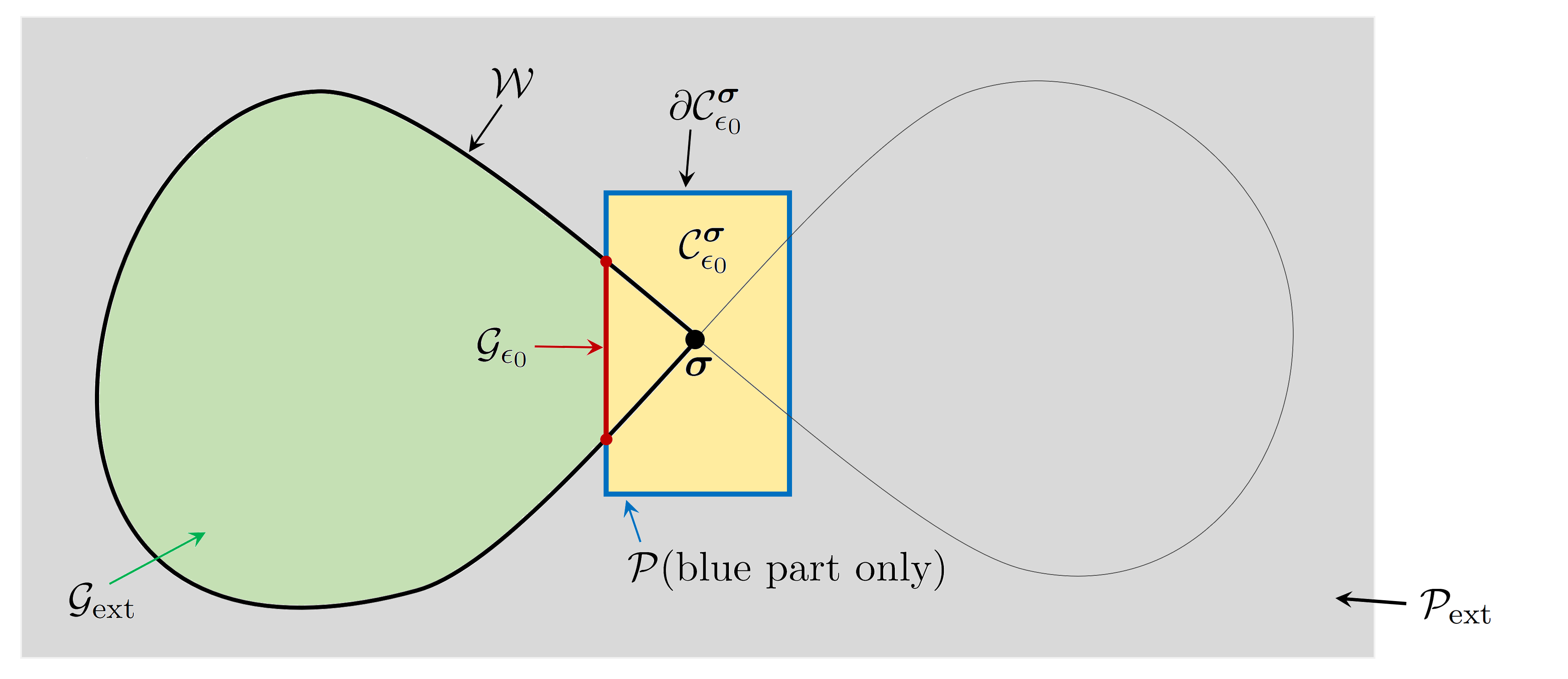}
\caption{The sets $\mathcal{G}_{\rm ext}$ and $\mathcal{P}_{\rm ext}$
introduced in the proof of Proposition \ref{l10}.}
\label{fig3}
\end{figure}

To prove \eqref{eq_esc1}, let
$\color{blue} \mathcal{P} \,:=\, \partial
\mathcal{C}_{\epsilon_0}^{\bm{\sigma}} \setminus \mc G_{\epsilon_0}$,
$\color{blue} \mathcal{G}_{\rm ext} = \overline{\mc W \setminus \mc
C^{\bs \sigma}_{\epsilon_0}}$,
$\color{blue} \mathcal{P}_{\rm ext} = \overline {\bb R^d \setminus
[ \mathcal{G}_{\rm ext} \cup \mc C_{\epsilon_0}^{\bm{\sigma}}]
}$. Figure \ref{fig3} illustrates these sets.
By definition,
\begin{equation*}
\{\,  \tau_{\partial\mathcal{C}_{\epsilon_{0}}^{\bm{\sigma}}}
= \tau_{\mc G_{\epsilon_0}} \,\}  \;=\;
\{\,  \tau_{\mc G_{\epsilon_0}} < \tau_{\mc P} \,\} \;=\;
\{\,
\tau_{\mc G_{\rm ext}} < \tau_{\mc P_{\rm ext}} \,\}
\end{equation*}
for all
$\bm{x}\in\partial_{+}\mathcal{B}_{\epsilon}^{\bm{\sigma}}
\cap\Lambda_{c, \epsilon}$, $\epsilon < \epsilon_0$. Therefore,
by definition of the set $\mc P_{\rm ext}$,
$\{\tau_{\mathcal{E}(\bs m_{\bs \sigma})} < \tau_{\mathcal{P}_{\rm
ext}}\} \subset \{
\tau_{\partial\mathcal{C}_{\epsilon_{0}}^{\bm{\sigma}}} = \tau_{\mc
G_{\epsilon_{0}}}\}$.

Fix $\epsilon < \epsilon_0$. Recall that we denote by $B(\bs x, r)$
the open ball of radius $r$ centered at $\bs x$. By \cite[Lemma
9.2]{LS-22}, there exists a finite constant $C_0$, whose value may
change from line to line, such that
\begin{equation*}
\textup{cap}_{\epsilon}( B(\bs y, {\epsilon})
\,,\, \mathcal{E}(\bs m_{\bs \sigma})) \,\geq\,
C_0\,\epsilon^{d}\,Z_{\epsilon}^{-1}\,e^{-U(\bm{y})/\epsilon}
\end{equation*}
for all
$\bm{y}\in\partial_{+}\mathcal{B}_{\epsilon}^{\bm{\sigma}}\cap\Lambda_{c,
\epsilon}$. In this formula, ${\rm cap}_\epsilon (\mc A, \mc B)$
stands for the capacity between the sets $\ms A$, $\mc B$ for the
diffusion $\bs x_\epsilon (\cdot)$ and is defined in Appendix
\ref{sec-ap2}.  On the other hand, by the proof of \cite[Lemma
9.3]{LS-22}, there exists a finite constant $C_0$ such that
\begin{equation*}
\textup{cap}_{\epsilon}(B(\bm{y}, \epsilon) \,,\,
\mathcal{P}_{\rm ext}) \le C_0 \,Z_{\epsilon}^{-1} \,e^{-U(\bm{\sigma})/\epsilon}
\end{equation*}
for all
$\bm{y}\in\partial_{+}\mathcal{B}_{\epsilon}^{\bm{\sigma}}\cap\Lambda_{c,
\epsilon}$.  By \cite[Proposition 7.2]{LMS2}, there exists a finite
constant $C_0$ such that
\begin{equation*}
\mathbb{P}_{\bm{x}}^{\epsilon}[\,\tau_{\mathcal{P}_{\rm ext}}
< \tau_{\mathcal{E}(\bs m_{\bs \sigma})}\,] \,\le\,
C_0\,\frac{\textup{cap}_{\epsilon}( B(\bs x , \epsilon)
\, ,\, \mathcal{P}_{\rm ext})}
{\textup{cap}_{\epsilon}
(B(\bs x, \epsilon) \, ,\,\mathcal{E}(\bs m_{\bs \sigma}))}
\end{equation*}
for all $\bs x\in \bb R^d$.  Combining the previous estimates yields
that for all
$\bm{x}\in\partial_{+}\mathcal{B}_{\epsilon}^{\bm{\sigma}}\cap\Lambda_{c,
\epsilon}$
\begin{equation*}
\mathbb{P}_{\bm{x}}^{\epsilon}
[\,\tau_{\mathcal{P}_{\rm ext}}<\tau_{\mathcal{E}(\bs m_{\bs \sigma})}\,]
\,\le\,  C_0 \, \epsilon^{-d} \, e^{(U(\bm{x})-U(\bm{\sigma}))/\epsilon}
\,\le\,  C_0 \, \epsilon^{-d}\, e^{-c J^2\delta^{2}/\epsilon} \,= \,
C_0 \, \epsilon^{cJ^2-d}\;.
\end{equation*}
This expression vanishes as $\epsilon\to 0$ for large enough $J$.  To
complete the proof of assertion \eqref{eq_esc1}, it remains to recall
that
$\{\tau_{\mathcal{E}(\bs m_{\bs \sigma})} < \tau_{\mathcal{P}_{\rm
ext}}\} \subset \{
\tau_{\partial\mathcal{C}_{\epsilon_{0}}^{\bm{\sigma}}} = \tau_{\mc
G_{\epsilon_{0}}}\}$.

We turn to the proof of \eqref{eq:flat}. For
$\bm{x}\in\partial_{+}\mathcal{B}_{\epsilon}^{\bm{\sigma}}\cap\Lambda_{c,
\epsilon}$, by the strong Markov property and \eqref{eq_esc1},
\begin{align*}
\mathbb{P}_{\bm{x}}^{\epsilon}[\tau_{\mathcal{E}(\bs m_{\bs \sigma})}
=\tau_{\mathcal{E}(\mathcal{M}_{0})}] &
\, \ge\, \mathbb{P}_{\bm{x}}^{\epsilon}[\tau_{\mathcal{E}(\bs m_{\bs \sigma})}
=\tau_{\mathcal{E}(\mathcal{M}_{0})},\,
\tau_{\partial\mathcal{C}_{\epsilon_{0}}^{\bm{\sigma}}}
=\tau_{\mc G_{\epsilon_0}}] \\
&
\,\ge\, \inf_{\boldsymbol{y}\in \mc G_{\epsilon_0}}
\mathbb{P}_{\bm{y}}^{\epsilon}[\tau_{\mathcal{E}(\bs m_{\bs \sigma})}
=\tau_{\mathcal{E}(\mathcal{M}_{0})}] \,
\mathbb{P}_{\bm{x}}^{\epsilon}[\tau_{\partial\mathcal{C}_{\epsilon_{0}}^{\bm{\sigma}}}
=\tau_{\mc G_{\epsilon_0}}] \\
& \,=\,  (1-o_{\epsilon}(1))
\inf_{\boldsymbol{y}\in \mc G_{\epsilon_0}}
\mathbb{P}_{\bm{y}}^{\epsilon}[\tau_{\mathcal{E}(\bs m_o)}
=\tau_{\mathcal{E}(\mathcal{M}_{0})}]\;.
\end{align*}
The last infimum is $1-o_{\epsilon}(1)$ by Lemma \ref{lem_exit}
because $\mc G_{\epsilon_0} \subset \mc D(\bs m_0)$.
\end{proof}

Proposition \ref{l10} provides a simple formula for the quantities
$J_\pm (\bs\sigma)$ introduced in Lemma \ref{p03}.

\begin{lem}
\label{p05}
For all ${\bf g}\colon \mc M_0 \to \bb R$,
\begin{equation}
\label{37}
e^{H/\epsilon} \, \int_{\Omega}\,
Q_{\epsilon} \,(-\mathscr{L}_{\epsilon}\,\phi_{\epsilon})
\,d\mu_{\epsilon}  \; =\; \frac{1}{2\pi\nu_\star}\,
\sum_{\bm{\sigma}\in \mc S_H(\mc W) }
\, [\phi_\epsilon (\bs m_{\bs\sigma}) - \phi_\epsilon (\bs n_{\bs\sigma})]\,
\frac{\mu^{\bs\sigma}} {\sqrt{  - \det \mathbb{H}^{\bs\sigma}}}
\,+\, o_{\epsilon}(1) \;.
\end{equation}
\end{lem}

\begin{proof}
By Proposition \ref{l10}, in the formula for $J_+(\bs\sigma)$
presented in Lemma \ref{p03}, we may replace $\phi_\epsilon(\bs x)$ by
$\phi_\epsilon (\bs m_{\bs \sigma})$ at a cost $o_\epsilon(1)$, and we
are left with a Gaussian type integral. A straightforward computation,
presented in the proof of \cite[Lemma 7.5]{LS-22b}, together with
\cite[Lemma 7.3]{LS-22b} yields that
\begin{equation*}
\frac{\epsilon\,\sqrt{\mu^{\bs\sigma}}
\,(\bm{v}\cdot\bm{e}_{1})}
{(2\pi\epsilon)^{(d+1)/2}\, \nu_\star} \,
\int_{\partial_{+}\mathcal{B}^{\bs \sigma}_{\epsilon} \cap \Lambda_{c,\epsilon}}\,
e^{-\frac{1}{2\epsilon}\bm{x}\cdot\,
(\mathbb{H}^{\bs \sigma} +\mu^{\bs\sigma} \,\bm{v}\otimes\bm{v})}\,
{\rm S} (d\bm{x})
\;=\; \frac{1}{2\pi \, \nu_\star}\,
\frac{\lambda^{\bs\sigma}_1}
{ \sqrt{  - \det \mathbb{H}^{\bs\sigma}}}  + o_{\epsilon}(1) \;.
\end{equation*}
Similarly, by the proof of \cite[Lemma 7.7]{LS-22b},
\begin{equation*}
\begin{aligned}
& \frac{1} {(2\pi)^{(d+1)/2}\ \nu_\star\,
\sqrt{\mu^{\bs \sigma}}\,\epsilon{}^{(d+3)/2}}\,
\int_{\mc A^{\bs \sigma, +}_{\epsilon} \cap \Lambda_{c,\epsilon} }\,
\frac{\mathbb{L}^{\bs \sigma}
\overline{\boldsymbol{x}}\cdot\boldsymbol{e}_{1}}
{\overline{\boldsymbol{x}}\cdot\bm{v}}\,
e^{-\frac{1}{2\epsilon}\overline{\boldsymbol{x}}
\cdot\,(\mathbb{H}^{\bs\sigma}+\mu^{\bs\sigma}\,\bm{v}\otimes\bm{v})\,
\overline{\boldsymbol{x}}}\,d\bm{x}
\\
& \quad \;=\;
\frac{1}{2\pi\nu_\star}\,
\frac{\lambda^{\bs\sigma}_1}{ \sqrt{  - \det \mathbb{H}^{\bs\sigma}}}
\, \frac{(\mathbb{L}^{\bs \sigma} (\mathbb{H}^{\bs\sigma})^{-1}
\bm{v}) \cdot\bm{e}_{1} }
{\bs v \cdot \bs e_1}\,  + \, o_{\epsilon}(1) \;.
\end{aligned}
\end{equation*}
By the proof of \cite[Proposition 5.7]{LS-22b},
\begin{equation*}
\bs v \cdot \bs e_1 \;+\;
(\mathbb{L}^{\bs \sigma} (\mathbb{H}^{\bs\sigma})^{-1} \bm{v})
\cdot\bm{e}_{1} \;=\; \frac{\mu^{\bs\sigma}}
{\lambda^{\bs\sigma}_1} \, \bs v \cdot \bs e_1\;.
\end{equation*}
Combining the previous estimates yields that
\begin{equation*}
J_+(\bs\sigma) \;=\;
-\, \frac{1}{2\pi \, \nu_\star}\,
\frac{\mu^{\bs\sigma}_1}
{ \sqrt{  - \det \mathbb{H}^{\bs\sigma}}}  \,
\phi_\epsilon(\bs m_{\bs\sigma}) \,+\, o_{\epsilon}(1) \;.
\end{equation*}
The same argument leads to the same formula for $J_-(\bs\sigma) $ with
a plus sign and $\phi_\epsilon(\bs n_{\bs\sigma})$ instead of
$\phi_\epsilon(\bs m_{\bs\sigma})$. This completes the proof of the
lemma.
\end{proof}

\section{Proof of Theorem \ref{t01}}
\label{sec9}

Recall from \eqref{28}, \eqref{35} the definitions of the generator
$\mf L_1$, and the function
$\bs{f}_{\epsilon} \colon \mc M_0 \to \bb R$, respectively. The
main result of this section reads as follows.

\begin{thm}
\label{p01}
For all $\lambda >0$, $\mb g \colon \mc M_0 \to \bb R$,
\begin{equation*}
(\lambda \,-\, \mf L_1)\, \mb f_{\epsilon}
\;=\; \mb g \;+\; o_\epsilon (1)\;.
\end{equation*}
\end{thm}

\begin{proof}[Proof of Theorem \ref{t01}]
The assertion follows from two observations. The sequence
$\mb f_{\epsilon}$ is uniformly bounded and the equation
$(\lambda \,-\, \mf L_1)\, \mb f \;=\; \mb g$ has a unique solution.
\end{proof}

The remainder of this section is devoted to the proof of Theorem
\ref{p01}.  Fix $\bs m \in \mc M_0$.  Let $\mc W$ be the connected
component of the set
$\{\bs x\in\bb R^d: U(\bs x) < U(\bs m) + \Gamma (\bs m)\}$ which
contains $\bs m$. By definition, $\mc W$ does not contain any other
local minimum of $U$ (in particular, the present situation is
different from the one represented in Figure \ref{fig2}, where
$\mc W$ contains more than one local minimum). Recall from \eqref{56}
the definition of the set $\mc S_H(\mc W)$.

\begin{lem}
\label{l13}
There exists a saddle point $\bs \sigma \in \partial \mc W$ satisfying
condition (a) in \eqref{56}. Condition (b) is fulfilled for all saddle
points $\bs \sigma' \in \partial \mc W$ satisfying (a).  Moreover,
$\mc S_H(\mc W) = \Upsilon (\bs m)$, where
$H= U(\bs m) + \Gamma(\bs m)$.
\end{lem}

\begin{proof}
By Proposition \ref{l05a}, there exist a local minimum $\bs m'$ of $U$
different from $\bs m$ and a continuous path
$\boldsymbol{z} \colon [0,1] \to \bb R^d$ such that
$U(\bs m) + \Gamma (\bs m) = \Theta(\bs m, \bs m')$,
$\boldsymbol{z}(0)=\boldsymbol{m}$,
$\boldsymbol{z}(1)=\boldsymbol{m}'$, and
\begin{equation}
\label{54a}
\max_{t\in[0,\,1]}U(\boldsymbol{z}(t)) \;=\;
U(\boldsymbol{z}(1/2)) \;=\;
\Theta(\boldsymbol{m},\,\boldsymbol{m}')\;,
\quad
U(\bs z(s)) \,<\, U(\bs z(1/2))\;,\;\;
s\in [0,1]\setminus\{1/2\} \;,
\end{equation}
and $\bs\sigma := \bs z(1/2)$ is a saddle point of $U$.  In
particular, $\bs \sigma \in \partial \mc W$. Condition (a) is
satisfied because $\bs m' \neq \bs m$ and $\mc W$ contains only the
local minimum $\bs m$.

We turn to condition (b). Let $\bs\sigma$ be a saddle point in
$\partial \mc W$ satisfying (a).  By definition of
$\mc W$ and with the help of the solution of the ODE \eqref{31}, it is
possible to construct a continuous path
$\boldsymbol{z} \colon [0,1] \to \bb R^d$ such that
$\boldsymbol{z}(0)=\boldsymbol{m}' \in \mc M_0$,
$\boldsymbol{z}(1/2) \;=\; \bs\sigma$,
$\boldsymbol{z}(1)=\boldsymbol{m}'' \in \mc M_0$, and
\begin{equation*}
U(\bs z(s)) \,<\, U(\bs \sigma)\;,\;\; s\, \neq\,  1/2 \;,
\end{equation*}
for some $\boldsymbol{m}'$, $\bs m'' \in \mc M_0$. As $\bs\sigma$
satisfies (a), we may assume without loss of generality that
$\boldsymbol{m}' \in \mc W$,
$\boldsymbol{m}'' \in \overline{\mc W}^c$. Since $\mc W$ contains a
unique local minimum, $\bs m'=\bs m$. Therefore, since
$U(\bs\sigma) = U(\bs m) + \Gamma(\bs m)$, by definition of
$\Upsilon (\bs m)$, $\bs\sigma \in \Upsilon (\bs m)$. Hence, by
condition \eqref{48}, there exists $\bs m''' \in \mc M_0$,
$\bs m'''\not\in \mc W$, such that $\bs \sigma \curvearrowright \bs m$,
$\bs \sigma \curvearrowright \bs m'''$, which is condition (b).

Assume that $\bs \sigma \in \mc S_H(\mc W)$. By definition, it
satisfies (a). Thus, by the previous paragraph,
$\bs\sigma\in \Upsilon (\bs m)$. Conversely, suppose that
$\bs \sigma \in \Upsilon (\bs m)$. By definition, there exists a local
minimum $\bs m'\neq \bs m$ and a continuous path
$\boldsymbol{z} \colon [0,1] \to \bb R^d$ such that
$\boldsymbol{z}(0)=\boldsymbol{m}$,
$\boldsymbol{z}(1)=\boldsymbol{m}'$ for which \eqref{54a} holds. By
Proposition \ref{l05a}, $\bs\sigma := \bs z(1/2)$ is a saddle point of $U$.
Since $\mc W$ has a unique local minimum, $\bs m'\not \in \mc
W$. Thus, condition (a) holds for $\bs\sigma$. By \eqref{48},
condition (b) also holds, so that $\bs\sigma\in \mc S_H(\mc W)$.
\end{proof}

\begin{proof}[Proof of Theorem \ref{p01}]

Fix $\bs m \in \mc M_0$.  Let $\mc W$ be the connected component of
the set $\{\bs x\in\bb R^d: U(\bs x) < U(\bs m) + \Gamma (\bs m)\}$
which contains $\bs m$. By Lemma \ref{l13}, there exists a saddle
point $\bs \sigma \in \partial \mc W$ satisfying condition (a) in
\eqref{56}, and condition (b) is fulfilled for all saddle points
$\bs \sigma' \in \partial \mc W$ satisfying (a). We may therefore
apply Lemma \ref{p03}.

Let $Q_\epsilon$ be the test function constructed in Section
\ref{sec10} associated to the well $\mc W$, and recall that
$H=U(\bs m) + \Gamma (\bs m)$. Multiply both sides of \eqref{e_res} by
the test function $Q_{\epsilon}$ and integrate over $\bb R^d$ to
deduce that
\begin{equation}
\label{14}
\int_{\Omega}\,Q_{\epsilon}\,(\lambda-\theta_{\epsilon}^{(1)}\,
\mathscr{L}_{\epsilon})\,\phi_{\epsilon}\,d\mu_{\epsilon}
\;=\;  \bs{g} (\bs m) \,
\int_{\mathcal{E}(\bs m)} \, Q_{\epsilon} \, \,d\mu_{\epsilon}\;,
\end{equation}
where $\Omega$ is given by \eqref{36}.

By definition of $Q_{\epsilon}$, the right-hand side is equal to
${\bf g}(\bs m) \, \mu_{\epsilon}(\mathcal{E}(\bs m))$.  Similarly, as
$\phi_{\epsilon}$ is uniformly bounded and
\begin{equation*}
\mu_{\epsilon}\Big(\, \bigcup_{\bm{\sigma}\in\mc S_H (\mc W)}
\mathcal{E}_{\epsilon}^{\bm{\sigma}} \,\cup\,
(\Omega \setminus \mathcal{K}_{\epsilon}) \,\Big)
\, =\, o_{\epsilon}(1)\, \mu_{\epsilon}  (\mathcal{E}(\bs m)) \;,
\end{equation*}
as $Q_\epsilon$ vanishes on $\mc W_2$ and is equal to $1$ on
$\mc W_1$, by definition of ${\bf f}_\epsilon$,
\begin{equation}
\label{25}
\lambda \int_{\Omega}\, Q_{\epsilon}
\,\phi_{\epsilon}\,d\mu_{\epsilon}
\;=\; \lambda \, {\bf f}_\epsilon (\bs m) \,
\mu_{\epsilon}(\mathcal{E}(\bs m))
\; + \; o_\epsilon (1) \, \mu_{\epsilon}  (\mathcal{E}(\bs m))
\, \,.
\end{equation}

It remains to consider the term in \eqref{14} involving the generator
$\ms L_\epsilon$. We examine two cases separately.

\smallskip\noindent {\it Case 1:} Assume that
$\Gamma (\bs m) > d^{(1)}$. As $\Gamma (\bs m) > d^{(1)}$ and
$e^{-U(\bs m)/\epsilon} / \mu_{\epsilon} (\mathcal{E}(\bs m)) \le C_0$
for some finite constant independent of $\epsilon$,
$\theta_{\epsilon}^{(1)} = e^{d^{(1)}/\epsilon} \prec e^{\Gamma (\bs
m)/\epsilon} \le C_0\, e^{[\Gamma (\bs m) + U (\bs m)]/\epsilon} \,
\mu_{\epsilon} (\mathcal{E}(\bs m))$. Hence, by Lemma \ref{p05}, as
the right-hand side of \eqref{37} is bounded and
$H=\Gamma(\bs m) + U (\bs m)$,
\begin{equation*}
\theta_{\epsilon}^{(1)}\,
\int_{\Omega}\,Q_\epsilon \, (-\mathscr{L}_{\epsilon})\,
\phi_{\epsilon} \,d\mu_{\epsilon} \;=\;
o_\epsilon (1) \, \mu_{\epsilon}  (\mathcal{E}(\bs m))\;.
\end{equation*}
Combining the previous estimates yields that
\begin{equation*}
{\bf f}_{\epsilon}(\bs m) =
\frac{1}{\lambda} \, {\bf g}(\bs m)  \;+\; o_\epsilon (1) \;.
\end{equation*}

By \eqref{40} and the definition of $\mf L_1$, as
$\Gamma (\bs m) > d^{(1)}$, $(\mf L_1 \mb f_\epsilon)(\bs m) =0$,
which completes the proof of the theorem in Case 1.

\smallskip\noindent {\it Case 2:} Assume that
$\Gamma (\bs m) = d^{(1)}$. Multiply both sides of \eqref{37} by
$e^{- U(\bs m)/\epsilon}$. Since
$\theta_{\epsilon}^{(1)} = e^{d^{(1)}/\epsilon} = e^{\Gamma (\bs
m)/\epsilon} = e^{[H - U(\bs m)]/\epsilon}$, by Lemma \ref{p05},
\begin{equation*}
\theta_{\epsilon}^{(1)}\,
\int_{\Omega}\,Q_\epsilon \, (-\mathscr{L}_{\epsilon})\,
\phi_{\epsilon} \,d\mu_{\epsilon} \;=\;
\Big\{\frac{1}{2\pi\nu_\star}\,
\sum_{\bm{\sigma}\in \mc S_H(\mc W) }
\, [\phi_\epsilon (\bs m) - \phi_\epsilon (\bs n_{\bs\sigma})]\,
\frac{\mu^{\bs\sigma}} {\sqrt{  - \det \mathbb{H}^{\bs\sigma}}}
\,+\, o_{\epsilon}(1) \Big\}  \, e^{-U(\bs m)/\epsilon}
\end{equation*}
because $\bs m_\sigma = \bs m$ for all
$\bs \sigma \in \mc S_H(\mc W)$, as $\bs m$ is the only local minima
of $U$ in $\mc W$.

Since
$e^{-U(\bs m)/\epsilon} / \mu_{\epsilon} (\mathcal{E}(\bs m)) \le C_0$
for some finite constant independent of $\epsilon$, we may replace in
the previous formula, $o_{\epsilon}(1) \, e^{-U(\bs m)/\epsilon} $ by
$o_{\epsilon}(1) \, \mu_{\epsilon} (\mathcal{E}(\bs m))$. On the other
hand, by \eqref{55},
$e^{-U(\bs m)/\epsilon}/\nu_\star = [1+o_\epsilon(1)]\, \mu_{\epsilon}
(\mathcal{E}(\bs m))/\nu(\bs m)$. We may therefore rewrite the
right-hand side of the previous equation as
\begin{equation*}
\Big\{\frac{1}{2\pi\nu(\bs m)}\,
\sum_{\bm{\sigma}\in \mc S_H(\mc W) }
\, [\phi_\epsilon (\bs m) - \phi_\epsilon (\bs n_{\bs\sigma})]\,
\frac{\mu^{\bs\sigma}} {\sqrt{  - \det \mathbb{H}^{\bs\sigma}}}
\,+\, o_{\epsilon}(1) \Big\}  \, \mu_{\epsilon} (\mathcal{E}(\bs m))\;.
\end{equation*}

By Lemma \ref{l13}, $\mc S_H(\mc W) = \Upsilon (\bs m)$.  Thus, by
\eqref{38a} and by definition of $\bs n_{\bs\sigma}$, introduced at
the beginning of Section \ref{sec4},
$\{\bs n_{\bs\sigma} : \bm{\sigma}\in \mc S_H(\mc W) \} = \{\bs
n_{\bs\sigma} : \bm{\sigma}\in \Upsilon (\bs m) \} = \ms V(\bs
m)$. Hence, by Theorem \ref{p_flat}, the previous expression can
rewritten as
\begin{equation*}
\Big\{\frac{1}{2\pi\nu(\bs m)}\,
\sum_{\bs m' \in \ms V(\bs m)}
[\mb f_\epsilon (\bs m) - \mb f_\epsilon (\bs m')]\,
\sum_{\bm{\sigma}\in \mc S(\bs m, \bs m') }
\, \frac{\mu^{\bs\sigma}} {\sqrt{  - \det \mathbb{H}^{\bs\sigma}}}
\,+\, o_{\epsilon}(1) \Big\}  \, \mu_{\epsilon} (\mathcal{E}(\bs m))\;.
\end{equation*}
By \eqref{eq:nu}, \eqref{39a}, \eqref{40} and \eqref{28}, the previous
expression is equal to
\begin{equation*}
\big\{ \, (-\, \mf L_1 \mb f_\epsilon)(\bs m) \,+\,
o_{\epsilon}(1) \,\big\}  \, \mu_{\epsilon} (\mathcal{E}(\bs m))\;.
\end{equation*}
To complete the proof of the theorem, it remains to combine the
estimates obtained at the beginning of the proof with this last one.
\end{proof}

\subsection*{Trace processes}

Let $\bs y_\epsilon (t)$ be the process $\bs x_\epsilon(t)$ speeded-up
by $\theta^{(1)}_\epsilon$:
$\color{blue} \bs y_\epsilon(t) = \bs
x_\epsilon(t\theta^{(1)}_\epsilon)$. Denote by
$\color{blue} \bb Q^{\epsilon}_{\bs x}$ the probability measure on
$C(\bb R_+, \bb R^d)$ induced by the process $\bs y_\epsilon (t)$
starting from $\bs x$. We use the same symbol
$\bb Q^{\epsilon}_{\bs x}$ to represent the expectation with respect to
the measure $\bb Q^{\epsilon}_{\bs x}$.

Denote by $T_{\epsilon}(t)$ the time spent by
$\bs{y}_{\epsilon}(\cdot)$ on $\mathcal{E}(\mc M_0)$ up to time $t>0$:
\begin{equation*}
{\color{blue} T_{\epsilon}(t)} \; :=\;
\int_{0}^{t}\, \chi_{_{\mathcal{E} (\mc M_0)}}(\bs{y}_{\epsilon}(s))\,ds\;.
\end{equation*}
Let $S_{\epsilon}(\cdot)$ be the generalized inverse of the
non-decreasing process $T_{\epsilon}(\cdot)$:
\begin{equation*}
{\color{blue} S_{\epsilon}(t)} \;:=\; \sup
\{s\ge 0 : T_\epsilon (s) \le t\,\}\;, \quad t\ge 0 \;.
\end{equation*}

Define the trace process of $\bs y_\epsilon(\cdot)$ on
$\mc E(\mc M_0)$ by
\begin{equation}
\label{eq:trace}
{\color{blue} \bs y^{\rm T}_\epsilon(t)} \, :=\,
\bs{y}_{\epsilon} (\, S_{\epsilon}(t)\,)\;, \quad t\ge0 \;,
\end{equation}
which is an $\mc E(\mc M_0)$-valued Markov process. Let
$\Phi: \mc E(\mc M_0) \to \mc M_0$ be the projection given by
$\Phi = \sum_{\bs m\in \mc M_0} \bs m \, \chi_{\mc E(\bs m)}$. Next
result is a consequence of Theorem \ref{t01} and \cite[Theorem
2.3]{LMS}.

\begin{thm}
\label{t02}
Fix $\bs m \in \mc M_0$, and a sequence
$\bs x_\epsilon \in \mc E(\bs m)$.  Starting from $\bs x_\epsilon$,
the process $\Phi(\bs y^{\rm T}_\epsilon(t))$ converges in the
Skorohod topology to the $\mc M_0$-valued continuous-time Markov chain
induced by the generator $\mf L_1$ starting from $\bs m$. Moreover,
for all $T>0$,
\begin{equation}
\label{57}
\lim_{\epsilon\rightarrow0} \sup_{\boldsymbol{x}\in \mc E(\mc M_0)}
\mathbb{Q}_{\bm{x}}^{\epsilon}
\Big[\, \int_{0}^{T} \chi_{_{\bb R^d \setminus \mc E(\mc M_0)}}
(\boldsymbol{y}_{\epsilon}(t)) \; dt \, \Big] \,=\, 0\;.
\end{equation}
\end{thm}

Next assertion is a consequence of \eqref{57}. We refer to
\cite[display (3.2)]{LLM} for a proof.

\begin{lem}
\label{l: lem1}
For all $t\ge0$ and $\delta>0$,
\begin{equation*}
\limsup_{\epsilon\rightarrow0}\,
\sup_{\boldsymbol{x}\in\mathcal{E} (\mc M_0) }\,
\mathbb{Q}_{\boldsymbol{x}}^{\epsilon}
[\,S_{\epsilon}(t)>t+\delta\,]=0\;.
\end{equation*}
\end{lem}

\section{Proof of Theorem \ref{t00}: finite-dimensional distributions}
\label{sec3}

The main result of this section, Theorem \ref{t_fdd}, states that in
the time-scale $\theta^{(1)}_\epsilon$ the finite-dimensional
distributions of the diffusion $\bs x_\epsilon (t)$ converge to those
of the $\mc M_0$-valued Markov chain whose generator is given by $\mf
L_1$ introduced in \eqref{28}.

Denote by $\color{blue} D(\bb R_+, \mc M_0)$ the space of
right-continuous functions $\mb y: \bb R_+ \to \mc M_0$ with left
limits endowed with the Skorohod topology. Let $\color{blue} \mc Q_{\bs m}$,
$\bs m\in\mc M_0$, be the measure on $D(\bb R_+, \mc M_0)$ induced by
the continuous-time $\mc M_0$-valued Markov chain associated to the
generator $\mf L_1$ starting from $\bs m$. Expectation with respect to
$\mc Q_{\bs m}$ is also represented by $\mc Q_{\bs m}$.

\begin{thm}
\label{t_fdd}
Fix $\bs m\in \mc M_0$, and $\boldsymbol{x}\in\mathcal{D}(\bs
m)$. Then,
\begin{equation*}
\lim_{\epsilon\rightarrow0} \mathbb{E}_{\boldsymbol{x}}^{\epsilon}
\Big[\, \prod_{j=1}^{{\mf n}}
F_j (\boldsymbol{x}_{\epsilon}(\theta^{(1)}_{\epsilon} t_{j})) \, \Big]
\,=\, \mathcal{Q}_{\bs m}
\big[\, \prod_{j=1}^{{\mf n}} F_j (\mb y (t_{j})) \, \Big]
\end{equation*}
for all ${\mf n}\ge1$, $0<t_{1}<\cdots<t_{{\mf n}}$ and bounded continuous
functions $F_j\colon \bb R^d \to \bb R$, $1\le j\le {\mf n}$.
\end{thm}

The proof of Theorem \ref{t_fdd} is based on the next result.

\begin{prop}
\label{l11}
Fix $r_0>0$ small enough to fulfill the conditions above equation
\eqref{30}, and recall from this equation the definition of the wells
$\mc E(\bs m')$, $\bs m'\in \mc M_0$. Fix $\bs m\in \mc M_0$. Then,
\begin{equation*}
\lim_{\epsilon\rightarrow0} \mathbb{P}_{\boldsymbol{x}_\epsilon}^{\epsilon}
\Big[\, \bigcap_{j=1}^{{\mf n}}
\{\, \boldsymbol{x}_{\epsilon}(\theta^{(1)}_{\epsilon} t_{j,\epsilon} ))
\in \mc E(\bs m_j)\, \} \, \Big]
\,=\, \mathcal{Q}_{\bs m}
\big[\, \bigcap_{j=1}^{{\mf n}} \{ \mb y (t_{j}))
= \bs m_j \, \} \, \, \big]
\end{equation*}
for all ${\mf n}\ge 1$, $0<t_{1}<\cdots<t_{{\mf n}}$,
$\bs m_1, \dots, \bs m_{\mf n} \in \mc M_0$, and sequences
$\bs x_\epsilon \in \mc E(\bs m)$, $t_{j,\epsilon} \to t_j$.
\end{prop}

It follows from this result that
\begin{equation}
\label{42}
\lim_{\epsilon\rightarrow0} \mathbb{P}_{\boldsymbol{x}_\epsilon}^{\epsilon}
\Big[\, \bigcap_{j=1}^{{\mf n}}
\{\, \boldsymbol{x}_{\epsilon}(\theta^{(1)}_{\epsilon} t_{j,\epsilon} ))
\in \mc E(\mc M_0 )\, \} \, \Big] \;=\; 1
\end{equation}
for all $\bs m\in \mc M_0$, ${\mf n}\ge 1$, $0<t_{1}<\cdots<t_{{\mf n}}$,
$\bs m_1, \dots, \bs m_{\mf n} \in \mc M_0$, and sequences
$\bs x_\epsilon \in \mc E(\bs m)$, $t_{j,\epsilon} \to t_j$.

\subsection*{Proof of Theorem \ref{t_fdd}}

We prove the result for ${\mf n}=1$ as the arguments in the general case are
identical. Fix $t>0$, $\eta>0$ and a bounded continuous function
$F\colon \bb R^d\to \bb R$. By continuity, there exists $\delta_0>0$ such
that
\begin{equation}
\label{39}
\max_{\bs m \in\mc M_0}\,
\sup_{\bs x \in \mc W^{2\delta_0}(\bs m)} |\, F(\bs x) - F(\bs m)\,| \;\le\; \eta\;.
\end{equation}
Fix $r_0 <\delta_0$ small enough to fulfill the conditions of
Proposition \ref{l11}.  Consider the wells $\mc E(\bs m)$ defined by
\eqref{30}.

Recall from \eqref{41} that we represent by $\tau_{\ms A}$ the hitting
time of the set $\ms A$, and let
$\color{blue} \tau = \tau_{\mc E(\bs m)}$. By Lemma \ref{lem_exit},
Corollary \ref{t_hitting}, and the strong Markov property, as
$\bs x\in \mc D(\bs m)$ and $F$ is bounded,
\begin{equation*}
\mathbb{E}_{\boldsymbol{x}}^{\epsilon}
\big[\, F (\boldsymbol{x}_{\epsilon}(\theta^{(1)}_{\epsilon} t)) \, \big]
\;=\;
\mathbb{E}_{\boldsymbol{x}}^{\epsilon}
\Big[\, \mathbb{E}_{\boldsymbol{x}_\epsilon (\tau)}^{\epsilon} \big[\,
F (\boldsymbol{x}_{\epsilon}(\theta^{(1)}_{\epsilon} t -\tau )) \,
\big]\, \chi_{\tau \le \epsilon^{-1}} \, \Big] \,+\, R^{(1)}_\epsilon\;,
\end{equation*}
where $|R^{(1)}_\epsilon|\to 0$. The expectation on the right-hand
side has to be understood as the expectation of
$\mathbb{E}_{\boldsymbol{x}_\epsilon (\tau)}^{\epsilon} \big[\, F
(\boldsymbol{x}_{\epsilon}(\theta^{(1)}_{\epsilon} t - s )) \, \big]$
for $s=\tau$.  By definition of $r_0$, the wells $\mc E(\bs m')$,
$\bs m'\in\mc M_0$, and \eqref{39}, the right-hand side of the
previous equation is equal to
\begin{equation*}
\sum_{\bs m'\in \mc M_0}  F(\bs m') \;
\mathbb{E}_{\boldsymbol{x}}^{\epsilon}
\Big[\, \mathbb{P}_{\boldsymbol{x}_\epsilon (\tau)}^{\epsilon} \big[\,
\boldsymbol{x}_{\epsilon}(\theta^{(1)}_{\epsilon} t -\tau ) \in \ms
E(\bs m') \, \big]\, \chi_{\tau \le \epsilon^{-1}} \, \Big] \,+\,
R^{(1)}_\epsilon \,+\, R^{(2)}_\epsilon \,+\, R_\eta \;,
\end{equation*}
where $|R_\eta|\le \eta$ and
\begin{equation*}
|R^{(2)}_\epsilon| \;\le\; \Vert F\Vert_\infty\,
\sup_{\bs y\in \mc E(\bs m)} \sup_{t - (\epsilon
\theta^{(1)}_\epsilon)^{-1} \le s\le t}
\mathbb{P}_{\boldsymbol{y}}^{\epsilon} \big[\,
\boldsymbol{x}_{\epsilon}(\theta^{(1)}_{\epsilon} s ) \not\in \ms
E(\mc M_0) \, \big]\;.
\end{equation*}
By \eqref{42}, $R^{(2)}_\epsilon \to 0$. By Proposition \ref{l11},
Lemma \ref{lem_exit} and Corollary \ref{t_hitting}, as
$\epsilon\to 0$, the sum converges to
\begin{equation*}
\sum_{\bs m'\in \mc M_0}  F(\bs m') \;
\mc Q_{\bs m} \big[\,
\mb y (t) =\bs m'\, \big] \;=\;
\mc Q_{\bs m} \big[\, F(\mb y (t))\, \big] \;,
\end{equation*}
which completes the proof of the theorem. \qed

\subsection*{Proof of Proposition \ref{l11}}

The proof relies on a lemma, which appeared before in \cite[Lemma
3.1]{LLM} for discrete-valued Markov processes.

\begin{lem}
\label{l: lem0}
Fix $t>0$ and $\bs m$, $\bs m' \in \mc M_0$.
Then, for all $\boldsymbol{x}\in\mathcal{E}(\bs m)$, $b\in(0,\,t/3)$
and sequence $t_\epsilon \to t$,
\begin{equation*}
\mathbb{Q}_{\boldsymbol{x}}^{\epsilon}
\big[\,\bs y^{\rm T}_\epsilon\left(t-3b\right)\in\mathcal{E}(\bs
m')\,\big] \, \le\,
\mathbb{Q}_{\boldsymbol{x}}^{\epsilon}\big[\,
\boldsymbol{y}_{\epsilon}(t_\epsilon)\in\mathcal{E}(\bs m')\,\big]
\,+ \, R_{\epsilon}(\bs x,\,t,\,b)\;,
\end{equation*}
where,
\begin{equation*}
\lim_{b\rightarrow0}\,\limsup_{\epsilon\to0}\,\sup_{\boldsymbol{x}\in\mathcal{E}(\bs m)}\, R_{\epsilon}(\bs x,\,t,\,b)=0\;.
\end{equation*}
\end{lem}

\begin{proof}
Fix $t>0$, $\bs m$, $\bs m' \in \mc M_0$,
$\boldsymbol{x}\in\mathcal{E}(\bs m)$, a sequence
$t_\epsilon \to t$, and $2<\alpha < 3$. By \eqref{eq:trace} and the
trivial fact that $S_{\epsilon}(t)\ge t$, for $b\in(0,\,t/3)$
\begin{equation*}
\mathbb{Q}_{\boldsymbol{x}}^{\epsilon}
[\,\bs y^{\rm T}_\epsilon(t-3b)\in\mathcal{E}(\bs m')\,]
\,=\, \mathbb{Q}_{\boldsymbol{x}}^{\epsilon}
[\,\bm{y}_{\epsilon} (S_{\epsilon}(t-3b)) \in\mathcal{E}(\bs m')\,]
\;\le\;
\mathbb{Q}_{\boldsymbol{x}}^{\epsilon}[\,A_{\epsilon}(t,\,b)\,]
\,+\, \mathbb{Q}_{\boldsymbol{x}}^{\epsilon}
[\,B_{\epsilon}(t,\,b)\,]\;,
\end{equation*}
where
\begin{gather*}
A_{\epsilon}(t,\,b) \,=\,
\{\,S_{\epsilon}(t-3b)>t-\alpha b\,\}\;,
\\
B_{\epsilon}(t,\,b)
\,=\,\{\,\bs{y}_{\epsilon}(s)\in\mathcal{E}(\bs m')\ \
\text{for some}\ s\in[t-3b,\,t- \alpha b]\,\}\ .
\end{gather*}
By Lemma \ref{l: lem1}, as $\alpha <3$,
\begin{equation*}
\limsup_{\epsilon\rightarrow0}\,
\sup_{\boldsymbol{x}\in\mathcal{E}(\mathcal{M}_0)}\,
\mathbb{Q}_{\boldsymbol{x}}^{\epsilon}[\,A_{\epsilon}(t,\,b)\,]=0\ .
\end{equation*}
On the other hand,
\begin{equation*}
\mathbb{Q}_{\boldsymbol{x}}^{\epsilon}[B_{\epsilon}(t,\,b)]
\le\mathbb{Q}_{\boldsymbol{x}}^{\epsilon}
[\, \bs{y}_{\epsilon}(t_\epsilon)\in\mathcal{E}(\bs m') \,]
+\mathbb{Q}_{\boldsymbol{x}}^{\epsilon}
[\,B_{\epsilon}(t,\,b),\,\bm{y}_{\epsilon}(t_\epsilon)
\notin\mathcal{E}(\bs m')\,]\;.
\end{equation*}
It remains to prove that
\begin{equation*}
\limsup_{b\rightarrow0}\,\limsup_{\epsilon\to 0}\,
\sup_{\bs x\in \mc E(\bs m)} \mathbb{Q}_{\boldsymbol{x}}^{\epsilon}
[\, B_{\epsilon}(t,\,b),\,\bm{y}_{\epsilon}(t_\epsilon)
\notin\mathcal{E}(\bs m')\,] \,=\, 0\;.
\end{equation*}
By Lemma \ref{l17}, the definition of $B_{\epsilon}(t,\,b)$, and
the strong Markov property,
\begin{equation*}
\limsup_{b\rightarrow0}\,\limsup_{\epsilon\to0}\,
\sup_{\boldsymbol{x}\in\mathcal{E}(\bs m)}\,
\mathbb{Q}_{\boldsymbol{x}}^{\epsilon}
[\,B_{\epsilon}(t,\,b),\,\bs{y}_{\epsilon}(t_\epsilon)\in
\mathcal{E}(\mc M_0) \setminus\mathcal{E}(\bs m')\,] \,=\, 0\;.
\end{equation*}
On the other hand, as $\alpha>2$, for $\epsilon$ sufficiently small,
$t_\epsilon -s \in [2b,4b]$ for all $s\in [t-3b, t-\alpha b]$. Hence,
by the strong Markov property and Proposition \ref{p06},
\begin{equation*}
\limsup_{b\rightarrow0}\, \limsup_{\epsilon\to0}\,
\sup_{\boldsymbol{x}\in\mathcal{E}(\bs m)}\,
\mathbb{Q}_{\boldsymbol{x}}^{\epsilon}
[\, B_{\epsilon}(t,\,b),\, \bm{y}_{\epsilon}(t_\epsilon)
\notin\mathcal{E}(\mc M_0)\, ] \;=\; 0\;.
\end{equation*}
The assertion of the lemma follows from the previous estimates.
\end{proof}

\begin{proof}[Proof of Proposition \ref{l11}]
The proof is similar to the one of \cite[Proposition 2.1]{LLM}.
We consider the case $\mf n=1$, the general one being similar.

Fix $t>0$, $\bs m$, $\bs m' \in \mc M_0$, and sequences
$\bs x_\epsilon \in \mc E(\bs m)$, $t_\epsilon \to t$. By Theorem
\ref{t02},
\begin{equation*}
\mc Q_{\bs m} [ \bs y(t) = \bs m']
\;=\; \lim_{\delta \to 0} \mc Q_{\bs m} [ \bs y(t-3\delta) = \bs m']
\;=\; \lim_{\delta \to 0}
\lim_{\epsilon \to 0} \bb Q^\epsilon_{\bs x_\epsilon}
[ \, \bs y^{\rm T}_\epsilon (t-3\delta) \in \mc E(\bs m') \, ] \;.
\end{equation*}
Thus, by Lemma \ref{l: lem0},
\begin{equation*}
\mc Q_{\bs m} [ \bs y(t) = \bs m'] \;\le \;
\liminf_{\epsilon\to 0} \bb Q^\epsilon_{\bs x_\epsilon}
[ \, \bs y_\epsilon  (t_\epsilon) \in \mc E(\bs m') \, ]
\;\le \;
\limsup_{\epsilon\to 0} \bb Q^\epsilon_{\bs x_\epsilon}
[ \, \bs y_\epsilon  (t_\epsilon) \in \mc E(\bs m') \, ]\;.
\end{equation*}
Since
\begin{equation*}
1 \;=\; \sum_{\bs m'\in \mc M_0} \mc Q_{\bs m} [ \bs y(t) = \bs m']
\quad\text{and}\quad
\sum_{\bs m'\in \mc M_0} \bb Q^\epsilon_{\bs x_\epsilon}
[ \, \bs y_\epsilon  (t_\epsilon) \in \mc E(\bs m') \, ] \;\le\; 1\;,
\end{equation*}
the inequalities in the penultimate formula must be identities for
each $\bs m\in \mc M_0$, which completes the proof of the proposition.
\end{proof}

\subsection*{Avoiding  wells}

We complete the proof Proposition \ref{l11} by showing that the
probability that the process is not in a well when it starts from a
well is very small. This is the content of Proposition \ref{p06}
below, the main result of this subsection.

\begin{prop}
\label{p06}
For all $\bs m \in \mc M_0$,
\begin{equation*}
\limsup_{b\rightarrow0}\,\limsup_{\epsilon\rightarrow0}\,
\sup_{\boldsymbol{x}\in\mathcal{E}(\bs m)}\,
\sup_{t\in[2b,\,4b]}\,\mathbb{Q}_{\boldsymbol{x}}^{\epsilon}
[\,\bs{y}_{\epsilon}(t)\notin\mathcal{E} (\mc M_0) \,] \;=\; 0\;.
\end{equation*}
\end{prop}

The proof of this proposition requires some preliminary estimates.
Fix $\color{blue} \eta\in(0,\,r_{0}/2)$ so that there is no critical
point $\bm{c} \in \mc C_0$ such that
$U(\bm{c})\in(U(\bs m'),\,U(\bs m')+\eta)$ for some
$\bs m'\in\mc M_0$.  Fix $\bs m\in \mc M_0$, and let
\begin{equation*}
{\color{blue} \mc R}  \,=\, \mathcal{R}(\bs m) \,=\,
(\mathbb{R}^{d}\setminus\mathcal{E} (\mc M_0) )
\cap \big\{\bm{x}\in\mathbb{R}^{d}:U(\bm{x})
< U(\bs m )+\eta/2\, \big\}\; .
\end{equation*}

Denote by $\mc W$ the connected component of the set
$\{\bs x\in \bb R^d: U(\bs x) < U(\bs m) + d^{(1)}\}$ which contains
$\bs m$. We claim that
\begin{equation}
\label{50}
\mc W \cap \mc R \;=\; \varnothing \;.
\end{equation}
Indeed, if $\bs y\in \mc R$, $U(\bs y) < U(\bs m) + \eta/2$. By
definition of $d^{(1)}$ and $\mc E(\bs m)$, all points
$\bs z \in \mc W$ such that $U(\bs z) \le U(\bs m) + \eta/2$ are
contained in $\mc E(\bs m)$. Hence,
$\mc W \cap \mc R = \mc E(\bs m) \cap \mc R$.  By definition of
$\mc R$, $\mc E(\bs m) \cap \mc R = \varnothing$, which completes the
proof of the claim.

\begin{lem}
\label{l12}
Fix $\bs m\in \mc M_0$.
Then,
\begin{equation*}
\limsup_{a\rightarrow0}\,
\limsup_{\epsilon\rightarrow0}\,
\sup_{\boldsymbol{x}\in\mathcal{E}(\bs m)}\,
\mathbb{Q}_{\boldsymbol{x}}^{\epsilon}
\big[\,\tau_{_{\mathcal{R}}} \le a \,\big]=0\;.
\end{equation*}
\end{lem}

\begin{proof}
By Lemma \ref{l17}, it suffices to show that
\begin{equation*}
\mathbb{Q}_{\boldsymbol{x}}^{\epsilon}\big[\,
\tau_{_{\mathcal{R}}} \le a \,\big]
\;\le\;
2\, \max_{\bs m'\in \mc M} \sup_{\bm{z}\in \mathcal{E}(\bs m')}
\mathbb{Q}_{\bm{z}}^{\epsilon}\big[\,
\tau_{\mc E(\mc M_0) \setminus \mathcal{E}(\bs m')} <
2 a \,\big ]
\,+\, R_{\epsilon}(\bs x) \;,
\end{equation*}
where $\sup_{\bs x \in \mc E(\bs m)} |R_{\epsilon}(\bs x)| \to 0$.

To prove the previous bound, first observe that
\begin{equation}
\label{45}
\mathbb{Q}_{\bm{x}}^{\epsilon} [\, \tau_{_{\mathcal{R}}} \le a \, ]
\,=\, \mathbb{Q}_{\bm{x}}^{\epsilon}
[\, \tau_{_{\mathcal{R}}} \le a \,,\, \sigma_{_{\mc E(\mc M_0)}} \le
\iota_\epsilon\,]
\,+\, \mathbb{Q}_{\bm{x}}^{\epsilon}
[\, \tau_{_{\mathcal{R}}}  \le a \,,\, \sigma_{_{\mc E(\mc M_0)}} >
\iota_\epsilon \, ] \;,
\end{equation}
where $\sigma_{_{\mc A}}$, $\mc A\subset\bb R^d$, is the first time
after $\tau_{\mc R}$ that the process visits $\mc A$:
\begin{equation*}
{\color{blue} \sigma_{_{\mc A}}} \;:=\;
\inf\{t> \tau_{\mc R} : \bs x_\epsilon (t) \in \mc A \,\}\;,
\end{equation*}
$\iota_\epsilon = a + \epsilon^{-1}/\theta^{(1)}_\epsilon$.  By the
strong Markov property, the second term on the right-hand is bounded
by
\begin{equation*}
\sup_{\bm{z}\in\mathcal{R}}
\mathbb{P}_{\bm{z}}^{\epsilon}\big[ \, \tau_{\mc E(\mc M_0)}
\ge \epsilon^{-1} \, \big]\ .
\end{equation*}
To keep notation simple, we replaced the measure
$\bb Q^\epsilon_{\bs z}$ by $\bb P^\epsilon_{\bs z}$.  By Corollary
\ref{t_hitting}, this expression is bounded by a remainder
$R_{\epsilon}(\bs x)$ such that
$\sup_{\bs x \in \mc E(\bs m)} |R_{\epsilon}(\bs x)| \to 0$.

We turn to the first term on the right-hand side of \eqref{45}.
It can be written as
\begin{equation}
\label{49}
\mathbb{Q}_{\bm{x}}^{\epsilon}
[\, \tau_{_{\mathcal{R}}} \le a \,,\, \sigma_{_{\mc E(\bs m)}} \le
\iota_\epsilon\,] \;+\;
\mathbb{Q}_{\bm{x}}^{\epsilon}
[\, \tau_{_{\mathcal{R}}} \le a \,,\, \sigma_{_{\mc E(\mc M_0) \setminus
\mc E(\bs m) }} \le \iota_\epsilon\,]
\end{equation}
By the strong Markov property, the first term is bounded by
\begin{equation*}
\sup_{\bm{z}\in\mathcal{R}}
\mathbb{P}_{\bm{z}}^{\epsilon}
\big[\, \tau_{\mathcal{E}(\bs m)}< 2 a \theta^{(1)}_{\epsilon} \,\big]
\;=\;
\max_{\mc A} \sup_{\bm{z}\in\mathcal{A}}
\mathbb{P}_{\bm{z}}^{\epsilon}
\big[\, \tau_{\mathcal{E}(\bs m)}<  2 a \theta^{(1)}_{\epsilon} \,\big] \;,
\end{equation*}
where the maximum is carried over all connected components of $\mc R$.
The number of connected component is finite because $U(x) \to\infty$
as $|x|\to\infty$.  Fix a connected component $\mc A$ of $\mc R$, and
let $\mc B$ be the connected component of
$\{\bm{x}\in\mathbb{R}^{d}:U(\bm{x})<U(\bs m)+ \eta \}$ containing
$\mathcal{A}$.  Since there are no critical points $\bs c\in \mc C_0$
such that $U(\bs c) \in (U(\bs m), U(\bs m) +\eta)$, by Corollary \ref{l16},
\begin{equation*}
\lim_{\epsilon \to 0} \sup_{\bm{z}\in\mathcal{A}}
\mathbb{P}_{\bm{z}}^{\epsilon}\big[\,
\tau_{\partial \mc B} < \tau_{_{\mathcal{E}(\mc B)}} \,\big] \; =\;
0 \;.
\end{equation*}
On the other hand, by \eqref{50},
$\mathcal{E}(\bs m) \subset \mathbb{R}^{d} \setminus
\mathcal{B}$, so that $\tau_{\partial \mc B} < \tau_{\mc E(\bs m)}$.
Hence,
\begin{equation*}
\sup_{\bm{z}\in\mathcal{A}}
\mathbb{P}_{\bm{z}}^{\epsilon}\big[\,
\tau_{_{\mathcal{E}(\bs m)}} < 2 a \theta^{(1)}_{\epsilon} \,\big] \;\le\;
\sup_{\bm{z}\in\mathcal{A}}
\mathbb{P}_{\bm{z}}^{\epsilon}\big[\,
\tau_{\mathcal{E}(\bs m)} <   2 a \theta^{(1)}_{\epsilon} \,,\,
\tau_{_{\mathcal{E}(\mc B)}} < \tau_{\mathcal{E}(\bs m)} \,\big ]
\;+\; o_\epsilon(1) \;.
\end{equation*}
By the strong Markov property, this expression is bounded by
\begin{equation*}
\sup_{\bm{z}\in \mathcal{E}(\mc B)}
\mathbb{P}_{\bm{z}}^{\epsilon}\big[\,
\tau_{\mathcal{E}(\bs m)} <   2 a \theta^{(1)}_{\epsilon} \,\big ]
\;+\; o_\epsilon(1) \;.
\end{equation*}
Since $\mc B$ and $\mc E(\bs m)$ are disjoint this expression is less
than or equal to
\begin{equation*}
\max_{\bs m'\in \mc M} \sup_{\bm{z}\in \mathcal{E}(\bs m')}
\mathbb{P}_{\bm{z}}^{\epsilon}\big[\,
\tau_{\mc E(\mc M_0) \setminus \mathcal{E}(\bs m')} <
2 a \theta^{(1)}_{\epsilon} \,\big ] \;+\; o_\epsilon(1) \;.
\end{equation*}

We turn to the second term of \eqref{49}.  Since
$\epsilon^{-1} \prec\theta_{\epsilon}^{(1)}$, it is bounded by
\begin{equation*}
\mathbb{P}^{\epsilon}_{\bm{x}} \big[\, \tau_{_{\mathcal{E} (\mc M_0)
\setminus\mathcal{E}(\bs m)}} < 2a\theta_{\epsilon} \,\big] \;,
\end{equation*}
which completes the proof of the lemma.
\end{proof}

\begin{proof}[Proof of Proposition \ref{p06}]
Recall the definition of the set $\mc R$ introduced just before Lemma
\ref{l12}. Denote by $\mc W$ the connected component of the set
$\{\bs x \in\mathbb{R}^{d}: U(\bs{x})<U(\bs m) + d^{(1)}\, \}$ which
contains $\bs m$. By \eqref{50}, $\mc R \cap \mc W =
\varnothing$. Clearly,
\begin{align*}
\mathbb{Q}_{\boldsymbol{x}}^{\epsilon}
\big[\, \boldsymbol{y}_{\epsilon}(t)
\in\mathbb{R}^{d}\setminus\mathcal{E} (\mc M_0) \,\big ]
\; \le \;
\mathbb{Q}_{\boldsymbol{x}}^{\epsilon}
\big[\, \bs{y}_{\epsilon} (t) \in \mathbb{R}^{d}
\setminus \{ \mathcal{E} (\mc M_0) \cup
\mathcal{R} \} \, \big ]
\;+\; \mathbb{Q}_{\boldsymbol{x}}^{\epsilon}
\big[\, \tau_{_{\mathcal{R}}} \le t\,\big]\;.
\end{align*}
Recall that $\eta<r_0/2$, choose a time-scale $\varrho_\epsilon$
satisfying \eqref{58}, and let
$\kappa_\epsilon = \varrho_\epsilon/\theta^{(1)}_\epsilon$.  By
Corollary \ref{l15}, the first term on the right-hand side is bounded
by
\begin{equation*}
\mathbb{Q}_{\mu_{\epsilon}^{\rm R}}^{\epsilon}
\big[\, \bm{y}_{\epsilon} (t-\kappa_{\epsilon})
\in \mathbb{R}^{d}\setminus \{\mathcal{E} (\mc M_0) \cup\mathcal{R} \}
\big] \,+\, o_{\epsilon}(1)\;,
\end{equation*}
where the error is uniform over $t\in [2b, 4b]$,
$\bs x\in \mc E(\bs m)$. As $\mu_\epsilon$ is the stationary state and
$\mu_{\epsilon}^{\rm R}$ the measure $\mu_\epsilon$ conditioned to
$\mathcal{W}^{2r_{0}}(\bm{m})$, the previous expression is equal to
\begin{align*}
\frac{\mu_{\epsilon}(\mathbb{R}^{d}\setminus
\{ \mathcal{E} (\mc M_0) \cup\mathcal{R} \} )}
{\mu_{\epsilon}(\mathcal{W}^{2r_{0}}(\bm{m}))}
\,+\, o_{\epsilon}(1) \; = \; o_{\epsilon}(1)\ ,
\end{align*}
where the error terms are uniform on $\bm{x}\in\mathcal{E}(\bm{m})$
and $t\in[2b,\,4b]$.

It remains to show that
\begin{equation*}
\limsup_{b\rightarrow0}\,\limsup_{\epsilon\rightarrow0}\,
\sup_{\boldsymbol{x}\in\mathcal{E}(\bm{m})}\,
\sup_{t\in[2b,\,4b]}\mathbb{Q}_{\boldsymbol{x}}^{\epsilon}
\left[\tau_{\mathcal{R}}\le t\right]=0\;.
\end{equation*}
This is a direct consequence of Lemma \ref{l12} since
$\mathbb{Q}_{\boldsymbol{x}}^{\epsilon} [\tau_{\mathcal{R}} \le t ]
\le \mathbb{Q}_{\boldsymbol{x}}^{\epsilon} [\tau_{\mathcal{R}}\le 4b
]$ for all $t\le 4b$.
\end{proof}

\subsection*{Proof of Theorem \ref{t00}}

The assertion of Theorem \ref{t00} in the time scale
$\theta^{(1)}_\epsilon$ is a particular case of Theorem \ref{t_fdd}.
We turn to the second claim.

Fix a time-scale $\varrho_\epsilon$ such that
$1\prec \varrho_\epsilon \prec \theta^{(1)}_\epsilon$,
$\bs m\in \mc M_0$, $\bs x\in \mc D(\bs m)$, $\eta>0$, and a bounded
continuous function $F$.  Define the wells $\mc E(\bs m')$,
$\bs m'\in\mc M_0$, as in the proof of Proposition \ref{l11}, to
fulfill \eqref{39}. First, assume that there exists $\epsilon_0$ such
that $\varrho_\epsilon \ge \epsilon^{-2}$ for all $\epsilon<\epsilon_0$.

By \cite[Theorem 2.1.2]{fw98}, there exists $T>0$ such that
\begin{equation*}
\bb P^\epsilon_x [ \bs x_\epsilon (T) \not\in \mc E (\bs m) \,] \,=\,
o_\epsilon(1)\;.
\end{equation*}
Hence, by the Markov property,
\begin{equation*}
\bb E^\epsilon_x [\, F(\bs x_\epsilon (\varrho_\epsilon))\,]
\;=\;
\bb E^\epsilon_x \Big[ \bb E^\epsilon_{ \bs x_\epsilon (T) }
\big[\, F(\bs x_\epsilon (\varrho_\epsilon - T)) \,\big] \,
\mb 1\{ \bs x_\epsilon (T) \in \mc E (\bs m)\,\}  \, \Big]
\,+\, o_\epsilon(1)\;.
\end{equation*}
As $\bs x_\epsilon (T)$ belongs to $\mc E (\bs m)$ and
$\varrho_\epsilon \prec \theta^{(1)}_\epsilon$, by Lemma \ref{l17},
inside the second expectation on the right-hand side we may insert the
indicator of the set
$\mc A_1 = \{\bs x_\epsilon (\varrho_\epsilon - T) \not \in \mc E(\mc
M_0) \setminus \mc E(\bs m) \}$ at a cost $ o_\epsilon(1)$. By Proposition
\ref{p_FW}, we may also insert the indicator of the set
$\mc A_2 = \{ U(\bs x_\epsilon (\varrho_\epsilon - T - \epsilon^{-1})
\le U(\bs m) + d^{(1)} + 2r_0\}$ at the same cost. Hence, the
left-hand side of the previous displayed equation is equal to
\begin{equation*}
\bb E^\epsilon_x \Big[ \bb E^\epsilon_{ \bs x_\epsilon (T) }
\big[\, F(\bs x_\epsilon (\varrho_\epsilon - T)) \, \bs 1\{\mc A_1
\cap \mc A_2\}\, \,\big] \,
\mb 1\{ \bs x_\epsilon (T) \in \mc E (\bs m)\,\}  \, \Big]
\,+\, o_\epsilon(1)\;.
\end{equation*}
By the Markov property the previous expectation is equal to
\begin{equation*}
\bb E^\epsilon_x \Big[ \, \bb E^\epsilon_{\bs x_\epsilon (T)}  \Big[
\bs 1\{\mc A_2\} \bb E^\epsilon_{ \bs x_\epsilon (\varrho_\epsilon - T - (1/\epsilon)) }
\big[\, F(\bs x_\epsilon (1/\epsilon)) \, \bs 1\{\mc A'_1\}\,\big] \, \Big]
\mb 1\{ \bs x_\epsilon (T) \in \mc E (\bs m)\,\}  \, \Big] \;,
\end{equation*}
where
$\mc A'_1 = \{\bs x_\epsilon (\epsilon^{-1}) \not \in \mc E(\mc M_0)
\setminus \mc E(\bs m) \}$. Since
$U(\bs x_\epsilon (\varrho_\epsilon - T - \epsilon^{-1})) \le U(\bs m)
+ d^{(1)} + 2r_0$, by Theorem \ref{t_hitting2} and Proposition
\ref{p_FW}, in the third expectation, we may insert the indicator of
the set
$\mc A_3 = \{\bs x_\epsilon (1/\epsilon) \in \mc E(\mc M_0) \}$ at a
cost $ o_\epsilon(1)$. If by bad luck, there are critical points
$\bs c$ such that $U(\bs c) = U(\bs m) + d^{(1)} + 2r_0$, we add to
this constant a positive value to make sure that this does not
happen. As
$\mc A_4 = \mc A'_1 \cap \mc A_3 = \{\bs x_\epsilon (\epsilon^{-1})
\in \mc E(\bs m) \}$, by \eqref{39}, the previous expression is equal
to
\begin{equation*}
F(\bs m) \,  \bb E^\epsilon_x \Big[ \, \bb E^\epsilon_{\bs x_\epsilon (T)}  \Big[
\bs 1\{\mc A_2\} \bb P^\epsilon_{ \bs x_\epsilon (\varrho_\epsilon - T - (1/\epsilon)) }
\big[\, \, \mc A_4 \,\big] \, \Big]
\mb 1\{ \bs x_\epsilon (T) \in \mc E (\bs m)\,\}  \, \Big] \;+\;
R(\epsilon , \eta) \;,
\end{equation*}
where $|R(\epsilon , \eta)| \le \eta + o_\epsilon(1)$. We may now go
backward in the argument to conclude that the previous expression is
equal to $F(\bs m)  + R(\epsilon , \eta)$, which completes the proof
of the theorem in the case where $\varrho_\epsilon \ge \epsilon^{-2}$
for all $\epsilon$ small.

Assume that this is not the case. We may suppose that
$\varrho_\epsilon \le \epsilon^{-2}$ for all $\epsilon$ small enough.
If there is a subsequence which does not satisfy this condition, it is
treated as in the first part of the proof.

By \cite[Theorem 2.1.2]{fw98}, there exists $T>0$ such that
\begin{equation*}
\bb P^\epsilon_x [ \bs x_\epsilon (T) \not\in \mc W^{r_0/2} (\bs m) \,] \,=\,
o_\epsilon(1)\;.
\end{equation*}
Hence, by the Markov property,
\begin{equation*}
\bb E^\epsilon_x [\, F(\bs x_\epsilon (\varrho_\epsilon))\,]
\;=\;
\bb E^\epsilon_x \Big[ \bb E^\epsilon_{ \bs x_\epsilon (T) }
\big[\, F(\bs x_\epsilon (\varrho_\epsilon - T)) \,\big] \,
\mb 1\{ \bs x_\epsilon (T) \in \mc W^{r_0/2} (\bs m) \,\}  \, \Big]
\,+\, o_\epsilon(1)\;.
\end{equation*}
As $\bs x_\epsilon (T) \in \mc W^{r_0/2} (\bs m)$, by Proposition
\ref{p_FW}, in the second expectation, we may insert the indicator of
the set
$\mc A = \{\bs x_\epsilon (\varrho_\epsilon - T ) \in \mc E(\bs m)
\}$ at a cost $ o_\epsilon(1)$. At this point, we may repeat the
arguments presented at the end of the first part of the proof to
conclude. \qed \smallskip

\appendix

\section{The potential $U$}

We present in this section elementary properties of the potential $U$
and the dynamical system \eqref{31}.  The main result establishes the
existence of a path which perfects the infimum in \eqref{Theta}.

\begin{prop}
\label{l05a}
Fix a local minimum $\bs m \in\mc M_0$. Then, there exist a local
minimum $\boldsymbol{m}'\in\mathcal{M}_{0}$, $\bs m'\neq \bs m$ and a
continuous path $\boldsymbol{z} \colon [0,1] \to \bb R^d$ such that
$\boldsymbol{z}(0)=\boldsymbol{m}$,
$\boldsymbol{z}(1)=\boldsymbol{m}'$, and
\begin{equation*}
\begin{gathered}
\max_{t\in[0,\,1]}U(\boldsymbol{z}(t)) \;=\;
U(\boldsymbol{z}(1/2)) \;=\; U(\bs m) \,+\, \Gamma(\bs m) \,=\,
\Theta(\boldsymbol{m},\,\boldsymbol{m}')\;,
\\
U(\bs z(s)) \,<\, U(\bs z(1/2))\;,\;\;
s\in [0,1]\setminus\{1/2\} \;.
\end{gathered}
\end{equation*}
Moreover, if $\bs z(\cdot)$ is such a path, then $\bs z(1/2)$ is a
saddle point of $U$.
\end{prop}

\subsection*{Proof of Proposition \ref{l05a}}

The proof is based on three lemmata. Fix $\bs m\in\mc M_0$, and let
$\mathcal{W}$ be the connected component of
$\{\bm{x}\in\mathbb{R}^{d}:U(\bm{x})<U(\bm{m})+\Gamma(\bm{m})\}$
containing $\bm{m}$. By definition of $\Gamma(\bs m)$,
\begin{equation}
\label{a01}
\mc M_0 \cap \mc W \;=\; \{\bs m\}\;.
\end{equation}

\begin{lem}
\label{l_105a-1}
Fix $\bs m\in\mc M_0$ There is a connected component $\mathcal{W}'$ of
$\{\bm{x}\in\mathbb{R}^{d}:U(\bm{x})<U(\bm{m})+\Gamma(\bm{m})\}$ such
that $\mathcal{W}\cap\mathcal{W}'=\varnothing$ and
$\overline{\mathcal{W}}\cap\overline{\mathcal{W}'}\ne\varnothing$.
\end{lem}

The proof of this result is given in a subsection below.  Recall that
we denote by $B(\bm{x},\,r)$ the open ball of radius $r$ centered at
$\bs x$. Let
\begin{equation*}
\mathcal{A}(\bm{x},\,r)=\left(B (\bm{x},\,r)
\setminus\{\bm{x}\}\right)\cap\{\bm{y}
\in\mathbb{R}^{d}:U(\bm{y})<U(\bm{x})\}\ .
\end{equation*}

\begin{lem}
\label{l_cap_saddle}
Fix $H\in\bb R$, and let $\mathcal{W}_{1}$ and $\mathcal{W}_{2}$ be
two disjoint connected components of
$\{\bm{x}\in\mathbb{R}^{d}:U(\bm{x})<H\}$. If
$\overline{\mathcal{W}_{1}}\cap\overline{\mathcal{W}_{2}}\ne\varnothing$,
then
$\overline{\mathcal{W}_{1}}\cap\overline{\mathcal{W}_{2}} =
\partial\mathcal{W}_{1}\cap\partial\mathcal{W}_{2}$ and any element
$\bm{\sigma}$ of
$\overline{\mathcal{W}_{1}}\cap\overline{\mathcal{W}_{2}}$ is a saddle
point such that $U(\bs\sigma)=H$.  Moreover, for all $r>0$ small
enough, $\mathcal{A}(\bm{\sigma},\,r)$ has two connected components:
$\mathcal{A}(\bm{\sigma},\,r)\cap\mathcal{W}_{1}$ and
$\mathcal{A}(\bm{\sigma},\,r)\cap\mathcal{W}_{2}$
\end{lem}

The proof of this lemma is presented in a later subsection below.

\begin{lem}
\label{l_path_max_saddle}
Let $\bm{m}',\,\bm{m}''\in\mathcal{M}_{0}$ and let
$\bm{z}:[0,1]\to\mathbb{R}^{d}$ be a continuous path such that
\begin{equation*}
\bm{z}(0)=\bm{m}',\quad
\bm{z}(1)=\bm{m}'',
\quad U(\bm{z}(1/2))=\Theta(\bm{m}',\,\bm{m}'')\ ,
\end{equation*}
\begin{equation}
\label{f01}
U(\bm{z}(t))<U(\bm{z}(1/2)\ \text{for}\ t\in[0,\,1]\setminus\{1/2\}\ .
\end{equation}
Then, $\bm{z}(1/2)$ is a saddle point.
\end{lem}

\begin{proof}
Recall that we denote by $\upsilon_{\bs x} (t)$ the solution of the
ODE \eqref{31} starting from $\bs x$. For $s\ge 0$, let $\psi_s \colon
[0,1]\to\mathbb{R}^{d}$ be the continuous path defined by
\begin{equation*}
\psi_s(t)=\upsilon_{\bm{z}(t)} (s)\ .
\end{equation*}
As $U$ decreases along the solutions of the ODE,
\begin{equation*}
U(\psi_s (t)) \,=\, U(\upsilon_{\bm{z}(t)} (s))
\,\le\, U(\bm{z}(t))\ .
\end{equation*}

We claim that
\begin{equation}
\label{f03}
U(\psi_s(1/2)) \,=\, U(\bm{z}(1/2)) \,=\, U(\psi_0(1/2)) \quad
\text{for all $s>0$}\;.
\end{equation}
Suppose, by contradiction, that there exists $s_{0}>0$
such that $U(\psi_{s_0}(1/2))<U(\bm{z}(1/2))$.  By \eqref{f01}, for
all $t\ne1/2$,
\begin{equation*}
U(\psi_{s_0}(t))\le U(\bm{z}(t))<U(\bm{z}(1/2)) \;.
\end{equation*}
Thus, since by hypothesis, $U(\psi_{s_0}(1/2))<U(\bm{z}(1/2))$,
\begin{equation}
\label{f02}
\max_{t\in[0,1]}U(\psi_{s_0}(t))<U(\bm{z}(1/2))=\Theta(\bm{m}',\,\bm{m}'')\ .
\end{equation}
As $\bm{m}'$, $\bs m''$ are critical points,
$\upsilon_{\bs n}(s)=\bs n$ for $\bs n=\bs m'$, $\bs m''$, $s>0$, so
that
$\psi_{s_0} (0) = \upsilon_{\bs z(0)}(s_0) = \upsilon_{\bs m'}(s_0) =
\bs m'$, $\psi_{s_0} (1) = \bs m''$. Therefore, the continuous path
$\psi_{s_0}\colon [0,1]\to \bb R^d$ satisfies $\psi_{s_0}(0)=\bm{m}'$,
$\psi_{s_0}(1)=\bm{m}''$ and \eqref{f02}.  This contradicts to the
definition of $\Theta(\bm{m}',\,\bm{m}'')$, and completes the proof of
claim \eqref{f03}.

It follows from \eqref{f03} and from the fact $U$ strictly decreases
along trajectories which do not start from critical points that
$\bm{z}(1/2)$ is a critical point of $U$.

It remains to show that $\bm{z}(1/2)$ is a saddle point.  Clearly,
$\bm{z}(1/2)$ is not a local minimum.  Suppose, by contradiction, that
$\bm{z}(1/2)$ is not a saddle point. Then, by Lemma \ref{l_ind2_conn}
below, the set $\mathcal{A}(\bm{z}(1/2),\,r)$ is connected for
sufficiently small $r>0$. Since $\bm{z}$ is continuous, there is
$\eta_{0}=\eta_{0}(r)>0$ such that $\bm{z}(t)\in B(\bm{z}(1/2),\,r)$
for all $t\in[1/2-\eta_{0},\,1/2+\eta_{0}]$. Therefore,
$\bm{z}(t)\in\mathcal{A}(\bm{z}(1/2),\,r)$ for all
$t\in[1/2-\eta_{0},\,1/2+\eta_{0}]\setminus\{1/2\}$. Since
$\mathcal{A}(\bm{z}(1/2),\,r)$ is connected and open, it is path
connected.  Therefore, there is a continuous path
$\bm{z}_{1}:[1/2-\eta_{0},\,1/2+\eta_{0}]\to\mathcal{A}(\bm{z}(1/2),\,r)$
such that $\bm{z}_{1}(1/2\pm \eta_{0})=\bm{z}(1/2\pm \eta_{0})$.
Define a path $\bm{z}_{2}:[0,1]\to\mathbb{R}^{d}$ as
\begin{equation*}
\bm{z}_{2}(t)=\begin{cases}
\bm{z}(t) & t\in[0,1/2-\eta_{0})\cup(1/2+\eta_{0},1]\ ,\\
\bm{z}_{1}(t) & t\in[1/2-\eta_{0},\,1/2+\eta_{0}]\ .
\end{cases}
\end{equation*}
Thus, $\bm{z}_{2}$ is a continuous trajectory from $\bm{m}'$ to
$\bm{m}''$ such that $U(\bm{z}(t))<U(\bm{z}(1/2))$ for all
$t\in[0,\,1]$. This contradicts the definition of
$\Theta(\bm{m}',\,\bm{m}'')$, and completes the proof of the lemma.
\end{proof}

\begin{proof}[Proof of Proposition \ref{l05a}]
Fix $\bs m\in\mc M_0$.  Let $\mathcal{W}'$ be given by Lemma
\ref{l_105a-1}, and denote by $\bm{\sigma}$ an element of
$\overline{\mathcal{W}}\cap\overline{\mathcal{W}}'$.  By Lemma
\ref{l_cap_saddle}, $\bs\sigma$ is a saddle point,
$\bs \sigma \in \Upsilon (\bs m)$, and, for sufficiently small $r>0$,
$\mathcal{A}(\bm{\sigma},\,r)$ has two connected components
$\mathcal{A}(\bm{\sigma},\,r)\cap\mathcal{W}$ and
$\mathcal{A}(\bm{\sigma},\,r)\cap\mathcal{W}'$. By Hartman--Grobman
theorem, there are two continuous path
$\phi_{1},\,\phi_{2} \colon (-\infty, 0] \to\mathbb{R}^{d}$ such that
\begin{equation*}
\lim_{t\to-\infty}\phi_{j}(t)=\bm{\sigma}\ ,\quad
\phi_{1}(s)\in\mathcal{A}(\bm{\sigma},\,r)\cap\mathcal{W}\ ,\quad
\phi_{2}(s) \in \mathcal{A}(\bm{\sigma},\,r)\cap\mathcal{W}'
\end{equation*}
for all $s\le 0$. Since $\mc W$, $\mc W'$ are connected, we may extend
continuously these trajectories to $s>0$ in such a way that
$\phi_{1}(s) \in\mc W$, $\phi_{2}(s) \in\mc W'$ for all $s\ge 0$.
As $\bs\sigma\in \Upsilon (\bs m)$, by \eqref{48},
\begin{equation*}
\lim_{s\to\infty} \phi_{1}(s)= \bs m\;, \quad
\lim_{s\to\infty} \phi_{2}(s)= \bs m'\;,
\end{equation*}
where $\bs m'$ is a local minimum of $U$ one in $\mc W'$.

Concatenating the paths $\phi_1$, $\phi_2$ and reparametrizing it, we
obtain a continuous path $\bm{z}:[0,1]\to\mathbb{R}^{d}$ from $\bs m$
to $\bs m'$ such that $\bs z(1/2) = \bs \sigma$. By Lemma
\ref{l_cap_saddle}, $U(\bs \sigma) = U(\bs m) + \Gamma(\bs
m)$. Therefore, by construction, $\bs z(\cdot)$ fulfills all
conditions required in Proposition \ref{l05a}.

It remains to check the final assertion of the proposition, which
follows from Lemma \ref{l_path_max_saddle}.
\end{proof}

\subsection*{Proof of Lemma \ref{l_cap_saddle}}

Throughout this subsection, we will use the fact that an open
connected subset of $\bb R^d$ is path-connected.

\begin{lem}
\label{l_homeo}
Homeomorphisms preserve the number of open connected components.
\end{lem}

\begin{proof}
Let $\mathcal{U}_1$ and $\mathcal{U}_2$ be open sets, and let
$\varphi:\mathcal{U}_1\to\mathcal{U}_2$ be a homeomorphism. Denote by
$\mc U_{j, 1}, \dots, \mc U_{j,n_j}$ the connected components of
$\mathcal{U}_j$, $j=1$, $2$. Since $\varphi$ is continuous,
$\varphi(\mathcal{U}_{1,k})$ is connected. As $\varphi$ is surjective,
$\mc U_2 = \cup_{1\le k\le n_1} \varphi(\mc U_{1,k})$, so that
$n_2 \le n_{1}$. Since $\varphi^{-1}$ is continuous, the same argument
yields the reverse inequality.
\end{proof}

\begin{lem}
\label{l_noncri_divide}Let $\bm{p}$ be non-critical point of $U$.
Then, for sufficiently small $r>0$, the manifold
$\{\bm{x}\in\mathbb{R}^{d}:U(\bm{x})=U(\bm{p})\}$ divides
$B(\bm{p},\,r)$ into two connected components which are
$B(\bm{p},\,r)\cap\{\bm{x}\in\mathbb{R}^{d}:U(\bm{x})<U(\bm{p})\}$ and
$B(\bm{p},\,r)\cap\{\bm{x}\in\mathbb{R}^{d}:U(\bm{x})>U(\bm{p})\}$.
In particular, $\mathcal{A}(\bm{p},\,r)$ is connected. Furthermore,
there is a continuous path $\bm{z}:[0,1]\to B(\bm{p},\,r)$ such that
\begin{equation*}
\bm{z}(0)=\bm{p}\ ,\ \bm{z}\left((0,1]\right)
\subset B(\bm{p},\,r)\cap\{\bm{x}\in\mathbb{R}^{d}:
U(\bm{x})<U(\bm{p})\}\ .
\end{equation*}
\end{lem}

\begin{proof}
Fix $\bm{p}=(p_{1},\,\dots,\,p_{d})\in\mathbb{R}^{d}$ be a
non-critical point. Then, $\nabla U(\bm{p})\ne0$ so that there is
$1\le j\le d$ such that
\begin{equation*}
\frac{\partial U}{\partial x_{j}}(\bm{p})\ne0\ .
\end{equation*}
Assume, without loss of generality, that $j=d$. For
$\bm{x}\in\mathbb{R}^{d}$, let
\begin{equation*}
\widetilde{\bm{x}}=(x_{1,}\,\dots,\,x_{d-1})\ .
\end{equation*}
By the implicit function theorem, there exist $r>0$ and a
$C^{1}$-function $g:\mathbb{R}^{d-1}\to\mathbb{R}$ such that
\begin{equation*}
g(\widetilde{\bm{p}})=p_{d} \;, \quad
U(\widetilde{\bm{x}},\,g(\widetilde{\bm{x}}))=U(\bm{p})\ \text{for
all}\ \widetilde{\bm{x}}\in B_{d-1}(\widetilde{\bm{p}},\,r)\ ,
\end{equation*}
where $B_{d-1}(\widetilde{\bm{p}},\,r)$ is a $(d-1)$-dimensional
ball with radius $r>0$ centered at $\widetilde{\bm{p}}$.

Decompose the set $B(\bm{p},\,r)$ into three parts:
\begin{equation*}
\mathcal{P}_{1}=B(\bm{p},\,r)\cap\{(\widetilde{\bm{x}},
\,\bm{y})\in\mathbb{R}^{d}:\bm{y}>g(\widetilde{\bm{x}})\}\ ,
\end{equation*}
\begin{equation*}
\mathcal{P}_{2}=B(\bm{p},\,r)\cap\{(\widetilde{\bm{x}},
\,\bm{y})\in\mathbb{R}^{d}:\bm{y}<g(\widetilde{\bm{x}})\}\ ,
\end{equation*}
\begin{equation*}
\mathcal{P}_{3}=B(\bm{p},\,r)\cap\{(\widetilde{\bm{x}},
\,\bm{y})\in\mathbb{R}^{d}:\bm{y}=g(\widetilde{\bm{x}})\}\ .
\end{equation*}
By definition of $g$,
$\mathcal{P}_{3}=B(\bm{p},\,r)\cap\{\bm{x}\in\mathbb{R}^{d}:
U(\bm{x})=U(\bm{p})\}$
and
\begin{equation}
U(\bm{x})\ne U(\bm{p})\ \text{for all}\,
\bm{x}\in\mathcal{P}_{1}\cup\mathcal{P}_{2}\ .
\label{e_nocri_divide-1}
\end{equation}

Suppose that there is $\bm{x},\,\bm{y}\in\mathcal{P}_{1}$ such that
$U(\bm{x})<U(\bm{p})<U(\bm{y})$. As
$\mathcal{P}_{1}$ is path-connected, there is a path in
$\mathcal{P}_{1}$ connecting $\bm{x}$ to $\bm{y}$. Since $U$
is continuous, this path must pass through a point
$\bm{z}\in\mathcal{P}_{1}$ such that $U(\bm{z})=U(\bm{p})$, and this
contradicts \eqref{e_nocri_divide-1}. Therefore,
$U(\bm{x})>U(\bm{p})$ for all $\bm{x}\in\mathcal{P}_{1}$ or
$U(\bm{x})<U(\bm{p})$ for all $\bm{x}\in\mathcal{P}_{1}$.

Let $\bm{v}=\nabla U(\bm{p})$. For sufficiently small $\eta>0$,
$U(\bm{p}+\eta\bm{v})>U(\bm{p})$ and $U(\bm{p}-\eta\bm{v})< U(\bs p)$.
Thus, there is $\bm{x},\,\bm{y}\in B(\bm{p},\,r)$ such that
$U(\bm{x})<U(\bm{p})<U(\bm{y})$. Therefore, one of the sets
$\mathcal{P}_{1}$, $\mathcal{P}_{2}$ is
$B(\bm{p},\,r)\cap\{\bm{x}\in\mathbb{R}^{d}:U(\bm{x})<U(\bm{p})\}$ and
the other one is
$B(\bm{p},\,r)\cap\{\bm{x}\in\mathbb{R}^{d}:U(\bm{x})>U(\bm{p})\}$.

Finally, since $\mathcal{P}_{3}$ is the graph of a $C^{1}$ function,
there are paths $\bm{z}_{i}:[0,1]\to B(\bm{p},\,r)$ such
that
\begin{equation*}
\bm{z}_{i}(0)=\bm{p}\ ,\ \bm{z}_{i}\left((0,1]\right)
\subset\mathcal{P}_{i}\ .
\end{equation*}
This completes the proof of the lemma.
\end{proof}

\noindent {\it Critical points}. The next two lemmata provide the
number of connected componentes of the set $\mc A(\bs c, r)$,
$\bs c \in \mc C_0$, in terms of the index of the critical points $\bs
c$.

\begin{lem}
\label{l_saddle_2comp}
Let $\bm{\sigma}$ be a saddle point of $U$.
Then, for sufficiently small $r>0$, the set
$\mathcal{A}(\bm{\sigma},\,r)=\left(B(\bm{\sigma},\,r)
\setminus\{\bm{\sigma}\}\right)\cap\{\bs x\in\bb R^d :
U(\bs x) <U(\bm{\sigma})\}$
has exactly two connected components.
\end{lem}

\begin{proof}
By\cite[Lemma 2.2]{M69}, since $U$ is nondegenerate at
$\bm{\sigma}$, the image of $U$ near $\bm{\sigma}$ is locally
diffeomorphic to the quadratic function $F:\bb R^d\to \bb R$ given by
\begin{equation*}
F(\bs x) \;=\; -\, x_{1}^{2}+\sum_{i=2}^{d}x_{i}^{2}\ .
\end{equation*}
Therefore, for sufficiently small $r>0$, $\mathcal{A}(\bm{\sigma},\,r)$
is diffeomorphic to the set
\begin{equation*}
[\, B(\bs 0 ,\,r) \setminus
\{\bm{0}\} \,] \cap F^{-1}((-\infty, 0)) \,=\,
[\, B(\bs 0 ,\,r) \setminus
\{\bm{0}\} \,] \cap
\{\bm{x}\in\mathbb{R}^{d}:
-x_{1}^{2}+\sum_{i=2}^{d}x_{i}^{2}<0\}\ .
\end{equation*}
Since the set on the right-hand side has two connected components, by
Lemma \ref{l_homeo}, $\mathcal{A}(\bm{\sigma},\,r)$ has also two
connected components.
\end{proof}

\begin{lem}
\label{l_ind2_conn}
Let $\bm{c}$ be a critical point of $U$ with index greater or equal to
$2$. Then, for sufficiently small $r>0$, $\mathcal{A}(\bm{c},\,r)$ is
path-connected.
\end{lem}

\begin{proof}
By \cite[Lemma 2.2]{M69}, since $U$ is nondegenerate at $\bm{c}$,
the image of $U$ near $\bm{c}$ is locally diffeomorphic to to the
quadratic function $F:\bb R^d\to \bb R$ given by
\begin{equation*}
F(\bs x) \;=\; -\, \sum_{i=1}^{k} x_{i}^{2}
\,+\, \sum_{i=k+1}^{d}x_{i}^{2}\ ,
\end{equation*}
where $k\ge 2$ is the index of $\bm{c}$. Therefore, for sufficiently
small $r>0$, $\mathcal{A}(\bm{c},\,r)$ is diffeomorphic to
\begin{equation*}
[\, B(\bs 0 ,\,r) \setminus
\{\bm{0}\} \,]
\cap \Big\{\bm{x}\in\mathbb{R}^{d}:
-\, \sum_{i=1}^{k} x_{i}^{2}
\,+\, \sum_{i=k+1}^{d}x_{i}^{2}<0 \Big\}\;.
\end{equation*}
Since this set is connected, by Lemma \ref{l_homeo},
$\mathcal{A}(\bm{c},\,r)$ is also connected, and therefore
path-connected.
\end{proof}

\smallskip \noindent {\it Level sets}. In this subsection, we examine
the connected components of the level sets of $U$.

\begin{lem}
\label{l_level_connected}
Fix $H\in\mathbb{R}$. Let $\mathcal{H}$ be a connected component of
$\{\bm{x}\in\mathbb{R}^{d}:U(\bm{x})<H\}$.  Let
$\mathcal{G}\subset\{\bm{x}\in\mathbb{R}^{d}:U(\bm{x})<H\}$ be a
connected set satisfying $\mathcal{G}\cap\mathcal{H}\ne\varnothing$.
Then, $\mathcal{G}\subset\mathcal{H}$.  The same assertion holds if we
replace all strict inequalities by inequalities.
\end{lem}

\begin{proof}
Let $\bm{x}_{0}\in\mathcal{H}$. Then, $\mathcal{H}$ is the largest
connected set $\mc F$ satisfying
\begin{equation*}
\bm{x}_{0}\in\mathcal{F}\subset\{\bm{x}\in\mathbb{R}^{d}:U(\bm{x})<H\}\ .
\end{equation*}
As $\mc G \cap \mc H \neq\varnothing$, there exists
$\bs x_0 \in \mc G \cap \mc H$. As $\mathcal{G}$ belongs to the
previous class, $\mathcal{G}\subset\mathcal{H}$.

The same proof yields the second assertion of the lemma.
\end{proof}

\begin{lem}
\label{l_level_boundary}
Fix $H\in\mathbb{R}$. Let $\mathcal{H}$ be a connected component of
the set $\{\bm{x}\in\mathbb{R}^{d}:U(\bm{x})<H\}$ or one of the set
$\{\bm{x}\in\mathbb{R}^{d}:U(\bm{x})\le H\}$. Then, $U(\bm{x}_{0})=H$
for all $\bm{x}_{0}\in\partial\mathcal{H}$. Moreover,
\begin{enumerate}
\item If $\mathcal{H}$ is an open set, then $\bm{x}_{0}$ is not a
local minimum
\item If $\mathcal{H}$ is a closed set, then $\bm{x}_{0}$ is not a
local maximum
\end{enumerate}
\end{lem}

\begin{proof}
Fix $\bm{x}_{0}\in\partial\mathcal{H}$. Since $U$ is continuous,
$U(\bm{x}_{0})\le H$.  Assume by contradiction that $U(\bm{x}_{0})<H$.
Let $\mathcal{G}$ be the connected component of the set
$\{\bm{x}\in\mathbb{R}^{d}:U(\bm{x})<H\}$ containing
$\bm{x}_{0}$. Since $U$ is smooth, there exists $r>0$ such that
\begin{equation*}
\max_{\bm{y}\in B(\bm{x}_{0},\,r)}U(\bm{y})<H
\end{equation*}
As $x_0\in \partial \mc H$, there exists
$z\in B(\bm{x}_{0},\,r) \cap \mc H$. Hence, by the previous displayed
equation, $B(\bm{x}_{0},\,r) \subset \mc H$, so that $x_0\in\mc H$, in
contradiction to the fact that
$\bm{x}_{0}\in\partial\mathcal{H}$. This completes the proof of the
first assertion.

Suppose that $\mathcal{H}$ is a connected component of the set
$\{ \bs x \in\bb R^d : U(\bs x) <H\}$, and fix
$\bm{x}_{0}\in\partial\mathcal{H}$.  By the first assertion of the
lemma, $U(\bm{x}_{0})=H$. Suppose by contradiction that $\bm{x}_{0}$
is a local minimum. Then, there exists $r>0$ such that
$U(\bs y) \ge U(\bs x_0)$ for all $\bs y\in
B(\bm{x}_{0},\,r)$. Therefore,
$B(\bm{x}_{0},\,r)\cap\{\bm{x}\in\mathbb{R}^{d}:U(\bm{x})<H\}=\varnothing$,
so that $B(\bm{x}_{0},\,r)\cap\mathcal{H}=\varnothing$. This
contradicts the fact that $\bm{x}_{0}\in\partial\mathcal{H}$.

Suppose that $\mathcal{H}$ is a connected component of the set
$\{ \bs x \in\bb R^d : U(\bs x) \le H\}$, and fix
$\bm{x}_{0}\in\partial\mathcal{H}$. By the first assertion of the
lemma, $U(\bm{x}_{0})=H$. Suppose by contradiction that $\bm{x}_{0}$
is a local maximum. Then, there exists $r>0$ such that
$U(\bs y) < U(\bs x_0)$ for all
$\bs y\in B(\bm{x}_{0},\,r) \setminus \{\bs x_0\}$. Therefore,
$B(\bm{x}_{0},\,r)\setminus\{\bm{x}_{0}\}\subset \{ \bs x \in\bb R^d :
U(\bs x) < H\}$.

Since $\bm{x}_{0}\in\partial\mathcal{H}$,
$B(\bm{x}_{0},\,r)\cap\mathcal{H}^{o}\ne\varnothing$, where
$\mathcal{H}^{o}$ is the interior of $\mathcal{H}$.  Fix
$\bm{x}_{1}\in B(\bm{x}_{0},\,r)\cap\mathcal{H}^{o}$ and let
$\mathcal{G}$ be the connected component of
$\{\bm{x}\in\mathbb{R}^{d}:U(\bm{x})<H\}$ containing $\bm{x}_{1}$. As
$\bs x_1 \in \mc H^o \subset \mc H$, by definition of $\mc H$,
$\mc G \subset \mc H$.  On the other hand, by Lemma
\ref{l_level_connected},
$B(\bm{x}_{0},\,r)\setminus\{\bm{x}_{0}\}\subset\mathcal{G}$, so that
$B(\bm{x}_{0},\,r) \setminus\{\bm{x}_{0}\}\subset\mathcal{H}$.  As
$\bm{x}_{0}\in\mathcal{H}$, $B(\bm{x}_{0},\,r) \subset\mathcal{H}$.
This contradicts the fact that $\bs x_0 \in\partial\mathcal{H}$, and
completes the proof of the lemma.
\end{proof}

For the next lemma, we extend the definition of
$\Theta(\bs m, \bs m')$ to subsets of $\mc M_0$.  For two disjoint
non-empty subsets $\mathcal{M}'$ and $\mathcal{M}''$ of
$\mathcal{M}_{0}$, define
\begin{equation}
\label{ap02}
\Theta(\mathcal{M}',\,\mathcal{M}'') \;=\;
\min_{\boldsymbol{m}'\in\mathcal{M}',\,\boldsymbol{m}''\in\mathcal{M}''}
\Theta(\boldsymbol{m}',\,\boldsymbol{m}'')\;.
\end{equation}

\begin{lem}
\label{lap01}
Let $\mathcal{H}\subset\mathbb{R}^{d}$ be a connected component of
level set
$\{\boldsymbol{x}\in\mathbb{R}^{d}:U(\boldsymbol{x})<c_{0}\}$ for some
$c_{0}\in\mathbb{R}$. Let $\mathcal{M},\,\mathcal{M}'$ be disjoint
non-empty subsets of $\mathcal{M}_{0}$.
\begin{enumerate}
\item If $\mathcal{M},\,\mathcal{M}'\subset\mathcal{H}$, then
$\Theta(\mathcal{M},\,\mathcal{M}')<c_{0}$.
\item If $\mathcal{M}\subset\mathcal{H}$ and $\mathcal{M}'
\subset\mathbb{R}^{d}\setminus\mathcal{H}$,
then $\Theta(\mathcal{M},\,\mathcal{M}')\ge c_{0}$.
\end{enumerate}
\end{lem}

\begin{proof}
Let $\mathcal{M},\,\mathcal{M}'$ be disjoint non-empty subsets of
$\mathcal{M}_{0}$ contained in $\mc H$.  Since $\mathcal{H}$ is open
connected set, it is a path connected set. Thus, there exists a
connected path $\boldsymbol{z}:[0,\,1]\rightarrow\mathcal{H}$ such
that $\boldsymbol{z}(0)\in\mathcal{M}$ and
$\boldsymbol{z}(1)\in\mathcal{M}'$.  Since
$\boldsymbol{z}(t)\in\mathcal{H}$ for all $t\in[0,\,1]$, we have
$\max_{t\in[0,\,1]}U(\boldsymbol{z}(t))<c_{0}$ and thus by
\eqref{ap02}, $\Theta(\mathcal{M},\,\mathcal{M}')<c_{0}$. This proves
the first assertion.

To prove the second assertion note that any path connecting
$\mathcal{M}$ and $\mathcal{M}'$ must pass through
$\partial\mathcal{H}$ on which the value of $U$ is $c_{0}$.
\end{proof}

\begin{proof}[Proof of Lemma \ref{l_cap_saddle}]
Let
$\bm{\sigma}\in\overline{\mathcal{W}_1}\cap\overline{\mathcal{W}_2}$.
We claim that $\bs \sigma \in \partial \mc W_1 \cap \partial \mc
W_2$. Indeed, by definition $\bs \sigma \in \overline{ \mc W_1}$.  It
remains to show that $\bs \sigma \not\in\mc W_1$. Assume, by
contradiction, that $\bs \sigma \in\mc W_1$. Then, there exists $r>0$
such that $B(\bs \sigma, r) \subset \mc W_1$. Since
$\mc W_1 \cap \mc W_2 = \varnothing$,
$B(\bs \sigma, r) \cap \mc W_2 = \varnothing$, which contradicts the
fact that $\bm{\sigma}\in\overline{\mathcal{W}_2}$. Thus,
$\bs \sigma \in \partial \mc W_1$. The same argument shows that
$\bs \sigma \in \partial \mc W_2$, proving the claim. By Lemma
\ref{l_level_boundary}, $U(\bs\sigma) = H$, and $\bs \sigma$ is not a
local minimum.

By definition, there exists $r>0$ such that
$B(\bm{\sigma},\,r)\cap\mathcal{W}_1\ne\varnothing$ and
$B(\bm{\sigma},\,r)\cap\mathcal{W}_2\ne\varnothing$.  Since
$\bs \sigma \in \partial \mc W_1 \cap \partial \mc W_2$,
$\bs\sigma \not\in \mc W_1 \cup \mc W_2$, so that
$\left(B(\bm{\sigma},\,r)\setminus\{\bm{\sigma}\}\right)
\cap\mathcal{W}_1\ne\varnothing$ and
$\left(B(\bm{\sigma},\,r)\setminus\{\bm{\sigma}\}\right)
\cap\mathcal{W}_2\ne\varnothing$.  Hence, by definition of
$\mathcal{W}_1$ and $\mathcal{W}_2$, $\mathcal{A}(\bm{\sigma},\,r)$ is
not empty.

We claim that $\mathcal{A}(\bm{\sigma},\,r)$ is not connected.
Suppose, by contradiction, that $\mathcal{A}(\bm{\sigma},\,r)$ is
connected. Let $\bm{x}_1\in B(\bm{\sigma},\,r)\cap\mathcal{W}_1$,
$\bm{x}_2\in B(\bm{\sigma},\,r)\cap\mathcal{W}_2$. Since
$U(\bm{x}_{j})< U(\bs \sigma)$,
$\bm{x}_{1},\,\bm{x}_{2}\in\mathcal{A}(\bm{\sigma},\,r)$.  Since
$\mathcal{A}(\bm{\sigma},\,r)$ is open, there exists a continuous path
$\bs z \colon [0,1] \to \bb R^d$ connecting $\bm{x}_{1}$ to
$\bm{x}_{2}$ in $\mathcal{A}(\bm{\sigma},\,r)$. In particular,
$\sup_{0\le t\le 1} U(\bs z(t)) < U(\bs\sigma) = H$. Since
$\bs x_1\in \mc W_1$ and $\mc W_1$ is a connected component of the set
$\{\bs x: U(\bs x) <H\}$, all points in this path, including
$\bs x_2$, belong to $\mc W_1$.  As $\bs x_2 \in\mc W_2$ and
$\mc W_1 \cap \mc W_2 = \varnothing$, this is a contradiction, which
proves the claim.

Since $\bs \sigma$ is not a local minimum, and
$\mathcal{A}(\bm{\sigma},\,r)$ is not empty and not connected, by
Lemmata \ref{l_noncri_divide}, \ref{l_ind2_conn}, $\bm{\sigma}$ is a
saddle point.  By Lemma \ref{l_saddle_2comp},
$\mathcal{A}(\bm{\sigma},\,r)$ has exactly two components. Let
$\mathcal{A}_{1}$, $\mc A_2$, be the connected component which
intersects with $\mc W_1$, $\mathcal{W}_2$, respectively. Since
$\mathcal{A}_{j}$ is a connected set contained in
$\{\bm{x}\in\mathbb{R}^{d}:U(\bm{x})< U(\bs\sigma)\}$, by Lemma
\ref{l_level_connected}, $\mathcal{A}_{1}\subset\mathcal{W}_1$ and
$\mathcal{A}_{2}\subset\mathcal{W}_2$.  Hence,
$\mathcal{A}_{1}\ne\mathcal{A}_{2}$ and
$\mathcal{A}_{1}=\mathcal{A}(\bm{\sigma},\,r)\cap\mathcal{W}_1$,
$\mathcal{A}_{2}=\mathcal{A}(\bm{\sigma},\,r)\cap\mathcal{W}_2$.
\end{proof}

\subsection*{Proof of Lemma \ref{l_105a-1}}

The proof relies on several lemmata.

\begin{lem}
\label{l_comp_conn}
Let $\mathcal{K}_{n}$ be a decreasing sequence of compact connected
sets and let $\mathcal{K}:=\bigcap_{n=1}^{\infty}\mathcal{K}_{n}$.
Then, $\mathcal{K}$ is connected.
\end{lem}

\begin{proof}
Suppose, by contradiction, that $\mathcal{K}$ is not connected. In
consequence, there are two disjoint open sets $\mathcal{U}$ and
$\mathcal{V}$ such that $\mathcal{K}\cap\mathcal{U}\ne\varnothing$,
$\mathcal{K}\cap\mathcal{V}\ne\varnothing$, and
$\mathcal{K}\subset\mathcal{U}\cup\mathcal{V}$.  Since
$\mathcal{K}_{n}\cap\mathcal{V} \ne\varnothing$ and
$\mc U \cap \mc V = \varnothing$,
$\mathcal{K}_{n}\setminus\mathcal{U}\ne\varnothing$.

We claim that $\mathcal{K}_{n}\cap\partial\mathcal{U}\ne\varnothing$.
Suppose by contradiction that
$\mathcal{K}_{n}\cap\partial\mathcal{U}=\varnothing$.  In this case,
$\bb R^d = [\mathcal{K}_{n}\cap\partial\mathcal{U}]^c =
\mathcal{K}_{n}^c\cup (\partial\mathcal{U})^c$, so that
$\mc K_n = \mc K_n \cap (\partial\mathcal{U})^c$. Hence,
$\mc K_n \setminus \mc U = \mc K_n \cap \mc U^c = \mc K_n \cap
(\partial\mathcal{U})^c \cap \mc U^c = \mc K_n \cap
[(\partial\mathcal{U}) \cup \mc U]^c = \mc K_n \cap \overline{ \mc
U}^c \subset \overline{ \mc U}^c$. Therefore, as $\overline{ \mc U}^c$
is an open set, for all
$\bm{x}\in\mathcal{K}_{n}\setminus\mathcal{U}$, there exists
$r(\bm{x})>0$ such that
$B(\bm{x},\,r(\bm{x}))\subset \overline{ \mc U}^c$.  Since
$\mathcal{K}_{n}$ is compact, $\mathcal{K}_{n}\setminus\mathcal{U}$ is
compact so that there are finitely many
$\bm{x}_{1},\,\dots,\,\bm{x}_{k}\in\mathcal{K}_{n}\setminus\mathcal{U}$
such that
\begin{equation*}
\mathcal{K}_{n}\setminus\mathcal{U}
\subset\bigcup_{j=1}^{k}B(\bm{x}_{j},\,r(\bm{x}_{j}))\ .
\end{equation*}
Therefore,
$\mathcal{K}_{n}\subset\mathcal{U}\cup\bigcup_{j=1}^{k}
B(\bm{x}_{j},\,r(\bm{x}_{j}))$.  However, since
$B(\bm{x}_{j},\,r(\bm{x}_{j}))\subset \overline{ \mc U}^c$
for all $j$,
$\mathcal{U}\cap\bigcup_{j=1}^{k}B(\bm{x}_{j},\,r(\bm{x}_{j}))=\varnothing$,
in contradiction with the connectedness of $\mathcal{K}_{n}$. This
proves the claim.

As $(\mathcal{K}_{n}\cap\partial\mathcal{U})$ is a decreasing sequence
of compact sets,
$\mathcal{K}\cap\partial\mathcal{U}=
\bigcap_{n=1}^{\infty}(\mathcal{K}_{n}\cap\partial\mathcal{U})\ne\varnothing$
by Cantor's intersection theorem.  Let
$\bm{x}_{0}\in\mathcal{K}\cap\partial\mathcal{U}$. Since $\mathcal{U}$
is open, $\bm{x}_{0}\notin\mathcal{U}$ so that
$\bm{x}_{0}\in\mathcal{V}$.  Since $\mathcal{V}$ is open, there exists
$r_{0}>0$ such that $B(\bm{x}_{0},\,r_{0})\subset\mathcal{V}$ so that
$B(\bm{x}_{0},\,r_{0})\cap\mathcal{U}=\varnothing$, which contradicts
the fact that $\bm{x}_{0}\in\partial\mathcal{U}$. This completes the
proof of the lemma.
\end{proof}

\begin{lem}
\label{l_decrea_comp}
Let $\mathcal{K}_{n}$ be a decreasing sequence of compact sets.
Suppose that $\mathcal{K}:=\bigcap_{n=1}^{\infty}\mathcal{K}_{n}$ is
contained in an open set $\mathcal{U}$. Then, there exists
$N\in\mathbb{N}$ such that $\mathcal{K}_{N}\subset\mathcal{U}$.
\end{lem}

\begin{proof}
Suppose, by contradiction, that each $\mc K_n$ is not contained in
$\mc U$.  Then, for each $n\in\mathbb{N}$, there exists
$\bm{x}_{n}\in\mathcal{K}_{n}\setminus\mathcal{U} \subset
\mathcal{K}_{1}\setminus\mathcal{U}$. As
$\mc {K}_{1}\setminus\mathcal{U}$ is compact, there is a
subsequence $(\bm{x}_{n}')_{n\ge1}$ which converges to a point
$\bm{x}_{0}\in\mathcal{K}_{1} \setminus\mathcal{U}$. Since
$\bm{x}_{j}'\in\mathcal{K}_{m}\setminus\mathcal{U}$ for all $j\ge m$,
$\bm{x}_{0}\in\mathcal{K}_{m}\setminus\mathcal{U}$.  Therefore,
\begin{equation*}
\bm{x}_{0}\in\bigcap_{n=1}^{\infty}(\mathcal{K}_{n}
\setminus\mathcal{U})=\mathcal{K}\setminus\mathcal{U}=\varnothing\ ,
\end{equation*}
which is a contradiction.
\end{proof}

\begin{lem}
\label{l_level_cover}
Let $\mathcal{H}$ be a connected component of the set
$\{\bm{x}\in\mathbb{R}^{d}:U(\bm{x})\le H\}$, and let $\mc U$ be an
open set containing $\mathcal{H}$. Then, there is $N\in\mathbb{N}$
such that the connected component of
\begin{equation*}
\{\bm{x}\in\mathbb{R}^{d}:U(\bm{x})\le H+\frac{1}{N}\}
\end{equation*}
containing $\mathcal{H}$ is contained in $\mc U$
\end{lem}

\begin{proof}
Let $\mathcal{R}_{n}$ be the connected component of
\begin{equation*}
\{\bm{x}\in\mathbb{R}^{d}:U(\bm{x}) \le H+\frac{1}{n}\}
\end{equation*}
containing $\mathcal{H}$ and let
$\mathcal{R}:=\bigcap_{n=1}^{\infty} \mc R_n$. Since
$\mathcal{H}\subset\mathcal{R}_{n}$ for all $n\ge1$,
$\mathcal{H}\subset\mathcal{R}$.  On the other hand, as
$(\mathcal{R}_{n})_{n\ge1}$ is a decreasing sequence of compact
connected sets, by Lemma \ref{l_comp_conn}, $\mathcal{R}$ is
connected.  Since $\mathcal{R}\cap\mathcal{H}\ne\varnothing$,
$\mathcal{R}\subset\{\bm{x}\in\mathbb{R}^{d}:U(\bm{x})\le H\}$ and
$\mathcal{H}$ is a connected component of
$\{\bm{x}\in\mathbb{R}^{d}:U(\bm{x})\le H\}$, by Lemma
\ref{l_level_connected}, $\mathcal{R}\subset\mathcal{H}$. By the
previous two estimates, $\mathcal{H}=\mathcal{R}$.

As $\mathcal{R}\subset\mc U$ by Lemma \ref{l_decrea_comp}, there
exists $N\in\mathbb{N}$ such that $\mathcal{R}_{N}\subset\mc U$.
\end{proof}

\begin{lem}
\label{l_pathcon_closure}
Fix $H\in\mathbb{R}$. Let $\mathcal{H}$ be a connected component of
$\{\bm{x}\in\mathbb{R}^{d}:U(\bm{x})<H\}$.  Then,
$\overline{\mathcal{H}}$ is path connected.
\end{lem}

\begin{proof}
As $\mc H$ is open and connected, it is path connected. It remains to
show that the boundary $\partial \mc H$ is path connected to $\mc
H$. Fix $\bs x_0\in \partial \mc H$.  By Lemma \ref{l_level_boundary},
$U(\bm{x}_{0})=H$.

Assume that $\bm{x}_{0}$ is not a critical point of $U$. By Lemma
\ref{l_noncri_divide}, there exists $r>0$ such that the manifold
$\{\bm{x}\in\mathbb{R}^{d}:U(\bm{x})=U(\bm{x}_{0})\}$ divides
$B(\bm{x}_{0},\,r)$ into two parts:
\begin{align*}
 & B(\bm{x}_{0},\,r)\cap\{\bm{x}\in\mathbb{R}^{d}:U(\bm{x})<U(\bm{x}_{0})\}\ ,\\
 & B(\bm{x}_{0},\,r)\cap\{\bm{x}\in\mathbb{R}^{d}:U(\bm{x})>U(\bm{x}_{0})\}\ .
\end{align*}
Since $\bm{x}_{0}\in\partial\mathcal{H}$,
$B(\bm{x}_{0},\,r)\cap\mathcal{H}\ne\varnothing$ so that
$B(\bm{x}_{0},\,r)\cap\{\bm{x}\in\mathbb{R}^{d}:U(\bm{x})
<U(\bm{x}_{0})\}\ne\varnothing$.  By Lemma \ref{l_level_connected},
$B(\bm{x}_{0},\,r)\cap\{\bm{x}\in\mathbb{R}^{d}:
U(\bm{x})<U(\bm{x}_{0})\}\subset\mathcal{H}$.  By Lemma
\ref{l_noncri_divide}, there is a path
$\bm{z}\colon [0,1]\to B(\bm{x}_{0},\,r)$ such that
$\bm{z}(0)=\bm{x}_{0}$ and
$\bm{z}\left((0,1]\right)\subset
B(\bm{x}_{0},\,r)\cap\{\bm{x}\in\mathbb{R}^{d}:
U(\bm{x})<U(\bm{x}_{0})\}\subset\mathcal{H}$.  Hence, $\bm{x}_{0}$ is
path-connected to $\mathcal{H}$.

Suppose that $\bm{x}_{0}$ is a critical point. By Lemma
\ref{l_level_boundary}, $\bm{x}_{0}$ is not a local minimum. By the
Hartman--Grobman Theorem, there is $T>0$ and a continuous path
$\bm{z}\colon [0,T]\to\mathbb{R}^{d}$ in the unstable manifold of
$\bm{x}_{0}$ such that $\bm{z}(0)=\bm{x}_{0}$ and
$\bm{z}\left((0,T]\right)\subset\mathcal{H}$. This completes the proof
of the lemma.
\end{proof}

\begin{lem}
\label{l_pathcon_level}
Let $\mathcal{H}$ be a connected component of
$\{\bm{x}\in\mathbb{R}^{d}:U(\bm{x})\le H\}$, and $\mathcal{H}^{o}$
the interior of $\mathcal{H}$. Denote by $\mc W_i$, $i\ge 1$, the
connected components of $\mathcal{H}^{o}$. Then, the number of
connected components is finite. Moreover,
\begin{enumerate}
\item Let $\mathcal{W}_{j}'$ be a connected component of
$\{\bm{x}\in\mathbb{R}^{d}:U(\bm{x})<H\}$ which intersects with
$\mathcal{W}_{j}$. Then,
$\overline{\mathcal{W}_{j}}=\overline{\mathcal{W}_{j}'}$.

\item $\mathcal{H}$ is path connected. In particular, for each $i$,
there is $j$ such that
$\overline{\mathcal{W}_{i}}\cap\overline{\mathcal{W}_{j}}\ne\varnothing$.
\end{enumerate}
\end{lem}

\begin{proof}
Consider the open set $\mc W_1$. Since critical points of $U$ are not
isolated points, it is not possible to have $U(\bm{x})=H$ for all
$\bm{x}\in\mathcal{W}_{1}$.  Hence, $\mathcal{W}_{1}'$ well defined
and $\mathcal{W}_{1}\cap\mathcal{W}'_1 \neq \varnothing$. Let
$\bm{x}_{0}\in\mathcal{W}_{1}\cap\mathcal{W}'_1$.

\smallskip\noindent{\it Claim 1:} Let $\bs x_1\in \mc W_1$ such that
$U(\bs x_1) = H$. Then, $\bs x_1$ is a local maximum. \smallskip

Fix $\bs x_1\in \mc W_1$ such that $U(\bs x_1) = H$. Since
$\mathcal{W}_{1}$ is open, there exists $r_{1}>0$ such that
$B(\bm{x}_{1},\,r_{1})\subset\mathcal{W}_{1}$. Let $r_{1}$ be small
enough so that there is no critical point in
$B(\bm{x}_{1},\,r_{1})\setminus\{\bm{x}_{1}\}$.  Let
$\bm{y}\in B(\bm{x}_{1},\,r_{1})$ such that $U(\bm{y})=H$.  Since
$U(\bm{x})\le H$ for all $\bm{x}\in B(\bm{x}_{1},\,r_{1})$, $\bm{y}$
is a critical point (since $\nabla U(\bm{y})=0$). Hence,
$U(\bm{x})<H$ for all
$\bm{x}\in B(\bm{x}_{1},\,r_{1})\setminus\{\bm{x}_{1}\}$.  Therefore,
$\bm{x}_{1}$ is a local maximum, as claimed.

Let
\begin{equation*}
\widehat{\mathcal{W}}_{1}:=\{\bm{x}\in \mc W_1:
\bm{x}\,\text{is not a local maximum}\} \,\subset\, \mc W_1\ .
\end{equation*}
Since there are finitely many local maximum in $\mathcal{W}_{1}$,
$\widehat{\mathcal{W}}_{1}$ is open and connected. By Claim 1,
$U(\bm{x})<H$ for all $\bm{x}\in\widehat{\mathcal{W}}_{1}$.

By construction, $\bm{x}_{0}\in\widehat{\mathcal{W}}_{1}$.

\smallskip\noindent{\it Claim 2:}
$\widehat{\mathcal{W}}_{1}\subset\mathcal{W}_{1}'\subset\mathcal{W}_{1}$.
\smallskip

Since $\mathcal{W}_{1}'$ is a connected component of
$\{\bm{x}\in\mathbb{R}^{d}:U(\bm{x})<H\}$ intersecting with
$\widehat{\mathcal{W}}_{1}$ and since $U(\bm{x})<H$ for all
$\bm{x}\in\widehat{\mathcal{W}}_{1}$, by Lemma
\ref{l_level_connected},
$\widehat{\mathcal{W}}_{1}\subset\mathcal{W}_{1}'$.  Let
$\bm{x}_{2}\in\mathcal{W}_{1}'$. Since $\bs x_0 \in \mc W'_1$, there
is a continuous path $\bs z\colon [0,1]\to \mc W'_1$ in
$\mathcal{W}'_1$ from $\bm{x}_{0}$ to $\bm{x}_{2}$ such that
$ U(\bs z(t)) <H$ for all $0\le t\le 1$.  Since $\mathcal{H}$ is a
connected component of $\{\bm{x}\in\mathbb{R}^{d}:U(\bm{x})\le H\}$
containing $\bm{x}_{0}$, this path is contained in $\mathcal{H}$:
$\bs z(t) \in\mathcal{H}$ for all $0\le t\le 1$. As $U(\bs z(t))<H$,
by Lemma \ref{l_level_boundary}, $\bs z(t) \in\mathcal{H}^{o}$ for all
$0\le t\le 1$. As $\mc W_1$ is a connected component of $\mc H^o$ and
$\bs x_0\in \mc W_1$, $\bs x_2 \in\mathcal{W}_{1}$, as
claimed. \smallskip

By definition, the set $\mc W'_1$ contains a local minimum. By Claim
2. so does $\mc W_1$. Since the connected components are disjoints,
each one contains at least one local minimum of $U$ and there are a
finite number of critical points, the set $\mc H^o$ has a finite
number of connected components. This is the first assertion of the
lemma.

By Claim 2,
$\overline{\widehat{\mathcal{W}}_{1}} \subset\overline{\mathcal{W}'_1}
\subset\overline{\mathcal{W}_{1}}$.  Since local maxima
$\bm{y}\in\mathcal{W}_{1}$ are accumulation points,
$\overline{\widehat{\mathcal{W}}_{1}}=\overline{\mathcal{W}_{1}}$ so
that $\overline{\mathcal{W}_{1}'}=\overline{\mathcal{W}_{1}}$. This
proves the second assertion of the lemma.

Denote by $n$ the number of connected components of $\mc H^o$, so that
$\mathcal{H}=\overline{\mathcal{H}^{o}}=\overline{\bigcup_{i=1}^{n}
\mathcal{W}_{i}}=\bigcup_{i=1}^{n}\overline{\mathcal{W}_{i}}$.  By
Lemma \ref{l_pathcon_closure},
$\overline{\mathcal{W}_{i}} = \overline{\mathcal{W}_{1}'}$ is
path-connected.

\smallskip\noindent{\it Claim 3:} for all
$i \neq j\in \{1, \dots, n\}$, there exists
$i= i_{0},\,\dots,\,i_{k} = j$ such that
\begin{equation}
\label{ap01}
\overline{\mathcal{W}_{i_{m}}}\cap\overline{\mathcal{W}_{i_{m+1}}}
\ne \varnothing \;, \quad 0\le m < k\;.
\end{equation}
\smallskip

Suppose this property does not hold. Then, there exists
$i\neq j\in \{1, \dots, n\}$ which are not connected in the sense
\eqref{ap01}. Let $A$ be the set of indices in $\{1, \dots, n\}$ which
are connected to $i$ in the sense \eqref{ap01}. The sets
$\cup_{k\in A} \overline{\mathcal{W}_{k}}$,
$\cup_{k\not\in A} \overline{\mathcal{W}_{k}}$ are compact, disjoint
and non-empty.  Thus, there exist disjoint opens sets $\mc U$, $\mc V$
such that $\cup_{k\in A} \overline{\mathcal{W}_{k}} \subset \mc U$,
$\cup_{k\not\in A} \overline{\mathcal{W}_{k}} \subset \mc V$. This
contradicts the fact that
$\mc H = \bigcup_{i=1}^{n}\overline{\mathcal{W}_{i}} $ is connected,
and proves Claim 3.

Since each set $\overline{\mathcal{W}_{i}}$ is path-connected, by
property \eqref{ap01}, the set $\mathcal{H}$ is also path-connected.
\end{proof}

\begin{lem}
\label{l_105a-2}
The connected component of the set
$\{\bm{x}\in\mathbb{R}^{d}:U(\bm{x})\le\Theta(\bm{m}',\,\bm{m}'')\}$
containing $\bm{m}'$ also contains $\bm{m}''$.
\end{lem}

\begin{proof}
Let $\mathcal{H}$ be the connected component of the set
$\{\bm{x}\in\mathbb{R}^{d}:U(\bm{x})\le\Theta(\bm{m}',\,\bm{m}'')\}$
containing $\bm{m}'$. Suppose, by contradiction, that
$\bm{m}''\notin\mathcal{H}$. Since $\mathcal{H}$ and $\{\bm{m}''\}$
are compact sets, there is an open set $\mc U$ such that
$\mathcal{H}\subset\mc U$ and $\bm{m}''\notin\mc U$. By Lemma
\ref{l_level_cover}, there is $n\in\mathbb{N}$ such that the
connected component of
\begin{equation*}
\{\bm{x}\in\mathbb{R}^{d}:U(\bm{x})\le\Theta(\bm{m}',\,\bm{m}'')
+\frac{1}{n}\}
\end{equation*}
containing $\mathcal{H}$ is contained in $\mc U$. This connected
component does not contain $\bm{m}''$.  Thus, by Lemma \ref{lap01},
$\Theta(\bm{m}',\,\bm{m}'')\ge
\Theta(\bm{m}',\,\bm{m}'')+\frac{1}{n}$, which is a
contradiction.
\end{proof}

\begin{lem}
\label{l_118}
Let $\bm{m},\,\bm{m}'\in\mathcal{M}_{0}$ be two different local minima.
Then,
\[
U(\bm{m}), \, U(\bm m')<\Theta(\bm{m},\,\bm{m}')\ .
\]
\end{lem}

\begin{proof}
We only prove for $\bm m$ because $\Theta(\cdot, \, \cdot)$ is
symmetric. Since $\bm{m}$ is a local minimum, there exists $\eta>0$
such that $\bm{m}$ is a unique local minimum of a connected component
of $\{\bm{x}\in\mathbb{R}^{d}:U(\bm{x})<U(\bm{m})+2\eta\}$ containing
$\bm{m}$. Therefore, by Lemma \ref{lap01}, for all continuous path
$\bm{z}$ connecting local minimum $\bm{m}$ and any other local
minimum, we have
\[
\max_{t\in[0,1]}U(\bm{z}(t))>U(\bm{m})+\eta\ ,
\]
which implies
\[
\Theta(\bm{m},\,\bm{m}')\ge U(\bm{m})+\eta>U(\bm{m})\ .
\]
\end{proof}

\begin{proof}[Proof of Lemma \ref{l_105a-1}]
Fix $\bs m\in\mc M_0$.  Let $\mathcal{H}$ be the connected component
of the set
$\{\bm{x}\in\mathbb{R}^{d}: U(\bm{x}) \le U (\bm{m}) + \Gamma (\bs
m)\}$ containing $\bm{m}$. By definition, there exists
$\bm{m}' \in\mathcal{M}_{0}$ such that
\begin{equation*}
\Theta(\bm{m},\,\bm{m}')=\min_{\bm{m}''\in
\mathcal{M}_{0}\setminus\{\bm{m}\}}\Theta(\bm{m},\,\bm{m}'')
=U(\bm{m})+\Gamma(\bm{m})\;.
\end{equation*}
By Lemma \ref{l_105a-2}, $\bs m'\in \mc H$.

As in Lemma \ref{l_pathcon_level}, denote by $\mathcal{H}^{o}$ the
interior of $\mc H$. Let $\mathcal{W}_{j}$, $1\le i\le n$, the open
connected components of $\mc H^{o}$.  Assume that
$\bm{m}\in\mathcal{W}_{1}$.

We assert that $\bm{m}$ is the unique local minimum in
$\mathcal{W}_{1}$. Indeed, as in Lemma \ref{l_pathcon_level}, let
$\mathcal{W}_{1}'$ be a connected component of
$\{\bm{x}\in\mathbb{R}^{d}:U(\bm{x})< U(\bm{m})+\Gamma(\bm{m})\}$
which intersects with $\mathcal{W}_{1}$. By \eqref{a01}, $\mc W'_1$
contains one and only one local minimum of $U$. On the other hand, by
Claim 2 in Lemma \ref{l_pathcon_level}, $\mc W'_1 \subset \mc W_1$ and
all elements of $\mc W_1\setminus \mc W'_1$ are local maxima. This
proves the assertion.

By Lemma
\ref{l_118}, $U(\bm m')<\Theta(\bm m, \, \bm m')$. As
$\bs m'\in\mc H$, by Lemma \ref{l_level_boundary},
$\bm m' \in \mc H^{o}$ so that $n\ge2$.

By Lemma \ref{l_pathcon_level}-(2), there is $1<k\le n$ such that
$\overline{\mathcal{W}_{1}}\cap\overline{\mathcal{W}_{k}}
\ne\varnothing$.  Let $\mc W$, $\mathcal{W}'$ be connected components
of $\{\bm{x}\in\mathbb{R}^{d}:U(\bm{x})<\Theta(\bm{m},\,\bm{m}')\}$
intersecting with $\mc W_1$, $\mathcal{W}_{k}$, respectively. By Lemma
\ref{l_pathcon_level},
$\overline{\mathcal{W}}=\overline{\mathcal{W}_{1}}$ and
$\overline{\mathcal{W}'}=\overline{\mathcal{W}_{k}}$ so that
$\overline{\mathcal{W}}\cap\overline{\mathcal{W}'}\ne\varnothing$.
\end{proof}

\section{Extension of the vector field $\mathbf  b$}
\label{sec-ap4}

Fix $\bs m\in \mc M_0$.  In this subsection, we define a new vector
field $\bs b_0 \colon \bb R^d \to \bb R^d$ which coincides with
$\bs b$ in a neighborhood of $\bs m$ and satisfies the hypotheses of
Section \ref{sec-ap3}. Assume that $\bs m = \bs 0$, and let
\begin{equation*}
\mathbb{H}= (\nabla^{2}U) (\boldsymbol{0})\;,\;\;\;
\mathbb{L}= (D\boldsymbol{\ell})
(\ensuremath{\boldsymbol{0})}\;.
\end{equation*}
By Taylor expansion, for $\boldsymbol{x}\simeq\boldsymbol{0}$,
\begin{align*}
-\, \bs b (\boldsymbol{x})\cdot\mathbb{H}\boldsymbol{x}
\,=\, \left[(\mathbb{H+\mathbb{L}})\boldsymbol{x}
+O(|\boldsymbol{x}|^{2})\right]\cdot\mathbb{H}
\boldsymbol{x}=|\mathbb{H}\boldsymbol{x}|^{2}+O(|\boldsymbol{x}|^{3})\;,
\end{align*}
where the second equality comes from the fact that
$\mathbb{H}\mathbb{L}$ is skew-symmetric. Thus, there exists $r_5>0$
such that
\begin{equation}
=\, -\, \bs b (\boldsymbol{x})\cdot\mathbb{H}\boldsymbol{x}
\ge\frac{1}{2}\, |\mathbb{H}\boldsymbol{x}|^{2}
\;\;\;\;\text{for all\;}\boldsymbol{x}\in B(\bs 0,2r_5) \;.
\label{eq:cond1}
\end{equation}
If needed, modify the definition of $r_5>0$ for
\begin{equation}
| \bb K_{\bs x} \bs y |
\,\le\, \frac{1}{2}\, | \mathbb{H} \bs y | \;\;\;\;
\text{for all\;}\boldsymbol{x}\in B(\bs 0,2r_5) \;, \;\;
\bs y\in\bb R^d\;.
\label{eq:cond2}
\end{equation}
where $\bb K_{\bs x}  = (\nabla^{2}U+D\boldsymbol{\ell})(\boldsymbol{x})
-(\mathbb{H}+\mathbb{L})$.

For $\boldsymbol{x}\not \in B(\bs 0, r_5)$, let
\begin{equation*}
\boldsymbol{r}(\boldsymbol{x})=
\frac{r_5}{|\boldsymbol{x}|}\boldsymbol{x}
\in\partial B(\bs 0, r_5) \;
\end{equation*}
and let $\boldsymbol{b}_0:\mathbb{R}^{d}\rightarrow\mathbb{R}^{d}$ be
given by
\begin{equation}
\boldsymbol{b}_0 (\boldsymbol{x}) \;=\;
\left\{
\begin{aligned}
& \boldsymbol{b}(\boldsymbol{x}) \;, \;\;
\boldsymbol{x}\in B(\bs 0, r_5)
\\
& \boldsymbol{b}(\boldsymbol{r}(\boldsymbol{x}))
+ (D\boldsymbol{b}) (\boldsymbol{r}(\boldsymbol{x}))(\boldsymbol{x}
-\boldsymbol{r}(\boldsymbol{x})) \;, \;\;
\boldsymbol{x}\in B(\bs 0, r_5)^{c}\;.
\end{aligned}
\right.
\label{eq:vecb}
\end{equation}

The main result of this section reads as follows.

\begin{prop}
\label{pap4}
The vector field $\boldsymbol{b}_0$ fullfils all conditions of
Section \ref{sec-ap3}. Condition (2) holds for $r_3=r_5$.
\end{prop}

The proof relies on two lemmata.

\begin{lem}
\label{prop:b1}
The vector field $\boldsymbol{b}_0$ belongs to
$C^{1}(\mathbb{R}^{d},\,\mathbb{R}^{d})$. Moreover, there exists a finite
constant $C_{1}$ such that
\begin{equation}
|\boldsymbol{b}_0(\boldsymbol{x})|\le
C_{1}|\boldsymbol{x}|
\quad\text{and}\quad
\left\Vert D\boldsymbol{b}_0 (\boldsymbol{x})\right\Vert
\le C_{1}|\boldsymbol{x}|
\label{eq:cond_b}
\end{equation}
for all $\boldsymbol{x}\in B(\bs 0, r_5)^{c}$.
\end{lem}

\begin{proof}
By a straightforward computation, for $|\boldsymbol{x}| > r_5$,
\begin{equation}
(\partial_{x_k}\boldsymbol{b}_0) (\boldsymbol{x})
\,=\, \partial_{x_k} \big\{\, \left[D\boldsymbol{b}(\bm{r}(\bm{x}))\right]
(\bm{x}-\bm{r}(\bm{x})) \,\big\}
\,+\, \partial_{x_k} \{ \, \boldsymbol{b} (\bm{r}(\bm{x})) \,\} \;.
\label{eq:part_b}
\end{equation}
Since
\begin{equation*}
\partial_{x_k}\bm{r}(\bm{x})
\,=\, r_5\, \frac{\boldsymbol{e}_{k}}{|x|}
\,-\, \frac{x_{k}}{|\boldsymbol{x}|^{3}}\, \boldsymbol{x}\;,
\end{equation*}
the matrix $\partial_{x_k}\left[D\boldsymbol{b}(\bm{r}(\bm{x}))\right]$
is uniformly bounded on $B(\bs 0, r_5)^{c}$.

Since $\boldsymbol{x}\rightarrow\boldsymbol{r}(\boldsymbol{x})$ as
$\boldsymbol{x}$ approaches $\partial B(\bs 0, r_5)$, the
boundedness of
$\partial_{x_k}\left[D\boldsymbol{b}(\bm{r}(\bm{x}))\right]$ yields
that
$\partial_{x_k}\boldsymbol{b}_0(\boldsymbol{x})\rightarrow
\partial_{x_k}\boldsymbol{b}(\boldsymbol{x})$ as $\boldsymbol{x}$
approaches to $\partial B(\bs 0, r_5)$.  This proves that
$\boldsymbol{b}_0\in C^{1}(\mathbb{R}^{d},\,\mathbb{R}^{d})$.  The
first assertion of (\ref{eq:cond_b}) follows from the definition of
$\bs b_0$. The second one from (\ref{eq:part_b}) and the boundedness
of $\partial_{k}\left[D\boldsymbol{b}(\bm{r}(\bm{x}))\right]$ on
$B(0,r_5)^{c}$.
\end{proof}

\begin{lem}
\label{prop:b2}
For all $\boldsymbol{x}\in\mathbb{R}^{d}$,
\begin{equation*}
-\, \boldsymbol{b}_0 (\boldsymbol{x})\cdot\mathbb{H}\boldsymbol{x}
\ge\frac{1}{2}|\mathbb{H}\boldsymbol{x}|^{2}
\end{equation*}
\end{lem}

\begin{proof}
By (\ref{eq:cond1}) the condition is satisfied for
$\boldsymbol{x}\in B(\bs 0, 2r_5)$. Fix
$\boldsymbol{x}\not \in B(\bs 0, r_5)$ so that
\begin{equation}
-\, \boldsymbol{b}_0 (\boldsymbol{x})\cdot\mathbb{H}\bm{x}
\,=\, -\, \boldsymbol{b}(\bm{r}(\bm{x}))\cdot\mathbb{H}\bm{x}
\,-\, (D\boldsymbol{b}) (\bm{r}(\bm{x}))(\bm{x}
-\bm{r}(\bm{x}))\cdot\mathbb{H}\bm{x}
\label{eq:cs1}
\end{equation}
Since
$\boldsymbol{x}= (|\boldsymbol{x}|/r_5)\,
\boldsymbol{r}(\boldsymbol{x})$ and since
$\boldsymbol{r}(\boldsymbol{x})\in B(\bs 0, 2r_5)$, by
(\ref{eq:cond1}), the first term on the right-hand side can be
estimated by
\begin{equation*}
-\, \frac{|\boldsymbol{x}|}{r_5} \, \bs b (\bm{r}(\bm{x}))\cdot
\mathbb{H}\bm{r}(\bm{x})
\, \ge\, \frac{|\boldsymbol{x}|}{2r_5}\, |\,
\mathbb{H}\bm{r}(\bm{x})\, |^{2}
\,=\, \frac{r_5}{2|\bm{x}|}\, |\mathbb{H}\bm{x}|^{2}\;.
\end{equation*}
For the second term, write
\begin{equation*}
-\, (D\boldsymbol{b}) (\bm{r}(\bm{x})) \,=\, \mathbb{H}
+\mathbb{L}+\mathbb{K}_{\bs x} \;\;\;
\text{and}\;\;\;\;\bm{x}-\bm{r}(\bm{x})
=\big(\, 1-\frac{r_5}{|\boldsymbol{x}|} \, \big)\, \boldsymbol{x}\;.
\end{equation*}
Since  $\mathbb{HL}$ is skew-symmetry, the second term of
(\ref{eq:cs1}) is equal to
\begin{align*}
& \big(\, 1-\frac{r_5}{|\boldsymbol{x}|} \, \big)\,
\left(\mathbb{H}+\mathbb{L}+\mathbb{K}_{\bs x} \right)
\boldsymbol{x}\cdot\mathbb{H}\bm{x}
\\
& \quad =\; \big(\, 1-\frac{r_5}{|\boldsymbol{x}|} \, \big)\,
\left(\left|\mathbb{H}\boldsymbol{x}\right|^{2}
+\mathbb{K}_{\bs x} \boldsymbol{x}\cdot\mathbb{H}\bm{x}\right)
\ge\frac{1}{2}\, \big(\, 1-\frac{r_5}{|\boldsymbol{x}|} \, \big)\,
\left|\mathbb{H}\boldsymbol{x}\right|^{2}\;.
\end{align*}
The last inequality comes from (\ref{eq:cond2}). Adding the previous
estimates completes the proof of the lemma.
\end{proof}

\begin{proof}[Proof of Proposition \ref{pap4}]
To check the first condition, suppose that
$\boldsymbol{b}_0 (\boldsymbol{x})=0$ for some
$\boldsymbol{x}\in\mathbb{R}^{d}$. Lemma \ref{prop:b2} implies that
$\boldsymbol{x}=\boldsymbol{0}.$ Thus $\boldsymbol{0}$ is the only
equilibrium of the dynamical system \eqref{eq:x_0}. Since the behavior
of this ODE near $\boldsymbol{0}$ is identical to that of
$\boldsymbol{x}(\cdot)$, the origin is a stable equilibrium.
Condition (2) in Section \ref{sec-ap3} for $r_3=r_5$ follows from the
definition of $\bs b_0$. The third and fourth conditions have been
derived in Lemmata \ref{prop:b1} and \ref{prop:b2}, respectively.
\end{proof}

\section{Potential Theory}
\label{sec-ap2}

In sake of completeness we introduce in this section the capacity
between sets.  Fix two disjoint non-empty bounded domains
$\mathcal{A}$ and $\mathcal{B}$ of $\mathbb{R}^{d}$ with
$C^{2,\,\alpha}$-boundaries for some $\alpha\in(0,\,1)$. Assume that
the perimeters of $\mc A$, $\mc B$ are finite and that the distance
between the sets is positive. Let
$\Omega=(\overline{\mathcal{A}}\cup\overline{\mathcal{B}})^{c}$ so
that $\partial\Omega=\partial\mathcal{A}\cup\partial\mathcal{B}$.

The equilibrium potentials $h_{\mathcal{A},\mathcal{\,B}}^{\epsilon}$
between $\mathcal{A}$ and $\mathcal{B}$ with respect to the processes
$\boldsymbol{x}_{\epsilon}(\cdot)$ is given by
\begin{equation*}
h_{\mathcal{A},\mathcal{\,B}}^{\epsilon}\,(\boldsymbol{x})
\,=\,\mathbb{P}_{\boldsymbol{x}}^{\epsilon}
\,[\,\tau_{\mathcal{A}}<\tau_{\mathcal{B}\:}]\;, \quad
\boldsymbol{x}\in\mathbb{R}^{d}\;,
\end{equation*}
and the capacity by
\begin{equation*}
\textup{cap}_{\epsilon}(\mathcal{A},\,\mathcal{B})\,=\,
\epsilon\int_{\Omega}|\nabla
h_{\mathcal{A},\,\mathcal{B}}^{\epsilon}|^{2}\,d\mu_{\epsilon} \;.
\end{equation*}
We refer to \cite{LMS} for equivalent formulations and properties of
the capacity.

\section{Analysis of a linear ODE}
\label{sec:ODE}

In this section, we prove Lemma \ref{lem_esclin}. To simplify
notation, we fix $\boldsymbol{c}\in\mathcal{Y}_{0}$, assumed to be
equal to $0$, $\boldsymbol{c}=\boldsymbol{0}$, and write
$\mathbb{A}=-(\mathbb{H}^{\bm{c}}+\mathbb{L}^{\bm{c}})$. By Lemma
\ref{lem:hyper}, the matrix $\mathbb{A}$ is invertible and does not
have a pure imaginary eigenvalue. All the results given in this
section holds for such a matrix $\mathbb{A}$.

\subsection*{Real Jordan canonical form}

Suppose that a matrix $\mathbb{K}$ can be written as a block matrix of
the form
\begin{equation*}
\mathbb{K}=\begin{bmatrix}\mathbb{K}_{1} & \mathbb{O} & \mathbb{O}\\
\mathbb{O} & \ddots & \mathbb{O}\\
\mathbb{O} & \mathbb{O} & \mathbb{K}_{n}
\end{bmatrix}
\;,
\end{equation*}
where $\mathbb{K}_{1},\dots,\,\mathbb{K}_{n}$ are matrices of possibly
different sizes and $\mathbb{O}$ denotes the zero matrix of suitable
size. We represent such a matrix as
$\mathbb{K}=\text{diag}(\mathbb{K}_{1},\,\dots,\,\mathbb{K}_{n})$.

We start by a review of a real Jordan canonical form of $\mathbb{A}$.
By \cite[Theorem 2.5]{ODE}, there exists an invertible matrix
$\mathbb{U}$ such that
\begin{equation}
\label{62}
\mathbb{A}=\mathbb{U}\mathbb{J}\mathbb{U}^{-1}
\end{equation}
where $\mathbb{J}$ is of the form
\begin{equation*}
\mathbb{J}={\rm
diag}(\mathbb{E}_{1}^{-},\,\dots,\,\mathbb{E}_{u_{1}}^{-},\,
\mathbb{F}_{1}^{-},\,\dots,\,\mathbb{F}_{u_{2}}^{-},\,
\mathbb{E}_{1}^{+},\,\dots,\,\mathbb{E}_{s_{1}}^{+},\,
\mathbb{F}_{1}^{+},\,\dots,\,\mathbb{F}_{s_{2}}^{+})\,
\end{equation*}
\begin{equation*}
\mathbb{E}_{k}^{\pm}=\left[\begin{array}{ccc}
\lambda_{\,k}^{\pm} & 1 & 0\\
0 & \ddots & 1\\
0 & 0 & \lambda_{\,k}^{\pm}
\end{array}\right]\ \;, \quad
\mathbb{F}_{k}^{\pm}=\left[\begin{array}{ccc}
\mathbb{B}_{k}^{\pm} & \mathbb{I}_{2} & 0\\
0 & \ddots & \mathbb{I}_{2}\\
0 & 0 & \mathbb{B}_{k}^{\pm}
\end{array}\right]\;,
\end{equation*}
and
\begin{equation*}
\mathbb{I}_{2}=\left[\begin{array}{cc}
1 & 0\\
0 & 1
\end{array}\right]\ \;, \quad \ \mathbb{B}_{k}^{\pm}
=\left[\begin{array}{cc}
\alpha_{k}^{\pm} & \beta_{k}^{\pm}\\
-\beta_{k}^{\pm} & \alpha_{k}^{\pm}
\end{array}\right]\ .
\end{equation*}
In this formula, $\lambda_{k}^{+}$ and $\alpha_{k}^{+}$ are positive
while $\lambda_{k}^{-}$ and $\alpha_{k}^{-}$ are negative real
numbers. The eigenvalues of $\mathbb{A}$ are $\lambda_{k}^{\pm}$ and
$\alpha_{k}^{\pm}+i\beta_{k}^{\pm}$. Thus, by Lemma \ref{lem:hyper},
$\lambda_{k}^{\pm}$ and $\alpha_{k}^{\pm}$ cannot be $0$. (Note that
the real numbers $\beta^\pm_k$ are also different from $0$ because the
eigenvalues $\alpha_{k}^{\pm}+i\beta_{k}^{\pm}$ are complex numbers,
but this will not be used below).

By \eqref{62}
\begin{equation*}
e^{t\mathbb{A}}=\mathbb{U}e^{t\mathbb{J}}\mathbb{U}^{-1}
\end{equation*}
where
\begin{equation}
e^{t\mathbb{J}}={\rm diag}(e^{t\mathbb{E}_{1}^{-}},\,\dots,\,e^{t\mathbb{E}_{u_{1}}^{-}},\,e^{t\mathbb{F}_{1}^{-}},\,\dots,\,e^{t\mathbb{F}_{u_{2}}^{-}},\,e^{t\mathbb{E}_{1}^{+}},\,\dots,\,e^{t\mathbb{E}_{s_{1}}^{+}},\,e^{t\mathbb{F}_{1}^{+}},\,\dots,\,e^{t\mathbb{F}_{s_{2}}^{+}})\ .\label{e: Jordan_exp}
\end{equation}
Suppose that $\mathbb{E}_{k}^{\pm}$ is a $j\times j$ matrix.  An
elementary computation yields that
\begin{equation}
\label{e:jor1}
e^{t\mathbb{E}_{k}^{\pm}}=e^{t\lambda_{k}^{\pm}}\left[\begin{array}{cccccc}
1 & t & \frac{t^{2}}{2} & \cdots & \frac{t^{j-1}}{(j-1)!} & \frac{t^{j}}{j!}\\
0 & 1 & t & \cdots & \frac{t^{j-2}}{(j-2)!} & \frac{t^{j-1}}{(j-1)!}\\
\vdots & \vdots & \vdots & \ddots & \vdots & \vdots\\
0 & 0 & 0 & \cdots & 1 & t\\
0 & 0 & 0 & \cdots & 0 & 1
\end{array}\right]\;.
\end{equation}
Similarly, if $\mathbb{F}_{k}^{\pm}$ is a $2j\times2j$ matrix,
\begin{equation}
e^{t\mathbb{F}_{k}^{\pm}}=e^{t\alpha_{k}^{\pm}}\left[\begin{array}{cccccc}
\mathbb{S} & t\mathbb{S} & \frac{t^{2}}{2}\mathbb{S} & \cdots & \frac{t^{j-1}}{(j-1)!}\mathbb{S} & \frac{t^{j}}{j!}\mathbb{S}\\
\mathbb{O} & \mathbb{S} & t\mathbb{S} & \cdots & \frac{t^{j-2}}{(j-2)!}\mathbb{S} & \frac{t^{j-1}}{(j-1)!}\mathbb{S}\\
\vdots & \vdots & \vdots & \ddots & \vdots & \vdots\\
\mathbb{O} & \mathbb{O} & \mathbb{O} & \cdots & \mathbb{S} & t\mathbb{S}\\
\mathbb{O} & \mathbb{O} & \mathbb{O} & \cdots & \mathbb{O} & \mathbb{S}
\end{array}\right]\;,
\label{jor2}
\end{equation}
where $\mathbb{O}$ denotes the $2\times2$ zero matrix and
\begin{equation*}
\mathbb{S}=\left[\begin{array}{cc}
\cos(t\beta_{k}^{\pm}) & \sin(t\beta_{k}^{\pm})\\
-\sin(t\beta_{k}^{\pm}) & \cos(t\beta_{k}^{\pm})
\end{array}\right]\;.
\end{equation*}

\subsection*{Stable and unstable manifolds}

Recall from \eqref{34} that we represent by $\upsilon_{L,\bs x} (t)$
the solution of the linear ODE \eqref{34}. With the notation of this
section, it can be written as
\begin{equation}
\label{e: sol_ODElin}
\frac{d}{dt} \upsilon_{L,\bs x} (t)
\,=\, \bb A\, \upsilon_{L,\bs x} (t) \;, \quad
\upsilon_{L,\bs x} (t)
\;=\; e^{t\mathbb{A}}\boldsymbol{x}
=\mathbb{U} e^{t\mathbb{J}}\mathbb{U}^{-1}\boldsymbol{x}\ .
\end{equation}

Recall that $\mathcal{M}_{L,s}$, $\mathcal{M}_{L,u}$ represent the
stable, unstable manifold of $\bs c = \bs 0$ for the linear ODE
\eqref{34}.  By \eqref{e: sol_ODElin},
\begin{equation*}
\mathcal{M}_{L,u} :=\big\{ \bm{y}\in\mathbb{R}^{d}:
\lim_{t\to-\infty}e^{t\mathbb{J}}\mathbb{U}^{-1}\bm{y}=\bm{0}\, \big\}
\ , \quad
\mathcal{M}_{L,s} :=\big\{ \, \bm{y}\in\mathbb{R}^{d}:
\lim_{t\to+\infty}e^{t\mathbb{J}}\mathbb{U}^{-1}\bm{y}=\bm{0}\, \big\} \ .
\end{equation*}
Denote by $m\in\mathbb{N}$ the size of the matrix
${\rm
diag}(e^{t\mathbb{E}_{1}^{-}},\,\dots,\,e^{t\mathbb{E}_{u_{1}}^{-}},
\,e^{t\mathbb{F}_{1}^{-}},\,\dots,\,e^{t\mathbb{F}_{u_{2}}^{-}})$, and
by $\{\boldsymbol{u}_{1},\,\dots,\,\boldsymbol{u}_{d}\}$ the column
vectors of $\mathbb{U}$,
($\boldsymbol{u}_{i}=\mathbb{U}\boldsymbol{e}_{i}$, where
$\{\bm{e}_{1},\,\dots,\,\bm{e}_{d}\}$ stands for the canonical basis
of $\mathbb{R}^{d}$). By \eqref{e: Jordan_exp}, \eqref{e:jor1}, and
\eqref{jor2},
\begin{equation}
\mathcal{M}_{L,u} =\left\langle \bm{u}_{1},\,\dots,\,
\bm{u}_{m}\right\rangle \;\;\;\text{and\;\;\;}
\mathcal{M}_{L,s} =\left\langle \bm{u}_{m+1},\,\dots,\,
\bm{u}_{d}\right\rangle \;,\label{eq:AuAs}
\end{equation}
where $\left\langle S\right\rangle $ denotes the vector space spanned
by $S$. The following lemma is a direct consequence of the discussion
above.

\begin{lem}
\label{lem_bdbarx}
There exists $C_{0}>0$ such that for all $\bs y\in \mc M_{L,s}$
and $t\ge0$,
\begin{equation*}
\Vert\upsilon_{L,\bs y} (t)\Vert
\le C_{0}\Vert\boldsymbol{y}\Vert\;.
\end{equation*}
\end{lem}

\begin{proof}
Write
$\gamma=\min\{|\lambda_{1}^{-}|,\,\dots,
\,|\lambda_{u_{1}}^{-}|,\,|\alpha_{1}^{-}|,\,\dots,\,|\alpha_{u_{2}}^{-}|\}>0$.
Then, by \eqref{e: Jordan_exp}, \eqref{e:jor1}, \eqref{jor2},
\eqref{e: sol_ODElin}, and \eqref{eq:AuAs}, it is clear that there
exists a polynomial $P(t)$ depending only on $d$ such that
\begin{equation*}
\Vert\upsilon_{L,\bs y} (t)\Vert\le e^{-\gamma t}P(t)\Vert\boldsymbol{y}\Vert\;.
\end{equation*}
The conclusion of lemma follows immediately.
\end{proof}

By \eqref{eq:AuAs}, $\mathbb{R}^{d}=\mc M_{L,u}\oplus\mc M_{L,s}$.
Hence, for each $\bm{y}\in\mathbb{R}^{d}$, there exists a unique
decomposition
\begin{equation}
\bm{y}=\boldsymbol{v}^{u}(\boldsymbol{y})
+\bm{v}^{s}(\boldsymbol{y})\ ,\label{eq:decus}
\end{equation}
such that $\boldsymbol{v}^{u}(\boldsymbol{y})\in\mc M_{L,u}$ and
$\bm{v}^{s}(\boldsymbol{y})\in\mc M_{L,s}$. Next lemma provides the
basic property of this decomposition.

\begin{lem}
\label{l: v_s<x}
There is $c_{0}<\infty $ such that
\begin{equation*}
\|\boldsymbol{v}^{s}(\bm{y})\|\le c_{0}\|\bm{y}\|\
\text{for all}\ \bm{y}\in\mathbb{R}^{d}\;.
\end{equation*}
\end{lem}

The proof is based on the following elementary result.

\begin{lem}
\label{lem:Cauchy}Let $V$ and $W$ be subspaces of $\mathbb{R}^{d}$
such that $V\cap W=\{\bm{0}\}$. Then, there exists $\zeta=\zeta(V,\,W)>0$
such that
\begin{equation*}
\sup_{\bm{v}\in V\setminus\{\boldsymbol{0}\},\,\bm{w}\in W\setminus\{\boldsymbol{0}\}}\frac{|\langle\bm{v},\,\bm{w}\rangle|}{\|\bm{v}\|\|\bm{w}\|}=1-\zeta
\end{equation*}
where sup is taken over all non-zero vectors.
\end{lem}

\begin{proof}
Let us define $F:(V\setminus\{\boldsymbol{0}\})
\times(W\setminus\{\boldsymbol{0}\})\rightarrow\mathbb{R}$
as
\begin{equation*}
F(\boldsymbol{v},\,\boldsymbol{w})=\frac{|\langle\bm{v},\,\bm{w}\rangle|}{\|\bm{v}\|\|\bm{w}\|}\;.
\end{equation*}
Since $F(c\boldsymbol{v},\,c'\boldsymbol{w})=F(\boldsymbol{v},\,\boldsymbol{w})$
for all $c,\,c'\neq0$, we have
\begin{equation*}
\sup_{\bm{v}\in V\setminus\{\boldsymbol{0}\},\,\bm{w}\in W\setminus\{\boldsymbol{0}\}}F(\boldsymbol{v},\,\boldsymbol{w})=\sup_{\bm{v}\in V,\,\bm{w}\in W:\,\Vert\boldsymbol{v}\Vert=\Vert\boldsymbol{w}\Vert=1}F(\boldsymbol{v},\,\boldsymbol{w})\;.
\end{equation*}
Since the set $S_{0}=\{(\boldsymbol{v},\,\boldsymbol{w}):\bm{v}\in V,\,\bm{w}\in W:\,\Vert\boldsymbol{v}\Vert=\Vert\boldsymbol{w}\Vert=1\}$
is compact and $F(\cdot,\,\cdot)$ is continuous, the function $F$
achieve the maximum at certain $(\boldsymbol{v}^{*},\,\boldsymbol{w}^{*})\in S_{0}$.
Then
\begin{equation*}
\sup_{\bm{v}\in V\setminus\{\boldsymbol{0}\},\,\bm{w}\in W\setminus\{\boldsymbol{0}\}}F(\boldsymbol{v},\,\boldsymbol{w})=F(\boldsymbol{v}^{*},\,\boldsymbol{w}^{*})\;.
\end{equation*}
Note that $F(\boldsymbol{v}^{*},\,\boldsymbol{w}^{*})<1$ by the
Cauchy-Schwarz inequality (the equality cannot hold because of the
assumption $V\cap W=\{\bm{0}\}$). This completes the proof.
\end{proof}

\begin{proof}[Proof of Lemma \ref{l: v_s<x}]
Since $\mc M_{L,u}\cap\mc M_{L,s}=\{\boldsymbol{0}\}$, by
Lemma \ref{lem:Cauchy}, there exists a constant $c>1$ such that,
for all $\boldsymbol{y}\in\mathbb{R}^{d}$,
\begin{equation*}
|\langle\bm{v}^{u}(\boldsymbol{y}),\,\boldsymbol{v}^{s}(\boldsymbol{y})\rangle|
\;\le\; \sqrt{(c-1)/c}\|\bm{v}^{u}(\boldsymbol{y})\|
\|\boldsymbol{v}^{s}(\boldsymbol{y})\|\\
\;\le\; \frac{1}{2}\|\bm{v}^{u}(\boldsymbol{y})\|^{2}
+\frac{c-1}{2c}\|\boldsymbol{v}^{s}(\boldsymbol{y})\|^{2}\;,
\end{equation*}
where we applied Young's inequality in the last step.
Therefore,
\begin{equation*}
-2c\langle\bm{v}^{u}(\boldsymbol{y}),\,
\boldsymbol{v}^{s}(\boldsymbol{y})\rangle
\le c\|\bm{v}^{u}(\boldsymbol{y})\|^{2}+(c-1)\|
\boldsymbol{v}^{s}(\boldsymbol{y})\|^{2}\;.
\end{equation*}
Reorganizing, we obtain
\begin{equation*}
\|\boldsymbol{v}^{s}(\boldsymbol{y})\|^{2}\le
c\,(\Vert\bm{v}^{u}(\boldsymbol{y})\|^{2}
+2\langle\bm{v}^{u}(\boldsymbol{y}),\,
\boldsymbol{v}^{s}(\boldsymbol{y})\rangle
+\|\boldsymbol{v}^{s}(\boldsymbol{y})\|^{2})=
c\, \Vert\boldsymbol{y}\Vert^{2}\ .
\end{equation*}
This completes the proof.
\end{proof}

\begin{proof}[Proof of Lemma \ref{lem_esclin}]
Recall the constant $C_{0}$ and $c_{0}$ from Lemmata \ref{lem_bdbarx}
and \ref{l: v_s<x}, respectively, and define $r>0$ as
\begin{equation}
r=\frac{a}{3C_{0}c_{0}}\;\cdot
\label{eq:defr}
\end{equation}

Suppose that $\boldsymbol{y}\in\mathcal{B}(\boldsymbol{0},\,r)$.  As
in \eqref{eq:decus}, decompose $\bs y\in\mathbb{S}^{d}$ into
\begin{equation*}
\bs y=\boldsymbol{v}^{u}(\bs y)+\bm{v}^{s}(\bs y)\ ,
\end{equation*}
so that
\begin{equation*}
\upsilon_{L,\bs y}(t) = \upsilon_{L,\bs v^u(\bs y)}(t)
+\upsilon_{L,\bs v^s (\bs y)}(t) \;.
\end{equation*}
Note that $\upsilon_{L,\bs v^u(\bs y)}(t) \in\mc M_{L,u}$ and
$\upsilon_{L,\bs v^s (\bs y)}(t) \in\mc M_{L,s}$ for all $t\ge0$ since
$\mc M_{L,u}$ and $\mc M_{L,s}$ are invariant under the dynamical
system \eqref{34}.  Recall the definition \eqref{eq:overlinet} of
$t_L(\cdot)$ and write
\begin{equation*}
\bm{w}^{u}(\boldsymbol{y}) \;:=\;
\upsilon_{L,\bs v^u(\bs y)}(t_L(\bs y))
\in \mc M_{L,u}
\;\;\;\text{and\;\;\;}
\bm{w}^{s}(\boldsymbol{y}) \;:=\;
\upsilon_{L,\bs v^s(\bs y)}(t_L(\bs y))
\in \mc M_{L,s} \;,
\end{equation*}
so that
\begin{equation}
\bs e_L (\bs y)=\bm{w}^{u}(\bm{y})+\bm{w}^{s}(\bm{y})\label{eq:decpb}
\end{equation}
By \eqref{eq:overlinet}, $\Vert \bs e_L(\bs y) \Vert =r_{1}$.
Moreover, by Lemmata \ref{lem_bdbarx}, \ref{l: v_s<x}, and by the
definition \eqref{eq:defr} of $r$,
\begin{equation}
\|\bm{w}^{s}(\bm{y})\|\le
C_{0}\Vert\boldsymbol{v}^{s}(\boldsymbol{y})\Vert\le
C_{0}c_{0}\Vert\boldsymbol{y}\Vert\le C_{0}c_{0}r=\frac{a}{3} \;\cdot
\label{eq:wr1}
\end{equation}
Therefore, by the triangle inequality,
\begin{align*}
\Big\Vert \, \bs e_L(\bs y)-\frac{r_{1}}{\|\bm{w}^{u}(\bm{y})\|}
\bm{w}^{u}(\bm{y}) \,\Big\Vert & \;\le\; \Big\Vert \,
\bm{w}^{u}(\bm{y})-\frac{r_{1}}{\|\bm{w}^{u}(\bm{y})\|}
\bm{w}^{u}(\bm{y})\,\Big \Vert \;+\;
\|\bm{w}^{s}(\bm{y})\|\\
& =\; \Big|\, 1-\frac{r_{1}}{\|\bm{w}^{u}(\bm{y})\|}\, \Big|
\, \left\Vert
\bm{w}^{u}(\bm{y})\right\Vert \;+\; \|\bm{w}^{s}(\bm{y})\|\;.
\end{align*}
Since $\Vert\bs e_L(\bs y) \Vert =r_{1}$, this expression is equal to
\begin{equation*}
\big|\, \|\bm{w}^{u}(\bm{y})\|-r_{1}\,\big|
\;+\; \|\bm{w}^{s}(\bm{y})\| \;=\;
\big|\, \|\bm{w}^{u}(\bm{y})\| \,-\,
\Vert\bs e_L(\bs y)\Vert\,\big|  \;+\; \|\bm{w}^{s}(\bm{y})\|\;.
\end{equation*}
By \eqref{eq:decpb} and \eqref{eq:wr1}, this expression is bounded by
$2\, \|\bm{w}^{s}(\bm{y})\|< (2/3) \, a$. This completes the proof of
the lemma since
\begin{equation*}
\frac{r_{1}}{\|\bm{w}^{u}(\bm{y})\|}
\, \bm{w}^{u}(\bm{y})\in\mc M_{L,u}\cap\partial
\mathcal{B}(\boldsymbol{0},\,r_{1})\ .
\end{equation*}
\end{proof}






\section*{Declarations}

The authors have no competing interests to declare that are relevant
to the content of this article.

Data sharing not applicable to this article as no datasets were
generated or analysed during the current study.

C. L. has been partially supported by FAPERJ CNE E-26/201.117/2021,
and by CNPq Bolsa de Produtividade em Pesquisa PQ 305779/2022-2.
J. L. was supported by the KIAS Individual Grant (MG093101) at Korea Institute for Advanced Study, and by the NRF grant funded by the Korea government (No. 2022R1A6A3A13065174, 2022R1F1A106366811).
I. S. was supported by the National Research
Foundation of Korea (NRF) grant funded by the Korea government (MSIT)
(No. 2022R1F1A106366811, 2022R1A5A600084012, 2023R1A2C100517311) and
Samsung Science and Technology Foundation (Project Number
SSTF-BA1901-03).
In addition, I. S. and J. L. are also supported by Seoul National University Research Grant in 2022.

\end{document}